\documentclass[12pt,reqno]{amsart}
\usepackage{amsmath,amssymb,extarrows}
\usepackage{url}
\usepackage{tikz,enumerate}
\usepackage{diagbox}
\usepackage{appendix}
\usepackage{epic}

\usepackage{float} %support many floats

\usepackage{cite}

\usepackage{hyperref}
\usepackage{array}

\usepackage{booktabs}

\allowdisplaybreaks[4]

\newtheorem{theorem}{Theorem}
\newtheorem{corollary}[theorem]{Corollary}
\newtheorem{conj}[theorem]{Conjecture}
\newtheorem{lemma}[theorem]{Lemma}

\theoremstyle{definition}

\theoremstyle{remark}
\newtheorem{rem}{Remark}
\numberwithin{equation}{section}
\numberwithin{theorem}{section}
\numberwithin{defn}{section}
\newcommand*\diff{\mathop{}\!\mathrm{d}}

\newcommand{\CT}{\mathrm{CT}}
\begin{document}
\title[Rogers--Ramanujan type identities]{Rogers--Ramanujan type identities involving double, triple and quadruple sums}

\author{Zhi Li and Liuquan Wang}
\address{School of Mathematics and Statistics, Wuhan University, Wuhan 430072, Hubei, People's Republic of China}
\email{2021202010017@whu.edu.cn}
\address{School of Mathematics and Statistics, Wuhan University, Wuhan 430072, Hubei, People's Republic of China}
\email{wanglq@whu.edu.cn;mathlqwang@163.com}



\begin{abstract}
We prove a number of new Rogers--Ramanujan type identities involving double, triple and quadruple sums. They were discovered after an extensive search using Maple. The main idea of proofs is to reduce them to some known identities in the literature. This is achieved by direct summation or the constant term method. We also obtain some new single-sum identities as consequences.
\end{abstract}

\maketitle
\tableofcontents

\section{Introduction}\label{sec-intro}

The famous Rogers--Ramanujan  identities state that
\begin{align}
    \sum_{n=0}^\infty \frac{q^{n^2}}{(q;q)_n}=\frac{1}{(q,q^4;q^5)_\infty}, \label{Rama-1}\\
    \sum_{n=0}^\infty \frac{q^{n^2+n}}{(q;q)_n}=\frac{1}{(q^2,q^3;q^5)_\infty}. \label{Rama-2}
\end{align}
Here we use $q$-series notation
\begin{align}
&(a;q)_\infty:=\prod\limits_{n=0}^\infty (1-aq^n), \quad |q|<1,\\
&(a;q)_n:=\frac{(a;q)_\infty}{(aq^n;q)_\infty}, \quad n \in \mathbb{N}, \\
&(a_1,a_2,\dots,a_m;q)_n:=\prod\limits_{k=1}^m (a_k;q)_n, \quad n \in \mathbb{N}\cup \{\infty\}.
\end{align}

The identities \eqref{Rama-1} and \eqref{Rama-2}  were first discovered by Rogers \cite{Rogers1894} and later rediscovered by Ramanujan.
They motivated people to search for identities of similar type. One of the famous works on this topic is Slater's list \cite{Slater}, which contain 130 such identities. For example, Slater proved that
\begin{align}
 &\sum_{n=0}^\infty\frac{ q^{n(n+2)}}{(q^4;q^4)_n}=\frac{1}{(q^2,q^3;q^5)_\infty(-q^2;q^2)_\infty},\quad  \text{(S.\ 16)} \label{Slater16}\\
    &\sum_{n=0}^\infty\frac{ q^{n^2}}{(q^4;q^4)_n}=\frac{1}{(q,q^4;q^5)_\infty(-q^2;q^2)_\infty}.\quad  \text{(S.\ 20)} \label{Slater20}
\end{align}
Here and throughout this paper, we shall use the label (S.\ $n$) to denote the $n$-th identity in Slater's list.

As a multi-sum generalization of the Rogers--Ramanujan identities, the Andrews--Gordon identity states that for integers $k,s$ such that $k\geq 2$ and $1\leq s \leq k$,
\begin{align}
\sum_{n_1,\dots,n_{k-1}\geq 0} \frac{q^{N_1^2+\cdots+N_{k-1}^2+N_s+\cdots +N_{k-1}}}{(q;q)_{n_1}(q;q)_{n_2}\cdots (q;q)_{n_{k-1}}}  =\frac{(q^s,q^{2k+1-s},q^{2k+1};q^{2k+1})_\infty}{(q;q)_\infty} \label{AG}
\end{align}
where if $j\leq k-1$, $N_j=n_j+\cdots+n_{k-1}$ and $N_k=0$.
The combinatorial version of this identity was earlier given by B. Gordon \cite{Gordon1961}, and the above analytical version was first formulated by G.E. Andrews \cite{Andrews1974}.

D.M.\ Bressoud \cite{Bressoud1979} gave the even moduli companion of the Andrews--Gordon identity: for integers $k\geq 2$ and $1\leq s \leq k$,
\begin{align}\label{eq-Bressoud}
\sum_{n_1,\dots,n_{k-1}\geq 0} \frac{q^{N_1^2+\cdots+N_{k-1}^2+N_s+\cdots +N_{k-1}}}{(q;q)_{n_1}(q;q)_{n_2}\cdots (q;q)_{n_{k-2}} (q^2;q^2)_{n_{k-1}}} =\frac{(q^s,q^{2k-s},q^{2k};q^{2k})_\infty}{(q;q)_\infty}
\end{align}
where as before, $N_j=n_j+\cdots+n_{k-1}$ if $j\leq k-1$ and $N_k=0$.
When $k=3$, we get the following identities as instances:
\begin{align}
\sum_{i,j\geq 0} \frac{q^{i^2+2ij+2j^2}}{(q;q)_i(q^2;q^2)_j}=\frac{(q^3,q^3,q^6;q^6)_\infty}{(q;q)_\infty}, \label{Bressoud-1} \\
\sum_{i,j\geq 0} \frac{q^{i^2+2ij+2j^2+j}}{(q;q)_i(q^2;q^2)_j} =\frac{(q^2,q^4,q^6;q^6)_\infty}{(q;q)_\infty}, \label{Bressoud-2} \\
\sum_{i,j\geq 0} \frac{q^{i^2+2ij+2j^2+i+2j}}{(q;q)_i(q^2;q^2)_j} =\frac{(q,q^5,q^6;q^6)_\infty}{(q;q)_\infty}. \label{Bressoud-3}
\end{align}

To classify identities according to their shapes, we shall adopt a notion from \cite{Wang2021}. Let $n_1,\ldots,n_k$ be positive integers whose greatest common divisor equals 1. Let $t(i_1,\ldots,i_k)$ denote some integer-valued functions, and let $Q(i_1,\dots,i_k)$ be a rational polynomial in variables $i_1,\dots,i_k$. If the sum side of a Rogers-Ramanujan type identity is a mixed sum of
$$\sum_{(i_1,\dots,i_k)\in \mathbb{N}^k}\frac{(-1)^{t(i_1,\dots,i_k)}q^{Q(i_1,\dots,i_k)}}{(q^{n_1};q^{n_1})_{i_1}\cdots (q^{n_k};q^{n_k})_{i_k}},$$
then we call it as an identity of index $(n_1,\dots,n_k)$. With this convention, we see that the Andrews--Gordon identity is of index $(1,1,\dots,1)$ and the Bressoud identity is of index $(1,\dots, 1,2)$.

In the last twenty years, many Rogers--Ramanujan type identities involving double, triple and quadruple sums were discovered. For instance, Sills \cite{Sills2003} discovered many identities of indices $(1,1)$, $(1,2)$, $(1,1,1)$, $(1,2,2)$ and $(1,1,1,2)$. A sample of his results are
\begin{align}
&\sum_{i,j,k\geq 0}\frac{q^{i^2/2+3j^2+2k^2+2ij+2ik+4jk+3i/2+3j+2k}}{(q;q)_i(q^2;q^2)_j(q^2;q^2)_k}=\frac{(q,q^7,q^8;q^8)_\infty}{(q;q)_\infty}, \quad \text{(\cite[Eq.\ (5.35)]{Sills2003})} \label{intro-Sills-1} \\
&\sum_{i,j,k\geq 0}\frac{q^{i^2/2+3j^2+2k^2+2ij+2ik+4jk+i/2+j}}{(q;q)_i(q^2;q^2)_j(q^2;q^2)_k}=\frac{(q^3,q^5,q^8;q^8)_\infty}{(q;q)_\infty}, \quad \text{(\cite[Eq.\ (5.37)]{Sills2003})} \label{intro-Sills-2} \\
&\sum_{i,j,k,l\geq 0}\frac{(-1)^lq^{i^2+j^2+k^2+6l^2+2ij+2ik+2jk+4il+4jl+4kl+j+2k+2l}}{(q^2;q^2)_i(q^2;q^2)_j(q^2;q^2)_k(q^4;q^4)_l}
=\frac{1}{(q;q^2)_\infty}, \nonumber \\
&\quad \quad \quad \quad \quad \quad \quad \quad \quad \quad \quad \quad \quad \quad \quad \quad  \quad \quad \quad \quad \quad \quad \quad \text{(\cite[Eq.\ (5.69)]{Sills2003})} \label{03.10-3} \\
&\sum_{i,j,k,l\geq 0}\frac{(-1)^iq^{i^2+j^2+k^2+6l^2+2ij+2ik+2jk+4il+4jl+4kl+j+2k+2l}}{(q^2;q^2)_i(q^2;q^2)_j(q^2;q^2)_k(q^4;q^4)_l} \nonumber \\
&\quad \quad \quad =\frac{(q,q^7,q^8,q^{9},q^{15};q^{16})_\infty}{(q^2,q^3,q^4,q^4,q^5,q^{11},q^{12},q^{12},q^{13},q^{14};q^{16})_\infty}. \quad \text{(\cite[Eq.\ (5.68)]{Sills2003})}  \label{intro-Sills-3}
\end{align}

Kanade and Russell \cite{KR-2015,KR-2019} disovered many conjectural identities of indices $(1,4,6)$ and $(1,2,3)$, and most of them have been confirmed by Bringmann, Jennings-Shaffer and Mahlburg \cite{Bringmann} and Rosengren \cite{Rosengren}.
Mc Laughlin \cite{Laughlin} proved some identities of indices such as $(1,1)$,  $(1,2,3)$, $(1,4,6)$ and $(1,2,3,4)$.
Cao and Wang \cite{Cao-Wang} proved some identities of indices $(1,1),(1,2)$, $(1,1,1)$, $(1,1,2)$, $(1,1,3)$, $(1,2,2)$, $(1,2,3)$ and $(1,2,4)$. For example, they proved that for any $u,v,a$ (see \cite[Theorems 3.1 and 3.8]{Cao-Wang}),
\begin{align}
    \sum_{i,j\geq 0} \frac{(-1)^{i+j}u^iv^jq^{((i-j)^2-i-j)/2}}{(q;q)_i(q;q)_j}&=\frac{(u,v;q)_\infty}{(uv/q;q)_\infty}, \label{eq-Cao-Wang-Thm31} \\
    \sum_{i,j\geq 0} \frac{(-1)^iq^{i^2+2ij+2j^2-i-j}a^{i+j}}{(q;q)_i(q^2;q^2)_j}&=(a;q^2)_\infty, \label{Cao-Wang-3.8-1} \\
    \sum_{i,j\geq 0} \frac{(-1)^iq^{i^2+2ij+2j^2-i-j}a^{i+2j}}{(q;q)_i(q^2;q^2)_j}&=(a;q)_\infty. \label{Cao-Wang-3.8-2}
\end{align}

Some identities of index $(1,2)$ were derived by Andrews \cite{Andrews2019} and Takigiku and Tsuchioka \cite{Takigiku2019}. Andrews and Uncu \cite{Andrews-Uncu} proved an identity of index $(1,3)$ and also proposed a conjectural identity of similar shape. Two proofs of their conjecture were subsequently given by Chern \cite{Chern} and Wang \cite{Wang2021}. Kur\c{s}ung\"{o}z \cite{Kursungoz} derived four identities of index $(1,4)$. Moreover, based on the work of Kanade and Russell \cite{KR-2015}, five conjectural identities of index $(1,3)$ were presented in \cite[Conjecture 6.1]{Kursungoz-AnnComb}. For instance, one of them states that
\begin{align}
\sum_{i,j\geq 0}\frac{q^{i^2+3j^2+3ij}}{(q;q)_i(q^3;q^3)_j}=\frac{1}{(q,q^3,q^6,q^8;q^9)_\infty}. \label{K-conj-1}
\end{align}

There is some particular interest on identities of index $(1,1,\dots,1)$. Given a positive integer $r$, Nahm's problem is to find all positive definite matrix $A\in M_n(\mathbb{Q})$, rational vector $B\in \mathbb{Q}^r$ and rational number $C\in \mathbb{Q}$ such that
$$f_{A,B,C}(q):=\sum_{n=(n_1,\dots,n_r)^\mathrm{T}\in (\mathbb{Z}_{\geq 0})^r} \frac{q^{\frac{1}{2}n^\mathrm{T} An+n^\mathrm{T} B+C}}{(q;q)_{n_1}\cdots (q;q)_{n_r}}$$
becomes a modular form. Such $(A,B,C)$ is called a modular triple in the literature. Since $C$ is uniquely determined by $A$ and $B$ when $(A,B,C)$ is a modular triple, we shall omit it and call $(A,B)$ as a modular pair.

Zagier \cite{Zagier} provided many possible modular triples when the rank $r=2,3$.  For example, Zagier \cite[p.\ 46]{Zagier} proved the following identity:
\begin{align}\label{eq-Zagier}
\sum_{i,j\geq 0}\frac{q^{\frac{\alpha}{2}i^2+(1-\alpha)ij+\frac{\alpha}{2}j^2+\alpha \nu i -\alpha \nu j}}{(q;q)_i(q;q)_j}=\frac{(-q^{\frac{\alpha}{2}+\alpha \nu},-q^{\frac{\alpha}{2}-\alpha \nu},q^{\alpha};q^{\alpha})_\infty}{(q;q)_\infty}, \quad \forall \alpha>0, \nu \in \mathbb{Q}.
\end{align}
This implies that
\begin{align}\label{intro-triple}
A=\begin{pmatrix}
\alpha & 1-\alpha \\ 1-\alpha &\alpha
\end{pmatrix} \quad \text{and} \quad B=\begin{pmatrix} \alpha \nu \\ -\alpha \nu \end{pmatrix}
\end{align}
form a modular pair. It was pointed out in \cite{Wang-rank2} that the above identity is a special instance of a general identity found by Cao and Wang (see \cite[Theorem 3.4]{Cao-Wang}). Moreover, from another general identity \cite[Theorem 3.6]{Cao-Wang}, we know that $B=\left(\begin{smallmatrix} \alpha \nu \\ -\alpha \nu-1 \end{smallmatrix}\right)$ is the vector part of another modular pair with the same $A$.

All of Zagier's examples have now been confirmed in the literature.  For instance,  Zagier \cite{Zagier}, Vlasenko and Zwegers \cite{VZ},   Calinescu, Milas and Penn \cite{CMP},  Wang \cite{Wang-rank2} and Cao, Rosengren and Wang \cite{Cao-Rosengren-Wang} proved some identities of index $(1,1)$, confirming all of Zagier's rank two examples. Wang \cite{Wang-rank3} proved a number of identities of index  $(1,1,1)$. This  togehether with Zagier's results \cite{Zagier} proves all of Zagier's rank three examples.  By the way, Milas and Wang \cite{Milas-Wang} proved some identities of index $(1,1,1)$ and in particular confirmed partially a conjecture in \cite{CMP}.

Motivated by the above works, the main object of this paper is to establish more Rogers--Ramanujan type identities involving double, triple and quadruple sums. We use Maple to search for identities of the form
\begin{align}\label{id-form}
\sum_{(i_1,\dots,i_k)\in \mathbb{N}^k} \frac{(-1)^{t(i_1,\dots,i_k)}q^{Q(i_1,\dots,i_k)}}{(q^{n_1};q^{n_1})_{i_1}\cdots (q^{n_k};q^{n_k})_{i_k}}=q^C\prod\limits_{n=1}^\infty (1-q^n)^{a_n}
\end{align}
where $C$ is a rational scalar, $Q(i_1,\dots,i_k)$ (resp.\ $t(i_1,\dots,i_k)$) is a rational quadratic (resp.\ linear) polynomial in the variables $i_1,\dots,i_k$, and $\{a_n\}$ is some bounded sequence of integers. The $q$-series package developed by Garvan \cite{Garvan} is employed in predicating the product form of the sum sides. Of course, it is unlikely to exhaust all such identities by Maple. Our aim is to find as more new identities as possible.

After an extensive search, we find about 100 new identities which satisfies the additional requirement that the product side is modular (after multiplying some suitable factor of $q$). For example, for any $\alpha>0$, we prove that (see Theorem \ref{thm-11-alpha})
\begin{align}\label{intro-id-alpha}
\sum_{i,j\geq 0} \frac{q^{\frac{\alpha}{2}i^2+(1-\alpha)ij+\frac{\alpha}{2}j^2+(1-\frac{\alpha}{2})i+\frac{\alpha}{2}j}}{(q;q)_i(q;q)_j}=\frac{(q^{2\alpha};q^{2\alpha})_\infty^2}{(q;q)_\infty (q^\alpha;q^\alpha)_\infty}.
\end{align}
This looks closely to Zagier's identity \eqref{eq-Zagier} but does not follow from it or the two identities in \cite[Theorems 3.4 and 3.6]{Cao-Wang}. The identity \eqref{intro-id-alpha} also means that the same matrix $A$ in \eqref{intro-triple} and the vector $B=\left( \begin{smallmatrix} 1-\alpha/2 \\ \alpha/2 \end{smallmatrix}\right)$ form a new modular pair.

Just as the above example, some of our results give new companions to some known identities.  For example, we find a new companion to \eqref{Cao-Wang-3.8-1} and \eqref{Cao-Wang-3.8-2} (see Theorem \ref{thm-4-new}):
\begin{align}
\sum_{i,j\geq 0}\frac{(-1)^iq^{i^2+2ij+2j^2-i-2j}a^{i+j}}{(q;q)_i(q^2;q^2)_j}&=(aq;q^2)_\infty.\label{intro-(1,2)-new-1}
\end{align}

We now give some examples of our results involving triple and quadruple sums. We proved many new companions to Sills' identities \eqref{intro-Sills-1} and \eqref{intro-Sills-2} such as
\begin{align}
&\sum_{i,j,k\geq 0}\frac{q^{i^2+6j^2+4k^2+4ij+4ik+8jk-i-2j}}{(q^2;q^2)_i(q^4;q^4)_j(q^4;q^4)_k}= \frac{2(q^6,q^{10},q^{16};q^{16})_\infty}{(q^2;q^2)_\infty}, \label{(1,2,2)-new-3} \\
&\sum_{i,j,k\geq 0}\frac{(-1)^jq^{i^2+6j^2+4k^2+4ij+4ik+8jk+2k}a^ib^{j+k}}{(q^2;q^2)_i(q^4;q^4)_j(q^4;q^4)_k}=(-aq;q^2)_\infty, \label{(1,2,2)-new-5} \\
&\sum_{i,j,k\geq 0}\frac{(-1)^jq^{i^2+6j^2+4k^2+4ij+4ik+8jk+i-4j-4k}}{(q^2;q^2)_i(q^4;q^4)_j(q^4;q^4)_k}= \frac{2}{(q^4,q^6,q^8;q^{12})_\infty}.
\end{align}
Meanwhile, we also find many companions to \eqref{intro-Sills-3} including
\begin{align}
&\sum_{i,j,k,l\geq 0}\frac{(-1)^iq^{i^2+j^2+k^2+6l^2+2ij+2ik+2jk+4il+4jl+4kl+2i+k+2l}}{(q^2;q^2)_i(q^2;q^2)_j(q^2;q^2)_k(q^4;q^4)_l}\nonumber \\
&\quad \quad \quad =\frac{(q^3,q^5,q^8,q^{11},q^{13};q^{16})_\infty}{(q,q^4,q^4,q^6,q^{7},q^{9},q^{10},q^{12},q^{12},q^{15};q^{16})_\infty}, \label{03.10-10} \\
&\sum_{i,j,k,l\geq 0}\frac{(-1)^iq^{i^2+j^2+k^2+6l^2+2ij+2ik+2jk+4il+4jl+4kl-j-2l}}{(q^2;q^2)_i(q^2;q^2)_j(q^2;q^2)_k(q^4;q^4)_l}=\frac{2(q^6,q^{10},q^{16};q^{16})_\infty}{(q^2;q^2)_\infty}, \label{03.10-13} \\
&\sum_{i,j,k,l\geq 0} \frac{(-1)^jq^{i^2+j^2+k^2+6l^2+2ij+2ik+2jk+4il+4jl+4kl+k}a^{i+j+2l}}{(q^2;q^2)_i(q^2;q^2)_j(q^2;q^2)_k(q^4;q^4)_l}=(-q^2;q^2)_\infty (-a^2q^4;q^8)_\infty. \label{1112-parameter-1}
\end{align}

Without the modular restriction for the product side, there are many more identities of the form \eqref{id-form}. For instance, we have
\begin{align}
\sum_{i,j,k\geq 0} \frac{(-1)^iq^{i^2+2j^2+k^2+2ij+2ik+2jk+3i}}{(q^2;q^2)_i(q^2;q^2)_j(q^2;q^2)_k}&=\frac{(q^6;q^4)_\infty}{(q;q^4)_\infty(q^7;q^4)_\infty},  \label{introd-nonmodular-1} \\
\sum_{i,j,k\geq 0}\frac{(-1)^kq^{2i^2+2j^2+9k^2+2ij+6ik+6jk+i+6j+6k}}{(q^2;q^2)_i(q^2;q^2)_j(q^6;q^6)_k}&=\frac{1+q^3}{(1-q^{9})(q^5,q^{7};q^6)_\infty}.\label{intro-nonmodular-2}
\end{align}
We will give some more examples in Section \ref{sec-remarks}.

The rest of this paper is organized as follows. We collect some auxiliary identities in Section \ref{sec-pre}. In Section \ref{sec-double} we prove some double sum identities of indices $(1,1)$, $(1,2)$ and $(1,4)$. In Section \ref{sec-triple} we prove some triple sum identities of indices $(1,1,1)$,  $(1,1,3)$, $(1,2,2)$ and $(1,2,4)$. Section \ref{sec-quadruple} includes a group of quadruple sum identities of index $(1,1,1,2)$. Finally, in Section \ref{sec-remarks} we present some identities in which the product side is non-modular. We then discuss some byproducts of our results including some single sum Rogers-Ramanujan type identities derived from our new identities. We also present a conjectural identity of index $(1,3)$, which can be regarded as a new companion of \eqref{K-conj-1}. It should be emphasized that the methods of proving our identities are not unique. For example, we can prove some identities by direct summation, constant term method or integral method. We give brief discussions of alternative methods using several examples in the last section.

\section{Preliminaries}\label{sec-pre}
Throughout this paper we will denote $\zeta_n=e^{2\pi i/n}$.

For any series $f(z)=\sum_{n=-\infty}^\infty a(n)z^n$, we define the constant term extractor
\begin{align}\label{CT-defn}
    \CT f(z)=a(0).
\end{align}
Note that for any nonzero integers $\alpha$ and $\beta$, we have
$$\CT f(\alpha z^\beta)=\CT f(z).$$
When using the constant term method, the first step in our deductions of identities is to express the sum sides as constant terms of some infinite products. For this we will rely on Euler's identities \cite[Corollary 2.2]{Andrews-book}
\begin{align}
&\sum_{n=0}^\infty\frac{z^n}{(q;q)_{n}}=\frac{1}{(z;q)_\infty}, \quad |z|<1, \label{Euler-1} \\
&\sum_{n=0}^\infty\frac{q^{(n^2-n)/2}z^n}{(q;q)_{n}}=(-z;q)_\infty,  \label{Euler-2}
\end{align}
and the Jacobi triple product identity \cite[Theorem 2.8]{Andrews-book}
\begin{align}
	&\sum_{n=-\infty}^\infty(-1)^nq^{(n^2-n)/2}z^n=(q,z,q/z;q)_\infty, \quad z\neq 0. \label{Jacobi}
\end{align}
Euler's identities are corollaries of the $q$-binomial theorem \cite[Theorem 2.1]{Andrews-book}:
\begin{align}\label{eq-qbinomial}
    \sum_{n=0}^\infty \frac{(a;q)_nz^n}{(q;q)_n}=\frac{(az;q)_\infty}{(z;q)_\infty}, \quad |z|<1.
\end{align}
We define the $q$-binomial coefficient or Gaussian coefficient as
$${n\brack m}={n \brack m}_q:=\left\{\begin{array}{ll}
\frac{(q;q)_n}{(q;q)_m (q;q)_{n-m}}, & 0\leq m \leq n, \\
0, & \text{otherwise}. \end{array}\right.$$
Now we can state a finite form of \eqref{Euler-2} \cite[p.\ 36, Theorem 3.3]{Andrews-book}:
\begin{align}\label{finite-Jacobi}
    (-z;q)_n=\sum_{k=0}^n {n\brack k} z^kq^{k(k-1)/2}.
\end{align}

We need the $q$-Gauss summation formula \cite[Eq.\ (\uppercase\expandafter{\romannumeral2}.8)]{GR-book}:
\begin{align}\label{eq-Gauss}
{}_2\phi_1 \bigg(\genfrac{}{}{0pt}{} {a,b}{c};q,\frac{c}{ab}\bigg)=\frac{(c/a,c/b;q)_\infty}{(c,c/ab;q)_\infty}.
\end{align}

Another key step in our proofs is to reduce multiple sums to some single sums, and then we can use some known identities in the literature. For our purposes, the following identities will be employed:
\begin{align}
        &\sum_{n=0}^\infty \frac{q^{n^2+n}(-q^{-1};q^2)_n}{(q^2;q^2)_n}=\frac{1}{(q,q^5,q^6;q^8)_\infty},\quad \text{(\cite[Eq.\ (2.22)]{Göllnitz})} \label{Gollnitz-(2.22)}\\
         &\sum_{n=0}^\infty \frac{{(-1)^n}q^{n^2+2n}(q;q^2)_n}{(q^4;q^4)_n}=\frac{(q^3;q^3)_\infty(q^{12};q^{12})_\infty}{(q^4;q^4)_\infty(q^6;q^6)_\infty},\quad  \text{(\cite[Entry\ 4.2.11]{Lost2})} \label{Wang-(1.22)} \\
        &\sum_{n=0}^\infty \frac{(-1)^nq^{n^2}(q;q^2)_n}{(q^2;q^2)_n^2}=\frac{(q;q^2)_\infty}{(q^2;q^2)_\infty},\quad  \text{(\cite[Entry\ 4.2.6]{Lost2})} \label{Ramanujan[10,Entry 4.2.6]}\\
        &\sum_{n=0}^\infty\frac{(-q;q^2)_n q^{n(n+1)}}{(q^2;q^2)_n}=\frac{1}{(q^2,q^3,q^7;q^8)_\infty},\quad \text{(\cite[Eq.\ (2.24)]{Göllnitz})} \label{Gollnitz-(2.24)}\\
        &\sum_{n=0}^\infty\frac{(-q;q^2)_n q^{n(n+2)}}{(q^4;q^4)_n}=\frac{(q^6;q^{12})_\infty}{(q^3,q^4,q^8,q^9;q^{12})_\infty},\quad  \text{(\cite[Entry\ 5.3.7]{Lost2})} \label{Ramanujan [10,Entry 5.3.7]}\\
        &\sum_{n=0}^\infty\frac{(-1)^n q^{n(3n+2)}}{(-q;q^2)_{n+1}(q^4;q^4)_n}=\frac{(q,q^4,q^5;q^5)_\infty}{(q^2;q^2)_\infty},\quad  \text{(\cite[Entry\ 11.2.7]{Lost1})} \label{Ramanujan [10,Entry 11.2.7]}\\
        &\sum_{n=0}^\infty\frac{q^{{n(3n+1)}/2}}{(q;q)_n(q;q^2)_{n+1}}=\frac{(q^4,q^6,q^{10};q^{10})_\infty}{(q;q)_\infty},\quad  \text{(\cite[ p.\ 330\ (2),\ line\ 1]{Rogers1917})} \label{(Rogers [37, p. 330 (2), line 1])} \\
        &\sum_{n=0}^\infty\frac{(a;q)_nq^{{n(n-1)}/2}\gamma^n}{(b;q)_n(q;q)_{n}}=\frac{(-\gamma;q)_\infty}{(b;q)_\infty}\sum_{n=0}^\infty\frac{({-a\gamma}/b;q)_nq^{{n(n-1)}/2}(-b)^n}{(-\gamma;q)_n(q;q)_{n}},\nonumber \\
        &\qquad \qquad \qquad \qquad \qquad \qquad \qquad \qquad  \text{(\cite[Eq.\ (6.1.3)]{MSZ2008})} \label{(Laughlin-(6.1.3))} \\
        &\sum_{n=0}^\infty\frac{q^{n^2-n}}{(q;q)^2_n}=\frac{2}{(q;q)_\infty}, \label{5yue14ri}  \\
	&\sum_{n=0}^\infty\frac{(-1)^n(-q;q^2)_n q^{n^2}}{(q^4;q^4)_n}=(q;q^2)_\infty(q^2;q^4)_\infty,\quad  \text{(S.\ 4)} \label{Slater4}\\
	&\sum_{n=0}^\infty\frac{q^{2n^2+n}}{(q;q)_{2n+1}}=(-q;q)_\infty,\quad  \text{(S.\ 5)} \label{Slater5}\\
	&\sum_{n=0}^\infty\frac{(-q;q)_n q^{n(n-1)/2}}{(q;q)_n}=\frac{(q^4;q^4)_\infty}{(q;q)_\infty}+\frac{(-q;q^2)_\infty}{(q;q^2)_\infty},\quad  \text{(S.\ 13)} \label{Slater13}\\
	&\sum_{n=0}^\infty\frac{(-1)^n q^{n(3n-2)}}{(-q;q^2)_n(q^4;q^4)_n}=\frac{(q,q^4,q^5;q^5)_\infty}{(q^2;q^2)_\infty},\quad  \text{(S.\ 15)} \label{Slater15}\\
	&\sum_{n=0}^\infty\frac{(-1)^n q^{3n^2}}{(-q;q^2)_n(q^4;q^4)_n}=\frac{(q^2,q^3,q^5;q^5)_\infty}{(q^2;q^2)_\infty},\quad  \text{(S.\ 19)} \label{Slater19}\\
	&\sum_{n=0}^\infty\frac{(-q;q^2)_n q^{n^2}}{(q^4;q^4)_n}=\frac{(q^3,q^3,q^6;q^6)_\infty(-q;q^2)_\infty}{(q^2;q^2)_\infty},\quad  \text{(S.\ 25)} \label{Slater25}\\
	&\sum_{n=0}^\infty\frac{(-q;q^2)_n q^{2n(n+1)}}{(q;q^2)_{n+1}(q^4;q^4)_n}=\frac{(-q,-q^5,q^6;q^6)_\infty}{(q^2;q^2)_\infty},\quad  \text{(S.\ 27)} \label{Slater27}\\
 &\sum_{n=0}^\infty \frac{(q;q^2)_nq^{2n^2}}{(-q;q)_{2n}(q^2;q^2)_n}=\frac{(q^3,q^3,q^6;q^6)_\infty}{(q^2;q^2)_\infty}, \quad \text{\cite[Eq.\ (2.13)]{BMS}}\label{eq-BMS} \\
	&\sum_{n=0}^\infty\frac{(-q^2;q^2)_n q^{n(n+1)}}{(q;q)_{2n+1}}=\frac{(-q,-q^5,q^6;q^6)_\infty(-q^2;q^2)_\infty}{(q^2;q^2)_\infty},\quad  \text{(S.\ 28)} \label{Slater28}\\
	&\sum_{n=0}^\infty\frac{(-q;q^2)_n q^{n^2}}{(q;q)_{2n}}=\frac{(-q^2,-q^4,q^6;q^6)_\infty(-q;q^2)_\infty}{(q^2;q^2)_\infty},\quad  \text{(S.\ 29)} \label{Slater29}\\
	&\sum_{n=0}^\infty\frac{ q^{2n(n+1)}}{(-q;q)_{2n+1}(q^2;q^2)_n}=\frac{(q,q^6,q^7;q^7)_\infty}{(q^2;q^2)_\infty},\quad  \text{(S.\ 31)} \label{Slater31}\\
	&\sum_{n=0}^\infty\frac{ q^{2n(n+1)}}{(-q;q)_{2n}(q^2;q^2)_n}=\frac{(q^2,q^5,q^7;q^7)_\infty}{(q^2;q^2)_\infty},\quad  \text{(S.\ 32)} \label{Slater32}\\
	&\sum_{n=0}^\infty\frac{ q^{2n^2}}{(-q;q)_{2n}(q^2;q^2)_n}=\frac{(q^3,q^4,q^7;q^7)_\infty}{(q^2;q^2)_\infty},\quad  \text{(S.\ 33)} \label{Slater33}\\
	&\sum_{n=0}^\infty\frac{(-q;q^2)_n q^{n(n+2)}}{(q^2;q^2)_n}=\frac{1}{(q^3,q^4,q^5;q^8)_\infty},\quad  \text{(S.\ 34)} \label{Slater34}\\
	&\sum_{n=0}^\infty\frac{(-q;q^2)_n q^{n^2}}{(q^2;q^2)_n}=\frac{1}{(q,q^4,q^7;q^8)_\infty},\quad  \text{(S.\ 36)} \label{Slater36}\\
	&\sum_{n=0}^\infty\frac{ q^{2n(n+1)}}{(q;q)_{2n+1}}=\frac{(q^3,q^5,q^8;q^8)_\infty(q^2,q^{14};q^{16})_\infty}{(q;q)_\infty},\quad  \text{(S.\ 38)} \label{Slater38}\\
	&\sum_{n=0}^\infty\frac{ q^{2n^2}}{(q;q)_{2n}}=\frac{(q,q^7,q^8;q^8)_\infty(q^6,q^{10};q^{16})_\infty}{(q;q)_\infty},\quad  \text{(S.\ 39)} \label{Slater39}\\
 &\sum_{n=0}^\infty\frac{q^{(3n^2+3n)/2}}{(q;q)_{n}(q;q^2)_{n+1}}=\frac{(q^2,q^8,q^{10};q^{10})_\infty}{(q;q)_\infty},\quad  \text{(S.\ 44)} \label{Slater44}\\
 &\sum_{n=0}^\infty\frac{q^{(3n^2-n)/2}}{(q;q)_{n}(q;q^2)_{n}}=\frac{(q^4,q^6,q^{10};q^{10})_\infty}{(q;q)_\infty},\quad  \text{(S.\ 46)} \label{Slater46}\\
 &\sum_{n=0}^\infty\frac{ q^{n^2+2n}(-q;q^2)_n}{(q;q)_{2n+1}}=\frac{(q^2,q^{10},q^{12};q^{12})_\infty}{(q;q)_\infty},\quad  \text{(S.\ 50)} \label{Slater50}\\
&\sum_{n=0}^\infty \frac{q^{n^2}(-q;q^2)_n}{(q^2;q^2)_{2n}}=\sum_{n=0}^\infty \frac{q^{n^2}}{(q;q^2)_n(q^4;q^4)_n} \nonumber \\
&\quad \quad \quad \quad =\frac{(q^2;q^2)_\infty (q^{14};q^{14})_\infty (q^3,q^{11},q^{17},q^{25};q^{28})_\infty}{(q;q)_\infty (q^{28};q^{28})_\infty (q^4,q^{12},q^{16},q^{24};q^{28})_\infty}, \quad  \text{(S.\ 117)} \label{Slater117} \\
%=\frac{J_2J_{14}J_{3,28}J_{11,28}}{J_1J_{28}J_{4,28}J_{12,28}},
&\sum_{n=0}^\infty \frac{q^{n^2+2n}(-q;q^2)_n}{(q^2;q^2)_{2n}}=\sum_{n=0}^\infty \frac{q^{n^2+2n}}{(q;q^2)_n(q^4;q^4)_n}\nonumber \\
&\quad \quad \quad \quad =\frac{(q^2;q^2)_\infty (q,q^{13},q^{14};q^{14})_\infty (q^{12},q^{16};q^{28})_\infty}{(q;q)_\infty (q^4;q^4)_\infty}, \quad \text{(S.\ 118)} \label{Slater118} \\
%=\frac{J_2J_{1,14}J_{12,28}}{J_1J_4J_{28}},
&\sum_{n=0}^\infty \frac{q^{n^2+2n}(-q;q^2)_{n+1}}{(q^2;q^2)_{2n+1}}=\sum_{n=0}^\infty \frac{q^{n^2+2n}}{(q;q^2)_{n+1}(q^4;q^4)_n} \nonumber \\
&\quad \quad \quad \quad = \frac{(q^2;q^2)_\infty (q^5,q^9,q^{14};q^{14})_\infty (q^4,q^{24};q^{28})_\infty}{(q;q)_\infty (q^4;q^4)_\infty}. \quad  \text{(S.\ 119)} \label{Slater119}
%=\frac{J_2J_{4,28}J_{5,14}}{J_1J_4J_{28}}.
\end{align}
Note that \eqref{5yue14ri} follows from the following identity \cite[Eq.\ (6.1.4)]{MSZ2008}:
\begin{align}
    \sum_{n=0}^\infty \frac{(a;q)_nq^{n(n-1)/2}\gamma^n}{(q;q)_n}
    =(-\gamma;q)_\infty \sum_{n=0}^\infty \frac{(-a\gamma)^nq^{n(n-1)}}{(-\gamma;q)_n(q;q)_n}. \label{eq-parameter}
\end{align}
Indeed, by setting $a=q^{-1}$ and $\gamma=-q$ we obtain \eqref{5yue14ri}.

\section{Identities involving double sums}\label{sec-double}
In this section, we present some new double sum identities of indices $(1,1)$, $(1,2)$ and $(1,4)$. We do not find new identities of other indices except one of index $(1,3)$ (see Conjecture \ref{conj-13}). To save space, when the proof is straightforward, we will only give a sketch and omit the details.

\subsection{Identities of index $(1,1)$}
\begin{theorem}\label{thm-11-alpha}
For any $\alpha>0$ we have
\begin{align}\label{id-alpha}
\sum_{i,j\geq 0} \frac{q^{\frac{\alpha}{2}i^2+(1-\alpha)ij+\frac{\alpha}{2}j^2+(1-\frac{\alpha}{2})i+\frac{\alpha}{2}j}}{(q;q)_i(q;q)_j}=\frac{(q^{2\alpha};q^{2\alpha})_\infty^2}{(q;q)_\infty (q^\alpha;q^\alpha)_\infty}.
\end{align}
\end{theorem}
\begin{proof}
Using \eqref{Euler-1}--\eqref{Jacobi}, we have
\begin{align}
\mathrm{LHS}&=\sum_{i,j\geq 0} \frac{q^{\frac{1}{2}i^2+i+\frac{1}{2}j^2+\frac{\alpha-1}{2}(i-j)^2-\frac{\alpha}{2}(i-j)}}{(q;q)_i(q;q)_j} \nonumber \\
&=\mathrm{CT}\left[\sum_{i\geq 0} \frac{q^{\frac{1}{2}i^2+i}z^i}{(q;q)_i} \sum_{j\geq 0} \frac{q^{\frac{1}{2}j^2}z^{-j}}{(q;q)_j} \sum_{k=-\infty}^\infty q^{\frac{\alpha-1}{2}k^2-\frac{\alpha}{2}k}z^{-k} \right] \nonumber \\
&=\mathrm{CT}\left[(-q^{\frac{3}{2}}z;q)_\infty (-q^{\frac{1}{2}}z^{-1};q)_\infty (-q^{-\frac{1}{2}}z^{-1},-q^{\frac{2\alpha-1}{2}}z,q^{\alpha-1};q^{\alpha-1})_\infty  \right] \nonumber \\
&=\mathrm{CT} \left[(-qz;q)_\infty (-qz^{-1};q)_\infty (-z^{-1},-q^{\alpha-1}z,q^{\alpha-1};q^{\alpha-1})_\infty\right] \nonumber \\
&\quad \quad \quad \quad \quad \text{(Here we replaced $z$ by $q^{-\frac{1}{2}}z$)} \nonumber \\
&=\frac{1}{(q;q)_\infty} \mathrm{CT}\left[\frac{1}{1+z} (-z,-q/z,q;q)_\infty (-z^{-1},-q^{\alpha-1}z,q^{\alpha-1};q^{\alpha-1})_\infty  \right] \nonumber \\
&=\frac{1}{(q;q)_\infty} \mathrm{CT} \left[ \left(\sum_{k=0}^\infty (-1)^kz^k \right) \left(\sum_{i=-\infty}^\infty q^{\binom{i}{2}}z^i \right) \left( \sum_{j=-\infty}^\infty q^{(\alpha-1)\binom{j}{2}}z^{-j}\right)\right] \nonumber \\
&=\frac{1}{(q;q)_\infty} \sum_{k=0}^\infty \sum_{j-i=k} (-1)^k q^{\binom{i}{2}+(\alpha-1)\binom{j}{2}} \nonumber \\
&=\frac{1}{(q;q)_\infty} \sum_{k=0}^\infty \sum_{i=-\infty}^\infty (-1)^k q^{\binom{i}{2}+(\alpha-1)\binom{i+k}{2}}.  \label{LHS-new}
\end{align}
Note that
\begin{align}
&\sum_{k=-\infty}^0 \sum_{i=-\infty}^\infty (-1)^k q^{\binom{i}{2}+(\alpha-1)\binom{i+k}{2}} =\sum_{k=0}^\infty \sum_{i=-\infty}^\infty (-1)^k q^{\binom{i}{2}+(\alpha-1)\binom{i-k}{2}} \nonumber \\
&=\sum_{k=0}^\infty \sum_{i=-\infty}^\infty (-1)^k q^{\binom{-i+1}{2}+(\alpha-1)\binom{-i+1-k}{2}}  \quad \text{( $i$ was replaced by $-i+1$)}\nonumber \\
&=\sum_{k=0}^\infty \sum_{i=-\infty}^\infty (-1)^k q^{\binom{i}{2}+(\alpha-1)\binom{i+k}{2}}.
\end{align}
Therefore, we have
\begin{align}
&\sum_{k=0}^\infty \sum_{i=-\infty}^\infty (-1)^k q^{\binom{i}{2}+(\alpha-1)\binom{i+k}{2}} \nonumber \\
&=\frac{1}{2}\left(\sum_{k=-\infty}^\infty \sum_{i=-\infty}^\infty (-1)^kq^{\binom{i}{2}+(\alpha-1)\binom{i+k}{2}}+\sum_{i=-\infty}^\infty q^{\alpha \binom{i}{2}} \right) \nonumber \\
&=\frac{1}{2}\left(\sum_{j=-\infty}^\infty \sum_{i=-\infty}^\infty (-1)^{i+j}q^{\binom{i}{2}+(\alpha-1)\binom{j}{2}}+\sum_{i=-\infty}^\infty q^{\alpha \binom{i}{2}} \right) \nonumber \\
&=\frac{1}{2}\left( \Big(\sum_{i=-\infty}^\infty (-1)^iq^{\binom{i}{2}} \Big) \Big( \sum_{j=-\infty}^\infty (-1)^jq^{(\alpha-1)\binom{j}{2}} \Big)+\sum_{i=-\infty}^\infty q^{\alpha \binom{i}{2}} \right) \nonumber \\
&=\frac{1}{2}(q^\alpha,-1,-q^\alpha;q^\alpha)_\infty =\frac{(q^{2\alpha};q^{2\alpha})_\infty^2}{(q^\alpha;q^\alpha)_\infty}. \label{LHS-result}
\end{align}
Substituting \eqref{LHS-result} into \eqref{LHS-new}, we obtain the desired identity.
\end{proof}

\begin{rem}\label{rem-Warnaar-pq}
It is easy to see that Theorem \ref{thm-11-alpha} is equivalent to the identity
\begin{align}\label{eq-general-pq-id}
   \sum_{i,j\geq 0} \frac{p^{\binom{i-j}{2}} q^{ij+i}}{(q;q)_i(q;q)_j}
   =\frac{(p^2;p^2)_{\infty}(-p;p)_{\infty}}{(q;q)_{\infty}}, \quad \text{for $|p|,|q|<1$}.
\end{align}
Warnaar \cite{Warnaar} shared with us the following proof for \eqref{eq-general-pq-id}, which does not use the constant term method. 
\begin{proof}[Warnaar's proof of \eqref{eq-general-pq-id}] Replacing $(i,j)$ by  $(n,n-m)$, the double sum becomes
\begin{align} 
\sum_{m=-\infty}^{\infty} p^{\binom{m}{2}} g_m(q) \quad \text{where} \quad g_m(q):=\sum_{n=0}^{\infty}   \frac{q^{n(n-m+1)}}{(q;q)_n(q;q)_{n-m}}.
\end{align}
Here we adopt the standard convention that $1/(q;q)_k=0$ for $k$ a negative integer. Note that $g_m(q)=1+O(q)$ for $m\leq 0$ and $g_m(q)=q^m(1+O(q))$ for $m\geq 0$.

We claim that
\begin{align}\label{eq-W-claim}
g_m(q)=\frac{1}{(q;q)_{\infty}}\sum_{n=0}^{\infty}(-1)^n q^{\binom{n+1}{2}+(n+1)m}.
\end{align}
The Durfee rectangle identity (a special case of $q$-Chu-Vandermonde) says that
$$
\sum_{n=0}^{\infty}   \frac{q^{n(n-m)}}{(q;q)_n(q;q)_{n-m}} = \frac{1}{(q;q)_{\infty}},
$$
and what we have here is a slight variation this result:
\begin{align}
&g_m(q)=\sum_{n=0}^{\infty}   \frac{q^{n(n-m)}(q^n-1+1)}{(q;q)_n(q;q)_{n-m}}
=\sum_{n=0}^{\infty}   \frac{q^{n(n-m)}}{(q;q)_n(q;q)_{n-m}}-
 \sum_{n=1}^{\infty}   \frac{q^{n(n-m)}}{(q;q)_{n-1}(q;q)_{n-m}} \nonumber \\
&=\frac{1}{(q;q)_{\infty}}-q^{1-m} g_{m-1}(q).
\end{align}
In other words, 
\begin{align*}
g_m(q)+q^{1-m}g_{m-1}(q)=\frac{1}{(q;q)_{\infty}}.
\end{align*}
This implies that 
\begin{align}\label{g-rec-add}
g_m(q)=(-1)^k q^{\binom{k}{2}+mk}g_{m+k}(q)+\frac{1}{(q;q)_{\infty}}\sum_{n=0}^{k-1}(-1)^n q^{\binom{n+1}{2}+(n+1)m},
\end{align}
where $k\geq 0$.
%(Note that for $k=0$ you simply get $g_m(q)=g_m(q)$ and that for $k=1$ you get $g_m(q)+q^m g_{m+1}(q)=\frac{q^m}{(q;q)_{\infty}}$ which is the same as the original recursion thanks the symmetry of $g_m$ under negation of $m$).
If we take the $k\rightarrow \infty$ limit of \eqref{g-rec-add} we get \eqref{eq-W-claim}.

By \eqref{eq-W-claim} we know that \eqref{eq-general-pq-id} is reduced to proving that
\begin{align}\label{id-theta-add}
   \sum_{m=-\infty}^{\infty}\sum_{n=0}^{\infty}
            (-1)^n p^{\binom{m}{2}} q^{\binom{n+1}{2}+(n+1)m}
   = (p^2;p^2)_{\infty}(-p;p)_{\infty}.
\end{align}
This can be done in a way analogous to \eqref{LHS-result}.
Let
 \begin{align}
   f_m:=
     \sum_{n=0}^{\infty}(-1)^n p^{\binom{m}{2}} q^{\binom{n+1}{2}+(n+1)m}.
\end{align}
   Then
\begin{align}
  & f_m+f_{1-m}
     = (p/q)^{\binom{m}{2}}
        \left(\sum_{n=0}^{\infty}(-1)^n q^{\binom{n+m+1}{2}}
         +\sum_{n=0}^{\infty}(-1)^n q^{\binom{n-m+2}{2}}\right) \nonumber \\
&= (-1)^m (p/q)^{\binom{m}{2}} \Big(
       \sum_{n=m}^{\infty} (-1)^n q^{\binom{n+1}{2}}-\sum_{n=1-m}^{\infty}  (-1)^n q^{\binom{n+1}{2}}\Big)
     = p^{\binom{m}{2}}.
\end{align}
Therefore, the left side of \eqref{id-theta-add} becomes
\begin{align*}
\sum_{m=0}^\infty (f_{m}+f_{1-m})=\sum_{m=0}^\infty p^{\binom{m}{2}}=(p^2;p^2)_\infty (-p;p)_\infty. \quad \quad \qedhere
\end{align*}
\end{proof}
\end{rem}

\begin{theorem}\label{thm-3.2}
We have
\begin{align}
\sum_{i,j\geq 0} \frac{q^{i^2+j^2+2ij+2i+j}}{(q^2;q^2)_i(q^2;q^2)_j}&=\frac{1}{(q^2,q^3;q^5)_\infty}.\label{(2,2)-6}
\end{align}
\end{theorem}
\begin{proof}
Summing over $j$ using \eqref{Euler-2} first, the identity can be proved by using \eqref{Slater16}. We omit the details.
\end{proof}

\begin{theorem}\label{thm-3.3}
    We have
    \begin{align}
    \sum_{i,j\geq 0} \frac{(-1)^iq^{j^2+2ij+i+2j}}{(q^2;q^2)_i(q^2;q^2)_j}&=\frac{(q,q^7,q^9,q^{15};q^{16})_\infty}{(q^2,q^4,q^6,q^{10},q^{12},q^{14};q^{16})_\infty},\label{(2,2)-9}\\
        \sum_{i,j\geq 0} \frac{(-1)^iq^{j^2+2ij+i}}{(q^2;q^2)_i(q^2;q^2)_j}&=\frac{(q^3,q^5,q^{11},q^{13};q^{16})_\infty}{(q^2,q^4,q^6,q^{10},q^{12},q^{14};q^{16})_\infty},\label{(2,2)-10}\\
\sum_{i,j\geq 0} \frac{(-1)^iq^{j^2/2+ij+i-j/2}}{(q;q)_i(q;q)_j}&=\frac{1}{(q^2;q^4)_\infty}+\frac{(q^2;q^4)_\infty}{(q;q^2)_\infty},\label{(2,2)-11}\\
% \sum_{i,j\geq 0} \frac{(-1)^iq^{j^2/2+ij+j/2}}{(q;q)_i(q;q)_j}&=\frac{1}{2}(-q;q^2)_\infty,\label{(2,2)-12} \\
\sum_{i,j\geq 0}\frac{q^{j^2+2ij+j}a^i}{(q^2;q^2)_i(q^2;q^2)_j}&=\frac{1}{(a,q^2;q^4)_\infty}. \label{thm16-(2,2)}
\end{align}
\end{theorem}
\begin{proof}
Summing over $i$ using \eqref{Euler-1} first, the first three identities can be proved by using \eqref{Slater34}, \eqref{Slater36} and \eqref{Slater13}, respectively. Summing over $j$  using \eqref{Euler-2} first, we can prove \eqref{thm16-(2,2)}. We omit the details.
\end{proof}

\begin{theorem}
We have
\begin{align}\label{eq-u-id}
    \sum_{i,j\geq 0} \frac{q^{i^2-2ij+j^2+j-i}a^{i+j}}{(q^4;q^4)_i(q^4;q^4)_j}=\frac{1}{(a;q^2)_\infty}.
\end{align}
\end{theorem}
\begin{proof}
By \eqref{Euler-1} we have
\begin{align}
    \frac{1}{(a;q^2)_\infty} =\sum_{n=0}^\infty \frac{a^n}{(q^2;q^2)_n}. \label{u-expan}
\end{align}
Comparing the coefficients of $a^n$ on both sides of \eqref{eq-u-id}, it suffices to prove that
\begin{align}\label{coeff-equal}
    \sum_{i+j=n} \frac{q^{i^2-2ij+j^2+j-i}}{(q^4;q^4)_i(q^4;q^4)_j}=\frac{1}{(q^2;q^2)_n}.
\end{align}
This is equivalent to
\begin{align}\label{coeff-equivalent}
   q^{n^2+n} \sum_{i=0}^n {n \brack i}_{q^4} q^{4i^2-(4n+2)i}=(-q^2;q^2)_n.
\end{align}

 We denote the left and right sides of \eqref{coeff-equivalent} as $L_n(q)$ and $R_n(q)$, respectively. Obviously, we have for $n\geq 1$ that
 \begin{align}
     R_n(q)=(1+q^{2n})R_{n-1}(q). \label{R-rec}
 \end{align}
 For $n\geq 1$ we have
\begin{align*}
     &L_n(q)-q^{2n}L_{n-1}(q)= q^{n^2+n} \sum_{i=0}^n {n\brack i}_{q^4} q^{4i^2-(4n+2)i} -q^{n^2+n} \sum_{i=0}^{n-1} {n-1 \brack i}_{q^4} q^{4i^2-(4n-2)i} \nonumber \\
     &=q^{n^2-n}+q^{n^2+n} \sum_{i=0}^{n-1} \left({ n\brack i}_{q^4}-{n-1 \brack i}_{q^4} q^{4i} \right) q^{4i^2-(4n+2)i} \nonumber \\
     &=q^{n^2-n}+q^{n^2+n} \sum_{i=0}^{n-1} {n-1\brack i-1}_{q^4} q^{4i^2-(4n+2)i} \nonumber \\
    &=q^{n^2-n}+q^{n^2+n} \sum_{i=1}^{n} {n-1 \brack n-i-1} q^{4(n-i)^2-(4n+2)(n-i)} \nonumber \\
     &=q^{n^2-n}\left(1+\sum_{i=1}^{n-1} {n-1 \brack i}_{q^4} q^{4i^2-(4n-2)i} \right) \nonumber \\
     &=L_{n-1}(q).
 \end{align*}
 This proves that
 \begin{align}
     L_n(q)=(1+q^{2n})L_{n-1}(q). \label{L-rec}
 \end{align}
 Note that $L_0(q)=R_0(q)=1$. From \eqref{L-rec} and \eqref{R-rec} we know that for any $n\geq 0$, we always have $L_n(q)=R_n(q)$. This proves \eqref{coeff-equal} and hence the desired identity.
\end{proof}
\begin{rem}\label{rem-Warnaar}
Warnaar \cite{Warnaar} informed us that \eqref{coeff-equivalent} can be proved in the following different way using  the Rogers-Szeg\H{o} polynomial
\begin{align}\label{RS-defn}
    H_n(t;q):=\sum_{j=0}^n t^j{n\brack j}.
\end{align}
From \cite[Eq.\ (8.13)]{TAMS} we know that 
\begin{align}
\sum\limits_{n=0}^\infty \frac{x^n H_n(t^2;q^4)}{(q^2;q^2)_n} = \frac{(t^2x^2q^2;q^4)_{\infty}}{(x,t^2x;q^2)_{\infty}}.
\end{align}
If we set $t=q$, we get
\begin{align}
\sum\limits_{n=0}^\infty \frac{x^n H_n(q^2;q^4)}{(q^2;q^2)_n} = \frac{(x^2q^4;q^4)_{\infty}}{(x,q^2x;q^2)_{\infty}} = \frac{(-xq^2;q^2)_{\infty}}{(x;q^2)_{\infty}}=\sum_{n=0}^\infty \frac{x^n(-q^2;q^2)_n}{(q^2;q^2)_n}.
\end{align}
Comparing the coefficients of $x^n$ on both sides, we deduce that
\begin{align}\label{H-eval}
    H_n(q^2;q^4)=(-q^2;q^2)_n.
\end{align}
After replacing $q$ by $1/q$, we see that \eqref{coeff-equivalent} is equivalent to
\begin{align}\label{coeff-equivalent-1}
 \sum_{i=0}^n {n \brack i}_{q^{-4}} q^{-4i^2+(4n+2)i}=q^{n^2+n}(-q^{-2};q^{-2})_n=(-q^2;q^2)_n. 
\end{align}
We have 
\begin{align}
    \mathrm{LHS}=\sum_{i=0}^n {n \brack i}_{q^{4}}\frac{q^{-2i(2n-i+1)}}{q^{-2i(1+i)}} q^{-4i^2+(4n+2)i}=\sum_{i=0}^n {n \brack i}_{q^4} q^{2i}.
\end{align}
Now by \eqref{H-eval} we prove \eqref{coeff-equivalent-1}.
\end{rem}

\begin{theorem}\label{thm-14}
We have
\begin{align}
\sum_{i,j\geq 0}\frac{q^{i^2+j^2+2ij+4i}}{(q^8;q^8)_i(q^8;q^8)_j}&=\frac{(q^2;q^2)_\infty}{(q,q^4;q^{5})_\infty(q^4;q^4)_\infty}, \label{(8,8)-2}\\
\sum_{i,j\geq 0}\frac{q^{i^2+j^2+2ij+2i+6j}}{(q^8;q^8)_i(q^8;q^8)_j}&=\frac{(q^2;q^{2})_\infty}{(q^2,q^3;q^{5})_\infty(q^4;q^4)_\infty}. \label{(8,8)-3}
\end{align}
\end{theorem}

\begin{proof}
Let
\begin{align}
    F(u,v)=F(u,v;q):=\sum_{i,j\geq 0} \frac{u^iv^jq^{i^2+j^2+2ij-i-j}}{(q^8;q^8)_i(q^8;q^8)_j}=\sum_{i,j\geq 0} \frac{u^iv^jq^{(i+j)(i+j-1)}}{(q^8;q^8)_i(q^8;q^8)_j}.\nonumber
\end{align}
We have
\begin{align}
F(u,v)&=\CT \left[\sum_{i\geq 0} \frac{u^iz^i}{(q^8;q^8)_i} \sum_{j\geq 0} \frac{v^jz^j}{(q^8;q^8)_j} \sum_{k=-\infty}^\infty z^{-k}q^{k^2-k} \right] \nonumber \\
&=\CT\left[\frac{1}{(uz;q^8)_\infty (vz;q^8)_\infty} \sum_{k=-\infty}^\infty z^{-k}q^{k^2-k}\right].\nonumber
\end{align}
Let $(u,v)=(q^5,q)$. We have
\begin{align}
    &F(q^5,q)=\CT\left[\frac{1}{(zq;q^8)_\infty (zq^5;q^8)_\infty} \sum_{k=-\infty}^\infty z^{-k}q^{k^2-k}\right] \nonumber \\
    &=\CT\left[\frac{1}{(zq;q^4)_\infty}  \sum_{k=-\infty}^\infty z^{-k}q^{k^2-k}\right] = \CT\left[\sum_{n\geq 0} \frac{z^nq^n}{(q^4;q^4)_n}  \sum_{k=-\infty}^\infty z^{-k}q^{k^2-k}\right] \nonumber\\
    &=\sum_{n\geq 0} \frac{q^{n^2}}{(q^4;q^4)_n}.\nonumber
\end{align}
Now using \eqref{Slater20} we obtain \eqref{(8,8)-2}.

Similarly, setting $(u,v)=(q^3,q^7)$ and using \eqref{Slater16} we obtain \eqref{(8,8)-3}.
\end{proof}

\subsection{Identities of index $(1,2)$}
\begin{theorem}\label{thm-3}
	We have
	\begin{align}
		\sum_{i,j\geq 0}\frac{(-1)^jq^{3i^2/2+j^2+2ij-i/2}}{(q;q)_i(q^2;q^2)_j}&=\frac{1}{(q^2,q^8;q^{10})_\infty},\label{(1,2)-7}  \\
		\sum_{i,j\geq 0}\frac{(-1)^jq^{3i^2/2+j^2+2ij+i/2+2j}}{(q;q)_i(q^2;q^2)_j}&=\frac{1}{(q^2,q^8;q^{10})_\infty},\label{(1,2)-8}  \\
		\sum_{i,j\geq 0}\frac{(-1)^jq^{3i^2/2+j^2+2ij+3i/2+2j}}{(q;q)_i(q^2;q^2)_j}&=\frac{1}{(q^4,q^6;q^{10})_\infty},\label{(1,2)-9}  \\
		\sum_{i,j\geq 0}\frac{(-1)^jq^{3i^2/2+j^2+2ij+5i/2+4j}}{(q;q)_i(q^2;q^2)_j}&=\frac{1}{(q^4,q^6;q^{10})_\infty}.\label{(1,2)-10}
	\end{align}
\end{theorem}
\begin{proof}
We define
	\begin{align}
	F(u,v)=F(u,v;q):=\sum_{i,j\geq 0} \frac{u^iv^j(-1)^jq^{3i^2/2+j^2+2ij-3i/2-j}}{(q;q)_i(q^2;q^2)_j}.\nonumber
	\end{align}
	By \eqref{Euler-1}, \eqref{Euler-2} and \eqref{Jacobi} we have
	\begin{align}
		F(u,v)&=\CT \left[ \sum_{i\geq 0}\frac{(-1)^iu^i z^i q^{i^2/2-i/2}}{(q;q)_i} \sum_{j\geq 0}\frac{v^j z^j}{(q^2;q^2)_j}  \sum_{k=-\infty}^\infty (-1)^kq^{k^2-k}z^{-k} \right]  \nonumber \\
		&=\CT \left[ \frac{(uz;q)_\infty}{(vz;q^2)_\infty}\sum_{k=-\infty}^\infty (-1)^kq^{k^2-k}z^{-k} \right]. \nonumber
	\end{align}
Setting $(u,v)=(q,q)$, we have
	\begin{align}
		&F(q,q)=\CT \left[ (zq^2;q^2)_\infty\sum_{k=-\infty}^\infty (-1)^kq^{k^2-k}z^{-k} \right] \nonumber \\
		&=\CT \left[ \sum_{n=0}^\infty \frac{(-1)^n z^n q^{n^2+n}}{(q^2;q^2)_n} \sum_{k=-\infty}^\infty (-1)^kq^{k^2-k}z^{-k} \right] =\sum_{n=0}^\infty \frac{q^{2n^2}}{(q^2;q^2)_n}.\nonumber
	\end{align}
 Substituting \eqref{Rama-1} into the above identity, we get \eqref{(1,2)-7}.

Similarly,	setting $(u,v)=(q^2,q^3)$ and using \eqref{Rama-1} we obtain \eqref{(1,2)-8}. Setting $(u,v)=(q^3,q^3)$ and using \eqref{Rama-2} we obtain \eqref{(1,2)-9}. Setting $(u,v)=(q^4,q^5)$ and using \eqref{Rama-2} we obtain \eqref{(1,2)-10}.
\end{proof}

Summing over $j$ using \eqref{Euler-2} first, the first three identities can also be proved directly by using \eqref{Slater46}, \eqref{(Rogers [37, p. 330 (2), line 1])} and \eqref{Slater44}, respectively. We omit the details.
For the last identity, after summing over $j$, we find  that it is equivalent to the follow identity, which appears to be new and we state as a corollary.

\begin{corollary}
We have
\begin{align}
    \sum_{i\geq 0}\frac{q^{(3i^2+5i)/2}}{(q;q)_i(q;q^2)_{i+2}}&=\frac{(q^2,q^8,q^{10};q^{10})_\infty}{(q;q)_\infty}.
\end{align}
\end{corollary}

The following theorem provides a new identity which can be regarded as a companion to \eqref{Cao-Wang-3.8-1} and \eqref{Cao-Wang-3.8-2}.
\begin{theorem}\label{thm-4-new}
We have
\begin{align}
\sum_{i,j\geq 0}\frac{(-1)^iq^{i^2+2j^2+2ij-i-2j}a^{i+j}}{(q;q)_i(q^2;q^2)_j}&=(aq;q^2)_\infty.\label{(1,2)-new-1}
\end{align}
\end{theorem}
\begin{proof}
Recall the following identity \cite[Theorem 1.1]{Wei}:
\begin{align}\label{eq-Wei}
\sum_{i,j\geq 0} \frac{q^{i^2+2ij+2j^2-i-j}}{(q;q)_i(q^2;q^2)_j}x^iy^{2j}=(y;q)_\infty \sum_{k=0}^\infty \frac{(-x/y;q)_k}{(q;q)_k(y;q)_k}q^{\binom{k}{2}}y^k.
\end{align}
Setting $x=-a$ and $y=a^{1/2}q^{-1/2}$, we have
\begin{align}
\mathrm{LHS}&= (a^{1/2}q^{-1/2};q)_\infty\sum_{n=0}^\infty \frac{(a^{1/2}q^{1/2};q)_nq^{{(n^2-n)}/2}(a^{1/2}q^{-1/2})^n}{(q;q)_n(a^{1/2}q^{-1/2};q)_n}\nonumber \\
&= (a^{1/2}q^{-1/2};q)_\infty\sum_{n=0}^\infty \frac{(1-a^{1/2}q^{n-1/2})q^{{(n^2-n)}/2}(a^{1/2}q^{-1/2})^n}{(q;q)_n(1-a^{1/2}q^{-1/2})}\nonumber \\
&= (a^{1/2}q^{1/2};q)_\infty \Big(\sum_{n=0}^\infty \frac{q^{{(n^2-n)}/2}(a^{1/2}q^{-1/2})^n}{(q;q)_n}-a^{1/2}q^{-1/2}\sum_{n=0}^\infty \frac{q^{{(n^2-n)}/2}(a^{1/2}q^{1/2})^n}{(q;q)_n}\Big)\nonumber \\
&= (a^{1/2}q^{1/2};q)_\infty \Big((-a^{1/2}q^{-1/2};q)_\infty-a^{1/2}q^{-1/2}(-a^{1/2}q^{1/2};q)_\infty\Big)\nonumber \\
&=(a^{1/2}q^{1/2};q)_\infty(-a^{1/2}q^{1/2};q)_\infty=(aq;q^2)_\infty. \nonumber
\end{align}
This proves \eqref{(1,2)-new-1}.
\end{proof}

\begin{theorem}\label{thm-5}
We have
\begin{align}
 \sum_{i,j\geq 0}\frac{(-1)^iq^{i^2+j^2+2ij}}{(q^2;q^2)_i(q^4;q^4)_j}&=\frac{(q^3,q^{11},q^{14},q^{17},q^{25};q^{28})_\infty}{(q^4,q^{12},q^{16},q^{24};q^{28})_\infty},\label{(2,4)-12}  \\
\sum_{i,j\geq 0}\frac{(-1)^iq^{i^2+j^2+2ij+2j}}{(q^2;q^2)_i(q^4;q^4)_j}&=\frac{(q,q^{13},q^{14},q^{15},q^{27};q^{28})_\infty}{(q^4,q^8,q^{20},q^{24};q^{28})_\infty},\label{(2,4)-13}  \\
\sum_{i,j\geq 0}\frac{(-1)^iq^{i^2+j^2+2ij+2i+2j}}{(q^2;q^2)_i(q^4;q^4)_j}&=\frac{(q^5,q^9,q^{14},q^{19},q^{23};q^{28})_\infty}{(q^8,q^{12},q^{16},q^{20};q^{28})_\infty}.\label{(2,4)-14}
\end{align}
\end{theorem}
\begin{proof}
Summing over $i$ using \eqref{Euler-2} first, the identities can be proved by using \eqref{Slater117}, \eqref{Slater118} and  \eqref{Slater119}, respectively.  We omit the details.
\end{proof}

\begin{theorem}\label{thm-3.9}
We have
\begin{align}
\sum_{i,j\geq 0}\frac{q^{i^2+2j^2+2ij}}{(q^2;q^2)_i(q^4;q^4)_j}&=\frac{(-q;q^2)_\infty(q^3,q^4,q^7;q^7)_\infty}{(q^2;q^{2})_\infty},\label{(2,4)-1}  \\
\sum_{i,j\geq 0}\frac{q^{i^2+2j^2+2ij+2j}}{(q^2;q^2)_i(q^4;q^4)_j}&=\frac{(-q;q^2)_\infty(q^2,q^5,q^7;q^7)_\infty}{(q^2;q^{2})_\infty},\label{(2,4)-2}  \\
\sum_{i,j\geq 0}\frac{q^{i^2+2j^2+2ij+2i+2j}}{(q^2;q^2)_i(q^4;q^4)_j}&=\frac{(-q;q^2)_\infty(q,q^6,q^7;q^7)_\infty}{(q^2;q^{2})_\infty}.\label{(2,4)-3}
\end{align}
\end{theorem}
\begin{proof}
Summing over $i$ using \eqref{Euler-2} first, the identities can be proved by using \eqref{Slater33}, \eqref{Slater32} and \eqref{Slater31}, respectively.  We omit the details.
\end{proof}

\begin{theorem}\label{thm-3.10}
We have
\begin{align}
\sum_{i,j\geq 0}\frac{(-1)^jq^{i^2+2j^2+4ij+i}}{(q^2;q^2)_i(q^4;q^4)_j}&=\frac{1}{(q^4,q^6,q^8,q^{20},q^{22},q^{24};q^{28})_\infty},\label{(2,4)-15}  \\
\sum_{i,j\geq 0}\frac{(-1)^jq^{i^2+2j^2+4ij+i+4j}}{(q^2;q^2)_i(q^4;q^4)_j}&=\frac{1}{(q^2,q^8,q^{12},q^{16},q^{20},q^{26};q^{28})_\infty},\label{(2,4)-16}  \\
\sum_{i,j\geq 0}\frac{(-1)^jq^{i^2+2j^2+4ij+3i+4j}}{(q^2;q^2)_i(q^4;q^4)_j}&=\frac{1}{(q^4,q^{10},q^{12},q^{16},q^{18},q^{24};q^{28})_\infty}.\label{(2,4)-17}
\end{align}
\end{theorem}
\begin{proof}
Summing over $i$ using \eqref{Euler-2} first, the identities can be proved by using \eqref{Slater117}, \eqref{Slater118} and \eqref{Slater119}, respectively.  We omit the details.
\end{proof}

\begin{theorem}
We have
\begin{align}
\sum_{i,j\geq 0}\frac{(-1)^iq^{i^2+3j^2+2ij+2i+2j}}{(q^2;q^2)_i(q^4;q^4)_j}&=\frac{(q^3,q^7,q^{10},q^{13},q^{17};q^{20})_\infty}{(q^8,q^{12};q^{20})_\infty},\label{(2,4)-8}  \\
\sum_{i,j\geq 0}\frac{(-1)^iq^{i^2+3j^2+2ij}}{(q^2;q^2)_i(q^4;q^4)_j}&=\frac{(q,q^9,q^{10},q^{11},q^{19};q^{20})_\infty}{(q^4,q^{16};q^{20})_\infty},\label{(2,4)-7}   \\
\sum_{i,j\geq 0}\frac{(-1)^iq^{i^2+3j^2+2ij-2j}}{(q^2;q^2)_i(q^4;q^4)_j}&=\frac{(q^3,q^7,q^{10},q^{13},q^{17};q^{20})_\infty}{(q^8,q^{12};q^{20})_\infty},\label{(2,4)-19} \\
\sum_{i,j\geq 0}\frac{(-1)^iq^{i^2+3j^2+2ij-2i-4j}}{(q^2;q^2)_i(q^4;q^4)_j}&=\frac{(q,q^9,q^{10},q^{11},q^{19};q^{20})_\infty}{(q^4,q^{16};q^{20})_\infty}.\label{(2,4)-18}
\end{align}
\end{theorem}
\begin{proof}
We define
\begin{align}
F(u,v)=F(u,v;q):=\sum_{i,j\geq 0} \frac{(-1)^iu^iv^jq^{i^2+3j^2+2ij-i-3j}}{(q^2;q^2)_i(q^4;q^4)_j}.\nonumber
\end{align}
By \eqref{Euler-1}, \eqref{Euler-2} and \eqref{Jacobi} we have
\begin{align}
F(u,v)&=\CT \left[ \sum_{i\geq 0}\frac{u^i z^i}{(q^2;q^2)_i} \sum_{j\geq 0}\frac{(-1)^jv^j z^jq^{2j^2-2j}}{(q^4;q^4)_j}  \sum_{k=-\infty}^\infty (-1)^kq^{k^2-k}z^{-k} \right]  \nonumber \\
&=\CT \left[ \frac{(vz;q^4)_\infty}{(uz;q^2)_\infty}\sum_{k=-\infty}^\infty (-1)^kq^{k^2-k}z^{-k} \right]. \nonumber
\end{align}
Setting $(u,v)=(q^{3},q^{5})$, we have
\begin{align}
&F(q^{3},q^{5})=\CT \left[ \frac{1}{(zq^{3};q^4)_\infty}\sum_{k=-\infty}^\infty (-1)^kq^{k^2-k}z^{-k} \right] \nonumber \\
&=\CT \left[ \sum_{n=0}^\infty \frac{ z^n q^{3n}}{(q^4;q^4)_n} \sum_{k=-\infty}^\infty (-1)^kq^{k^2-k}z^{-k} \right] =\sum_{n=0}^\infty \frac{(-1)^nq^{n^2+2n}}{(q^4;q^4)_n}.\nonumber
\end{align}
Now by \eqref{Slater16} we obtain \eqref{(2,4)-8}.

Setting $(u,v)=(q,q^3)$ and using \eqref{Slater20} we obtain \eqref{(2,4)-7}. Setting $(u,v)=(q,q)$ and using \eqref{Slater16} we obtain \eqref{(2,4)-19}. Setting $(u,v)=(q^{-1},q^{-1})$ and using \eqref{Slater20} we obtain  \eqref{(2,4)-18}.
\end{proof}

Summing over $i$ using \eqref{Euler-2} first, the first three identities can also be proved directly by using \eqref{Ramanujan [10,Entry 11.2.7]}, \eqref{Slater19} and \eqref{Slater15}, respectively. We omit the details.
As for the last identity, similar process reduces it to a new identity which we state as a corollary.
\begin{corollary}
We have
\begin{align}
    \sum_{j\geq 0}\frac{(-1)^jq^{3j^2-4j}}{(q^4;q^4)_j(-q;q^2)_{j-1}}&=\frac{(q^2,q^3,q^5;q^5)_\infty}{(q^2;q^2)_\infty}.
\end{align}
\end{corollary}

\begin{theorem}\label{thm-10}
	We have
	\begin{align}
		\sum_{i,j\geq 0}\frac{q^{i^2+2j^2-2ij-i}a^j}{(q^2;q^2)_i(q^4;q^4)_j}&={(-1,-a;q^2)_\infty},\label{-(2,4)-4}  \\
		\sum_{i,j\geq 0}\frac{q^{i^2+2j^2-2ij}}{(q^2;q^2)_i(q^4;q^4)_j}&=\frac{{(q^2;q^2)_\infty^3}{(q^3;q^3)_\infty^2}}{{(q;q)_\infty^2}{(q^4;q^4)_\infty^2}{(q^6;q^6)_\infty}},\label{-(2,4)-5}  \\
		\sum_{i,j\geq 0}\frac{q^{i^2+2j^2-2ij+2j}}{(q^2;q^2)_i(q^4;q^4)_j}&=\frac{{(q^2;q^2)_\infty^2}{(q^6;q^6)_\infty^2}}{{(q;q)_\infty}{(q^3;q^3)_\infty}{(q^4;q^4)_\infty^2}},\label{-(2,4)-6}  \\
		\sum_{i,j\geq 0}\frac{(-1)^jq^{i^2+2j^2-2ij}}{(q^2;q^2)_i(q^4;q^4)_j}&=(q^2;q^4)_\infty^2.\label{-(2,4)-9}
	\end{align}
\end{theorem}
\begin{proof}
Summing over $i$ using \eqref{Euler-2} first, the identities can be proved by using \eqref{Euler-1}, \eqref{Euler-2}, \eqref{Slater25}, \eqref{Ramanujan [10,Entry 5.3.7]} and \eqref{Slater4}, respectively.  We omit the details.
\end{proof}

\begin{theorem}\label{thm-11}
We have
\begin{align}
\sum_{i,j\geq 0}\frac{q^{3i^2+4j^2+4ij-2i}}{(q^4;q^4)_i(q^8;q^8)_j}&=\frac{1}{(q,q^4;q^{5})_\infty (-q^2;q^2)_\infty},\label{(4,8)-1}  \\
\sum_{i,j\geq 0}\frac{q^{3i^2+4j^2+4ij-4i+4j+1}}{(q^4;q^4)_i(q^8;q^8)_j}&=\frac{1}{(q,q^4;q^5)_\infty (-q^2;q^2)_\infty},\label{(4,8)-2}  \\
\sum_{i,j\geq 0}\frac{q^{3i^2+4j^2+4ij+4j}}{(q^4;q^4)_i(q^8;q^8)_j}&=\frac{1}{(q^2,q^3;q^5)_\infty(-q^2;q^2)_\infty}. \label{(4,8)-3}
\end{align}
\end{theorem}
\begin{proof}
We define
\begin{align}
F(u,v)=F(u,v;q):=\sum_{i,j\geq 0} \frac{u^iv^jq^{3i^2+4j^2+4ij}}{(q^4;q^4)_i(q^8;q^8)_j}.\nonumber
\end{align}
We have
\begin{align*}
	F(u,v)&=\sum_{i\geq 0}\frac{ u^iq^{3i^2}}{(q^4;q^4)_i}\sum_{j\geq 0}\frac{v^jq^{4j^2+4ij}}{(q^8;q^8)_j}=\sum_{i\geq 0}\frac{u^i q^{3i^2}}{(q^4;q^4)_i}(-vq^{4i+4};q^8)_\infty.  \nonumber
\end{align*}
Setting $(u,v)=(q^{-2},1)$, we have
\begin{align}
F(q^{-2},1)=\sum_{i\geq 0}\frac{q^{3i^2-2i}}{(q^4;q^4)_i}(-q^{4i+4};q^8)_\infty=F_1(q)+F_2(q), \label{5yue28-0}
\end{align}
where
\begin{align}
    F_1(q)&=\sum_{\begin{smallmatrix}i\geq 0 \\ i ~~\text{odd} \end{smallmatrix}}\frac{q^{3i^2-2i}}{(q^4;q^4)_i}(-q^{4i+4};q^8)_\infty=(-q^8;q^8)_\infty\sum_{i\geq 0} \frac{q^{12i^2+8i+1}}{(q^4;q^4)_{2i+1}(-q^8;q^8)_i},\label{5yue28-1}\\
    F_2(q)&=\sum_{\begin{smallmatrix}i\geq 0 \\ i ~~\text{even} \end{smallmatrix}}\frac{q^{3i^2-2i}}{(q^4;q^4)_i}(-q^{4i+4};q^8)_\infty=(-q^4;q^8)_\infty\sum_{i\geq 0} \frac{q^{12i^2-4i}}{(q^4;q^4)_{2i}(-q^4;q^8)_i}.\label{5yue28-2}
\end{align}
Replacing $q$ by $-q^4$ in \eqref{Ramanujan [10,Entry 11.2.7]}, we can get the product representation of $F_1(q)$.  Replacing $q$ by $q^8$ in \eqref{Slater46}, we get the product representation of $F_2(q)$. Substituting these representations into \eqref{5yue28-0}, we prove \eqref{(4,8)-1}.

Similarly, setting $(u,v)=(q^{-4},q^4)$, we use \eqref{(Rogers [37, p. 330 (2), line 1])} and \eqref{Slater15} to  prove \eqref{(4,8)-2}.  Setting $(u,v)=(1,q^4)$, we use \eqref{Slater44} and \eqref{Slater19} to prove \eqref{(4,8)-3}.
\end{proof}

\begin{theorem}
We have
\begin{align}
\sum_{i,j\geq 0}\frac{(-1)^iq^{i^2+4j^2-4ij+3i+2j}}{(q^4;q^4)_i(q^8;q^8)_j}&={(q^4;q^8)_\infty},\label{(4,8)-7}  \\
\sum_{i,j\geq 0}\frac{(-1)^jq^{i^2+4j^2-4ij+4i}}{(q^4;q^4)_i(q^8;q^8)_j}&={(q^4,q^8,q^{12};q^{16})_\infty},\label{(4,8)-8}  \\
\sum_{i,j\geq 0}\frac{(-1)^jq^{i^2+4j^2-4ij+6i-4j}}{(q^4;q^4)_i(q^8;q^8)_j}&={-q^3}{(q^4,q^8,q^{12};q^{16})_\infty}.\label{(4,8)-9}
\end{align}
\end{theorem}
\begin{proof}
Let
\begin{align}
    F(u,v)=F(u,v;q):=\sum_{i,j\geq 0} \frac{(-1)^iu^iv^jq^{i^2+4j^2-4ij-i+2j}}{(q^4;q^4)_i(q^8;q^8)_j}=\sum_{i,j\geq 0} \frac{(-1)^iu^iv^jq^{(i-2j)(i-2j-1)}}{(q^4;q^4)_i(q^8;q^8)_j}.
\end{align}
We have
\begin{align}
&F(u,v)=\CT \left[\sum_{i\geq 0} \frac{u^iz^i}{(q^4;q^4)_i} \sum_{j\geq 0} \frac{v^jz^{-2j}}{(q^8;q^8)_j} \sum_{k=-\infty}^\infty (-1)^kz^{-k}q^{k^2-k} \right] \nonumber \\
&=\mathrm{CT}\left[\frac{(z^{-1},q^2z,q^2;q^2)_\infty}{(uz;q^4)_\infty (vz^{-2};q^8)_\infty} \right].
\end{align}
Let $(u,v)=(q^4,1)$. We have
\begin{align*}    &F(q^4,1)=\CT\left[\frac{(z^{-1},q^2z,q^2;q^2)_\infty}{(q^4z;q^4)_\infty (z^{-2};q^8)_\infty} \right] \nonumber \\
    &=(q^2; q^2)_\infty \CT\left[\frac{(z^{-1};q^2)_\infty (q^2z;q^2)_\infty}{(z^{-1};q^4)_\infty (-z^{-1};q^4)_\infty (q^4z;q^4)_\infty} \right]\nonumber \\
    &=(q^2;q^2)_\infty \CT\left[\frac{(z^{-1}q^2;q^4)_\infty (q^2z;q^4)_\infty}{(-z^{-1};q^4)_\infty} \right] \nonumber \\
    &=(q^2;q^2)_\infty \CT \left[\sum_{n\geq 0} \frac{(-1)^nq^{2n^2}z^n}{(q^4;q^4)_n} \sum_{k=0}^\infty \frac{(-q^2;q^4)_k}{(q^4;q^4)_k}(-z^{-1})^k \right]  \quad (\text{by \eqref{Euler-2} and \eqref{eq-qbinomial}})\nonumber \\
    &=(q^2;q^2)_\infty \sum_{n=0}^\infty \frac{q^{2n^2}(-q^2;q^4)_n}{(q^4;q^4)_n^2}.
\end{align*}
Now using \eqref{Ramanujan[10,Entry 4.2.6]} we obtain \eqref{(4,8)-7}.

Next, we write the 2-dissection formula:
\begin{align}\label{24-8-split}
    \sum_{i,j\geq 0} \frac{(-1)^jq^{i^2+4j^2-4ij+4i}}{(q^4;q^4)_i(q^8;q^8)_j}=S_0(q^2)+qS_1(q^2),
\end{align}
where
\begin{align}
    S_0(q^2)&=\sum_{i,j\geq 0} \frac{(-1)^jq^{4i^2+4j^2-8ij+8i}}{(q^4;q^4)_{2i}(q^8;q^8)_j},\\
    S_1(q^2)&=\sum_{i,j\geq 0} \frac{(-1)^jq^{4i^2+12i+4-8ij+4j^2-4j}}{(q^4;q^4)_{2i+1}(q^8;q^8)_j}.
\end{align}
We have
\begin{align}
&S_0(q^{\frac{1}{2}})=\sum_{i\geq 0} \frac{q^{i^2+2i}}{(q;q)_{2i}} \sum_{j\geq 0} \frac{(-1)^jq^{j^2-j}\cdot q^{(1-2i)j}}{(q^2;q^2)_j} =\sum_{i\geq 0} \frac{q^{i^2+2i}}{(q;q)_{2i}}(q^{1-2i};q^2)_\infty \nonumber \\
&=(q;q^2)_\infty \sum_{i\geq 0} \frac{(-1)^i q^{2i}(q;q^2)_i}{(q;q)_{2i}} =(q;q^2)_\infty \sum_{i\geq 0} \frac{(-1)^i q^{2i}}{(q^2;q^2)_i} =\frac{(q;q)_\infty}{(q^4;q^4)_\infty}.   \label{S0-result}
\end{align}
Similarly,
\begin{align}
&S_1(q^{\frac{1}{2}})=q\sum_{i\geq 0} \frac{(-1)^jq^{i^2+3i+j^2-2ij-j}}{(q;q)_{2i+1}(q^2;q^2)_j} =q\sum_{i\geq 0} \frac{q^{i^2+3i}}{(q;q)_{2i+1}} \sum_{j\geq 0} \frac{(-1)^jq^{j^2-j}\cdot q^{-2ij}}{(q^2;q^2)_j} \nonumber \\
&=q\sum_{i\geq 0} \frac{q^{i^2+3i}}{(q;q)_{2i+1}} (q^{-2i};q^2)_\infty =0. \label{S1-result}
\end{align}
Substituting \eqref{S0-result} and \eqref{S1-result} into \eqref{24-8-split}, we prove \eqref{(4,8)-8}.

In the same way, we write the 2-dissection formula
\begin{align}\label{24-9-split}
    \sum_{i,j\geq 0}
    \frac{(-1)^j q^{i^2+4j^2-4ij+6i-4j}}{(q^4;q^4)_{i}(q^8;q^8)_j}=T_0(q^2)+qT_1(q^2),
\end{align}
where
\begin{align}
    T_0(q^2)&=\sum_{i,j\geq 0} \frac{(-1)^jq^{4i^2+4j^2-8ij+12i-4j}}{(q^4;q^4)_{2i}(q^8;q^8)_j}, \\
    T_1(q^2)&=\sum_{i,j\geq 0} \frac{(-1)^jq^{4i^2+16i+4j^2-8ij-8j+6}}{(q^4;q^4)_{2i+1}(q^8;q^8)_j}.
\end{align}
We have
\begin{align}
    T_0(q^{\frac{1}{2}})&=\sum_{i\geq 0} \frac{q^{i^2+3i}}{(q;q)_{2i}} \sum_{j\geq 0} \frac{(-1)^jq^{j^2-j-2ij}}{(q^2;q^2)_j}=\sum_{i\geq 0}\frac{q^{i^2+3i}}{(q;q)_{2i}} (q^{-2i};q^2)_\infty=0. \label{T0-result}
\end{align}
Similarly,
\begin{align}
  &  T_1(q^{\frac{1}{2}})=q^{\frac{1}{2}}\sum_{i,j\geq 0} \frac{(-1)^jq^{i^2+4i+j^2-2ij-2j+1}}{(q;q)_{2i+1}(q^2;q^2)_j} =q^{\frac{1}{2}}\sum_{i\geq 0} \frac{q^{i^2+4i+1}}{(q;q)_{2i+1}} \sum_{j\geq 0} \frac{(-1)^jq^{j^2-j}q^{-(2i+1)j}}{(q^2;q^2)_j} \nonumber \\
    &=q^{\frac{1}{2}}\sum_{i\geq 0} \frac{q^{i^2+4i+1}}{(q;q)_{2i+1}} (q^{-2i-1};q^2)_\infty=-q^{\frac{1}{2}}(q;q^2)_\infty \sum_{i\geq 0} \frac{(-1)^{i}q^{2i}}{(q^2;q^2)_i} =-q^{\frac{1}{2}}\frac{(q;q)_\infty}{(q^4;q^4)_\infty}. \label{T1-result}
\end{align}
Substituting \eqref{T0-result} and \eqref{T1-result} into \eqref{24-9-split}, we obtain \eqref{(4,8)-9}.
\end{proof}
\begin{rem}
For the left side of \eqref{(4,8)-7}, if we sum over $j$ using \eqref{Euler-2} first and then separate the sums according to the parity of $i$, we obtain (after replacing $q$ by $q^{1/2}$)
\begin{align}
(-q^3;q^4)_\infty\sum_{n\geq0}\frac{q^{4n}(-q;q^4)_n}{(q^2;q^2)_{2n}}-(-q;q^4)_\infty\sum_{n\geq0}\frac{q^{4n+2}(-q^3;q^4)_n}{(q^2;q^2)_{2n+1}}&=(q^2;q^4)_\infty. \nonumber
\end{align}
If one can prove the above identity directly, then we have a new proof for \eqref{(4,8)-7}.
\end{rem}

\subsection{Identities of index $(1,4)$}

\begin{theorem}\label{thm-(1,4)}
	We have
	\begin{align}
		\sum_{i,j\geq 0}\frac{q^{3i^2/2+4j^2+4ij-i/2+2j}}{(q;q)_i(q^4;q^4)_j}&=\frac{1}{(q,q^5,q^6;q^8)_\infty}.\label{(1,4)-1}
\end{align}
\end{theorem}
This identity gives a companion to the identities (21)--(24) in \cite{Kursungoz}.
\begin{proof}
We define
\begin{align}
F(u,v)=F(u,v;q):=\sum_{i,j\geq 0} \frac{u^iv^jq^{3i^2/2+4j^2+4ij-3i/2-2j}}{(q;q)_i(q^4;q^4)_j}.\nonumber
\end{align}
By \eqref{Euler-1}, \eqref{Euler-2} and \eqref{Jacobi} we have
\begin{align}
F(u,v)&=\CT \left[ \sum_{i\geq 0}\frac{(-1)^iu^i z^i q^{i^2/2-i/2}}{(q;q)_i} \sum_{j\geq 0}\frac{v^j z^{2j}}{(q^4;q^4)_j}  \sum_{k=-\infty}^\infty (-1)^kq^{k^2-k}z^{-k} \right]  \nonumber \\
&=\CT \left[ \frac{(uz;q)_\infty}{(vz^2;q^4)_\infty}\sum_{k=-\infty}^\infty (-1)^kq^{k^2-k}z^{-k} \right]. \nonumber
\end{align}
Setting $(u,v)=(q,q^4)$, we have
\begin{align}
&F(q,q^4)=\CT \left[ \frac{(zq;q^2)_\infty}{{(-zq^2;q^2)_\infty}}\sum_{k=-\infty}^\infty (-1)^kq^{k^2-k}z^{-k} \right] \nonumber \\
&=\CT \left[ \sum_{n=0}^\infty \frac{(-1/q;q^2)_n}{{(q^2;q^2)_n}}(-zq^2)^n \sum_{k=-\infty}^\infty (-1)^kq^{k^2-k}z^{-k} \right] =\sum_{n=0}^\infty \frac{(-1/q;q^2)_n}{{(q^2;q^2)_n}}q^{n^2+n} .\nonumber
\end{align}
By \eqref{Gollnitz-(2.22)} we get \eqref{(1,4)-1}.
\end{proof}

\begin{theorem}\label{thm-15}
We have
\begin{align}
\sum_{i,j\geq 0}\frac{q^{i^2+6j^2-4ij-i+2j}}{(q^2;q^2)_i(q^8;q^8)_j}&=\frac{2}{(q^2;q^4)_\infty(q^2,q^8,q^{14};q^{16})_\infty}, \label{(2,8)-1}\\
\sum_{i,j\geq 0}\frac{q^{i^2+6j^2-4ij-i+4j}}{(q^2;q^2)_i(q^8;q^8)_j}&=\frac{2}{(q^2;q^4)_\infty(q^4,q^6,q^{14};q^{16})_\infty}, \label{(2,8)-2}\\
\sum_{i,j\geq 0}\frac{q^{i^2+6j^2-4ij-i+6j}}{(q^2;q^2)_i(q^8;q^8)_j}&=\frac{2}{(q^2;q^4)_\infty(q^6,q^8,q^{10};q^{16})_\infty}. \label{(2,8)-3}
\end{align}
\end{theorem}
\begin{proof}
Summing over $i$ using \eqref{Euler-2} first, the identities can be proved using \eqref{Slater36}, \eqref{Gollnitz-(2.24)} and \eqref{Slater34}, respectively. We omit the details.
\end{proof}

\section{Identities involving triple sums}\label{sec-triple}
In this section, we present some new Rogers--Ramanujan type identities involving triple sums. They are of indices $(1,1,1)$, $(1,1,2)$, $(1,1,3)$, $(1,2,2)$ and $(1,2,4)$ and serve as new companions to some identities in the literature.

\subsection{Identities of index $(1,1,1)$}
We find some new companions to \cite[Theorem 4.1]{Cao-Wang} (see \eqref{(2,2,2)-1} below).
\begin{theorem}\label{thm-7}
	We have
	\begin{align}
 \sum_{i,j,k\geq 0} \frac{q^{i^2+2j^2+k^2+2ij+2ik+2jk}a^{i+j}b^{k-i}}{(q^2;q^2)_i(q^2;q^2)_j(q^2;q^2)_k} &={(-aq/b,-bq;q^2)_\infty}, \quad \text{(\cite{Lovejoy})} \label{(2,2,2)-1}\\
\sum_{i,j,k\geq 0} \frac{{(-1)^k}q^{i^2+2j^2+k^2+2ij+2ik+2jk}}{(q^2;q^2)_i(q^2;q^2)_j(q^2;q^2)_k}&=\frac{(q^6,q^6,q^{12};q^{12})_\infty}{(q^2;q^2)_\infty}, \quad \text{(\cite[Eq.\ (5.25)]{Sills2003})} \label{(2,2,2)-2} \\
\sum_{i,j,k\geq 0} \frac{q^{i^2+2j^2+k^2+2ij+2ik+2jk+i+2j+2k}}{(q^2;q^2)_i(q^2;q^2)_j(q^2;q^2)_k}&=\frac{(q,q^5,q^6;q^6)_\infty}{(q;q)_\infty}, \quad \text{(\cite[Eq.\ (5.120)]{Sills2003})} \label{(2,2,2)-3} \\
\sum_{i,j,k\geq 0} \frac{q^{i^2+2j^2+k^2+2ij+2ik+2jk+i}}{(q^2;q^2)_i(q^2;q^2)_j(q^2;q^2)_k}&=\frac{(q^3,q^3,q^6;q^6)_\infty}{(q;q)_\infty}, \label{(2,2,2)-4} \\
\sum_{i,j,k\geq 0} \frac{{(-1)^{i}}q^{i^2+2j^2+k^2+2ij+2ik+2jk+i}a^{i+j}b^k}{(q^2;q^2)_i(q^2;q^2)_j(q^2;q^2)_k}&={(-bq;q^2)_\infty}, \label{(2,2,2)-5}\\
\sum_{i,j,k\geq 0} \frac{q^{i^2+2j^2+k^2+2ij+2ik+2jk-i-k}}{(q^2;q^2)_i(q^2;q^2)_j(q^2;q^2)_k}&=\frac{3}{{(q^2;q^4)^2_\infty}}, \label{(2,2,2)-6} \\
\sum_{i,j,k\geq 0} \frac{(-1)^iq^{i^2+2j^2+k^2+2ij+2ik+2jk+i-2j-k}}{(q^2;q^2)_i(q^2;q^2)_j(q^2;q^2)_k}&=\frac{3}{{(q^2;q^4)_\infty}}, \label{(2,2,2)-8} \\
\sum_{i,j,k\geq 0} \frac{(-1)^iq^{i^2+2j^2+k^2+2ij+2ik+2jk+2i+2j+k}}{(q^2;q^2)_i(q^2;q^2)_j(q^2;q^2)_k}&=\frac{(q^3,q^9,q^{12};q^{12})_\infty}{{(q^2;q^2)_\infty}}, \label{(2,2,2)-10} \\
\sum_{i,j,k\geq 0} \frac{(-1)^iq^{i^2+2j^2+k^2+2ij+2ik+2jk+2i+2j+2k}}{(q^2;q^2)_i(q^2;q^2)_j(q^2;q^2)_k}&=\frac{1}{{(q^4,q^6,q^{8};q^{12})_\infty}}. \label{(2,2,2)-11}
\end{align}
\end{theorem}
\begin{proof}
We define
\begin{align}\label{111-F-defn}
F(u,v,w)=F(u,v,w;q^2):=\sum_{i,j,k\geq 0} \frac{u^iv^jw^kq^{i^2+2j^2+k^2+2ij+2ik+2jk}}{(q^2;q^2)_i(q^2;q^2)_j(q^2;q^2)_k}.
\end{align}
Arguing similarly as in \cite[p.\ 76]{Sills2003}, we have
\begin{align}
&F(u,v,w;q^2)=\sum_{j,k\geq 0}\frac{v^jw^k q^{2j^2+k^2+2jk}}{(q^2;q^2)_j(q^2;q^2)_k}\sum_{i\geq 0}\frac{u^i q^{i^2+2ij+2ik}}{(q^2;q^2)_i} \nonumber \\
&=\sum_{j,k\geq 0}\frac{v^jw^k q^{2j^2+k^2+2jk}}{(q^2;q^2)_j(q^2;q^2)_k}{(-uq^{2j+2k+1};q^2)_\infty}\nonumber \\
&={(-uq;q^2)_\infty}\sum_{j,k\geq 0}\frac{v^jw^k q^{2j^2+k^2+2jk}}{(q^2;q^2)_j(q^2;q^2)_k(-uq;q^2)_{j+k}}\nonumber \\
&={(-uq;q^2)_\infty}\sum_{j=0}^\infty\sum_{n=j}^\infty\frac{w^n{(v/w)}^j q^{n^2+j^2}}{(q^2;q^2)_j(q^2;q^2)_{n-j}(-uq;q^2)_n}\nonumber \\
&={(-uq;q^2)_\infty}\sum_{n=0}^\infty\frac{w^nq^{n^2}}{(-uq;q^2)_n}\sum_{j=0}^n\frac{{(v/w)}^j q^{j^2}}{(q^2;q^2)_j(q^2;q^2)_{n-j}}\nonumber \\
&={(-uq;q^2)_\infty}\sum_{n=0}^\infty\frac{w^nq^{n^2}}{(-uq;q^2)_n(q^2;q^2)_n}\sum_{j=0}^n (v/w)^j q^{j^2} {n\brack j}_{q^2}\nonumber \\
&={(-uq;q^2)_\infty}\sum_{n=0}^\infty\frac{w^nq^{n^2}}{(-uq;q^2)_n(q^2;q^2)_n}(-vq/w;q^2)_n.\label{111-F-start}
\end{align}	
 Here for the last second equality we used \eqref{finite-Jacobi}.
It follows from \eqref{111-F-start} and \eqref{Euler-2} that
\begin{align}
    F(a/b,a,b)=(-aq/b;q^2)_\infty(-bq;q^2)_\infty,
\end{align}
and hence \eqref{(2,2,2)-1} holds.

Setting $(u,v,w)=(1,1,-1)$, by \eqref{111-F-start} we have
\begin{align*}
F(1,1,-1)&={(-q;q^2)_\infty}\sum_{n=0}^\infty\frac{(-1)^nq^{n^2}(q;q^2)_n}{(-q;q^2)_n(q^2;q^2)_n}.
\end{align*}
Using \eqref{Slater29} with $q$ replaced by $-q$, we obtain \eqref{(2,2,2)-2}.

Setting $(u,v,w)=(q,q^2,q^2)$, by \eqref{111-F-start} we have
\begin{align*}
F(q,q^2,q^2)&={(-q^2;q^2)_\infty}\sum_{n=0}^\infty\frac{q^{n^2+2n}(-q;q^2)_n}{(-q^2;q^2)_n(q^2;q^2)_n}.
\end{align*}
Using \eqref{Wang-(1.22)} with $q$ replaced by $-q$, we obtain
 \eqref{(2,2,2)-3}.

 Setting $(u,v,w)=(q,1,1)$, by \eqref{111-F-start} we have
\begin{align*}
F(q,1,1)&={(-q^2;q^2)_\infty}\sum_{n=0}^\infty\frac{q^{n^2}(-q;q^2)_n}{(-q^2;q^2)_n(q^2;q^2)_n}.
\end{align*}
Using \eqref{Slater25} we obtain  \eqref{(2,2,2)-4}.

We have
\begin{align}
&F(-aq,a,b)=(aq^2;q^2)_\infty \sum_{n=0}^\infty \frac{b^nq^{n^2}(-aq/b;q^2)_n}{(q^2;q^2)_n(aq^2;q^2)_n} \nonumber \\
&=(aq^2;q^2)_\infty \lim\limits_{c\rightarrow 0} {}_2\phi_{1}\bigg(\genfrac{}{}{0pt}{} {-aq/b,-1/c}{aq^2};q^2,bcq  \bigg) \nonumber \\
&=(aq^2;q^2)_\infty \lim\limits_{c\rightarrow 0} \frac{(-acq^2,-bq;q^2)_\infty}{(aq^2,bcq;q^2)_\infty}=(-bq;q^2)_\infty.
\end{align}
Here for the third equality we used \eqref{eq-Gauss}. This proves \eqref{(2,2,2)-5}.

Setting $(u,v,w)=(q^{-1},1,q^{-1})$,  by \eqref{111-F-start} we have
\begin{align*}
&F(q^{-1},1,q^{-1})={(-1;q^2)_\infty}\sum_{n=0}^\infty\frac{q^{n^2-n}(-q^2;q^2)_n}{(-1;q^2)_n(q^2;q^2)_n}={(-q^2;q^2)_\infty}\sum_{n=0}^\infty\frac{q^{n^2-n}(1+q^{2n})}{(q^2;q^2)_n} \nonumber \\
&=(-q^2;q^2)_\infty((-1;q^2)_\infty+(-q^2;q^2)_\infty)=3(-q^2;q^2)_\infty^2.
\end{align*}
This proves \eqref{(2,2,2)-6}.

Setting $(u,v,w)=(-q,q^{-2},q^{-1})$,  by \eqref{111-F-start} we have
\begin{align*}
&F(-q,q^{-2},q^{-1})={(q^2;q^2)_\infty}\sum_{n=0}^\infty\frac{q^{n^2-n}(-1;q^2)_n}{(q^2;q^2)_n(q^2;q^2)_n} \nonumber \\
&={(-1;q^2)_\infty}\sum_{n=0}^\infty\frac{(-1)^nq^{n^2+n}(q^{-2};q^2)_n}{(-1;q^2)_n(q^2;q^2)_n}  \quad \text{(by \eqref{(Laughlin-(6.1.3))})}          \nonumber \\
&={(-1;q^2)_\infty}\Big(1+\frac{(1-q^{-2})(-q^2)}{2(1-q^2)}\Big)=3(-q^2;q^2)_\infty^2. \nonumber
\end{align*}
This proves \eqref{(2,2,2)-8}.

Setting $(u,v,w)=(-q^2,q^2,q)$,  by \eqref{111-F-start} we have
\begin{align*}
F(-q^2,q^2,q)&={(q^3;q^2)_\infty}\sum_{n=0}^\infty\frac{q^{n^2+n}(-q^2;q^2)_n}{(q^3;q^2)_n(q^2;q^2)_n}={(q;q^2)_\infty}\sum_{n=0}^\infty\frac{q^{n^2+n}(-q^2;q^2)_n}{(q;q)_{2n+1}}.       \nonumber
\end{align*}
Using \eqref{Slater28} we prove \eqref{(2,2,2)-10}.

Setting $(u,v,w)=(-q^2,q^{2},q^{2})$,  by \eqref{111-F-start} we have
\begin{align*}
F(-q^2,q^{2},q^{2})&={(q^3;q^2)_\infty}\sum_{n=0}^\infty\frac{q^{n^2+2n}(-q;q^2)_n}{(q^3;q^2)_n(q^2;q^2)_n}={(q;q^2)_\infty}\sum_{n=0}^\infty\frac{q^{n^2+2n}(-q;q^2)_n}{(q;q)_{2n+1}}.   \nonumber
\end{align*}
Using \eqref{Slater50} we prove \eqref{(2,2,2)-11}.
\end{proof}
\begin{rem}
The identity \eqref{(2,2,2)-1} appeared earlier in the work of Lovejoy \cite{Lovejoy}, and then was rediscovered by Cao and Wang \cite[Theorem 4.1]{Cao-Wang}. Our proof is different from the proofs given there.
\end{rem}

\subsection{Identities of index $(1,1,3)$} We find the following new companion to \cite[Theorem 4.4]{Cao-Wang}.
\begin{theorem}\label{thm-(1,1,3)}
We have
\begin{align}
\sum_{i,j,k\geq 0}\frac{(-1)^kq^{i^2+j^2+9k^2/2+ij+3ik+3jk-i-3k/2}}{(q;q)_i(q;q)_j(q^3;q^3)_k}&=\frac{2}{(q,q^{2};q^3)_\infty}. \label{(1,1,3)-6}
\end{align}
\end{theorem}
\begin{proof}
We define
\begin{align}
F(u,v,w)=F(u,v,w;q):=\sum_{i,j,k\geq 0} \frac{u^iv^jw^k(-1)^kq^{i^2+j^2+9k^2/2+ij+3ik+3jk-i-j-3k/2}}{(q;q)_i(q;q)_j(q^3;q^3)_k}.
\end{align}
By \eqref{Euler-1}, \eqref{Euler-2} and \eqref{Jacobi} we have
\begin{align}
&F(u,v,w)=\CT \sum_{i\geq 0}\frac{(-1)^iu^i z^i q^{(i^2-i)/2}}{(q;q)_i} \sum_{j\geq 0}\frac{(-1)^jv^j z^j q^{(j^2-j)/2}}{(q;q)_j} \sum_{k\geq 0}\frac{w^k z^{3k}}{(q^3;q^3)_k} \nonumber \\
&\quad \quad \quad \quad \quad \quad \sum_{n=-\infty}^\infty (-1)^nq^{(n^2-n)/2}z^{-n}   \nonumber \\
&=\CT \frac{(uz,vz,z^{-1},zq,q;q)_\infty}{(wz^3;q^3)_\infty}.
\end{align}
Setting $(u,v,w)=(1,q,1)$, we have
\begin{align}
&F(1,q,1)=\CT \frac{(zq,zq,z^{-1},q;q)_\infty}{(\zeta_3z,\zeta^2_3 z;q)_\infty}  \nonumber \\
&= \CT \left[\frac{(zq;q)_\infty}{(\zeta_3z;q)_\infty}\cdot\frac{1}{(\zeta^2_3 z;q)_\infty}\cdot(zq,z^{-1},q;q)_\infty\right] \nonumber \\
&=\CT \left[ \sum_{i\geq 0}\frac{ (\zeta^2_3q;q)_i(\zeta_3z)^i}{(q;q)_i} \sum_{j\geq 0}\frac{(\zeta^2_3z)^j }{(q;q)_j}  \sum_{n=-\infty}^\infty (-1)^nq^{(n^2-n)/2}z^{-n}\right] \nonumber \\
&=\sum_{i,j\geq 0}\frac{ (\zeta^2_3q;q)_i\zeta^i_3\zeta^{2j}_3(-1)^{i+j}q^{(i^2+2ij+j^2-i-j)/2}}{(q;q)_i(q;q)_j}  \nonumber \\
&=\sum_{i\geq 0}\frac{ (-1)^i\zeta^i_3 q^{(i^2-i)/2}(\zeta^2_3q;q)_i}{(q;q)_i} \sum_{j\geq 0}\frac{(-1)^j\zeta^{2j}_3 q^{(j^2+2ij-j)/2}}{(q;q)_j}  \nonumber \\
&=\sum_{i\geq 0}\frac{ (-1)^i\zeta^i_3 q^{(i^2-i)/2}(\zeta^2_3q;q)_i}{(q;q)_i}\cdot(\zeta^2_3q^{i};q)_\infty \nonumber \\
&=(\zeta^2_3q;q)_\infty\sum_{i\geq 0}\frac{ (-1)^i\zeta^i_3 q^{(i^2-i)/2}(1-\zeta^2_3q^{i})}{(q;q)_i} \nonumber \\
&=(\zeta^2_3q;q)_\infty \Big(\sum_{i\geq 0}\frac{ (-1)^i\zeta^i_3 q^{(i^2-i)/2}}{(q;q)_i}-\sum_{i\geq 0}\frac{ (-1)^i\zeta^i_3 \zeta^2_3q^{i}q^{(i^2-i)/2}}{(q;q)_i}\Big) \nonumber \\
&=(\zeta^2_3q;q)_\infty \big((\zeta_3;q)_\infty-\zeta^2_3(\zeta_3q;q)_\infty \big) \nonumber \\
&=(\zeta^2_3q;q)_\infty\cdot(1-\zeta_3-\zeta^2_3)\cdot(\zeta_3q;q)_\infty \nonumber \\
&=\frac{2(q,\zeta_3q,\zeta^2_3q;q)_\infty}{(q;q)_\infty}=\frac{2(q^3;q^3)_\infty}{(q;q)_\infty}=\frac{2}{(q,q^2;q^3)_\infty}.  \nonumber
\end{align}
 This proves \eqref{(1,1,3)-6}.
\end{proof}

\subsection{Identities of index $(1,2,2)$}
We find some new companions to two identities in \cite{Sills2003} (see \eqref{(1,2,2)-1} and \eqref{(1,2,2)-2} below).
\begin{theorem}\label{thm-6}
We have
\begin{align}
&\sum_{i,j,k\geq 0}\frac{q^{i^2/2+3j^2+2k^2+2ij+2ik+4jk+3i/2+3j+2k}}{(q;q)_i(q^2;q^2)_j(q^2;q^2)_k}=\frac{(q,q^7,q^8;q^8)_\infty}{(q;q)_\infty}, \quad \text{(\cite[Eq.\ (5.35)]{Sills2003})} \label{(1,2,2)-1} \\
&\sum_{i,j,k\geq 0}\frac{q^{i^2/2+3j^2+2k^2+2ij+2ik+4jk+i/2+j}}{(q;q)_i(q^2;q^2)_j(q^2;q^2)_k}=\frac{(q^3,q^5,q^8;q^8)_\infty}{(q;q)_\infty}, \quad \text{(\cite[Eq.\ (5.37)]{Sills2003})} \label{(1,2,2)-2} \\
&\sum_{i,j,k\geq 0}\frac{(-1)^{i+j}q^{i^2/2+3j^2+2k^2+2ij+2ik+4jk+3i/2+2j+k}}{(q;q)_i(q^2;q^2)_j(q^2;q^2)_k}={(q^2;q^2)_\infty},\label{(1,2,2)-4}\\
&\sum_{i,j,k\geq 0}\frac{(-1)^jq^{i^2/2+3j^2+2k^2+2ij+2ik+4jk+3i/2+2j+2k}}{(q;q)_i(q^2;q^2)_j(q^2;q^2)_k}=\frac{(q,q^5,q^6;q^6)_\infty}{(q;q)_\infty},\label{(1,2,2)-5} \\
&\sum_{i,j,k\geq 0}\frac{(-1)^jq^{i^2+6j^2+4k^2+4ij+4ik+8jk+i-4j-4k}}{(q^2;q^2)_i(q^4;q^4)_j(q^4;q^4)_k}= \frac{2}{(q^4,q^6,q^8;q^{12})_\infty}, \label{(1,2,2)-new-6} \\
&\sum_{i,j,k\geq 0}\frac{(-1)^jq^{i^2+6j^2+4k^2+4ij+4ik+8jk+i}}{(q^2;q^2)_i(q^4;q^4)_j(q^4;q^4)_k}= \frac{(q^6,q^6,q^{12};q^{12})_\infty}{(q^2;q^{2})_\infty}, \label{(1,2,2)-new-8} \\
&\sum_{i,j,k\geq 0}\frac{q^{i^2/2+3j^2+2k^2+2ij+2ik+4jk+i/2}a^{j+k}}{(q;q)_i(q^2;q^2)_j(q^2;q^2)_k}=(-q;q)_\infty (-aq^2;q^4)_\infty, \label{(1,2,2)-new-1} \\
&\sum_{i,j,k\geq 0}\frac{q^{i^2/2+3j^2+2k^2+2ij+2ik+4jk+3i/2+2j}a^{j+k}}{(q;q)_i(q^2;q^2)_j(q^2;q^2)_k}= (-q^2;q)_\infty (-aq^2;q^4)_\infty, \label{(1,2,2)-new-2} \\
&\sum_{i,j,k\geq 0}\frac{q^{i^2+6j^2+4k^2+4ij+4ik+8jk-i-2j}}{(q^2;q^2)_i(q^4;q^4)_j(q^4;q^4)_k}= \frac{2(q^6,q^{10},q^{16};q^{16})_\infty}{(q^2;q^2)_\infty}, \label{(1,2,2)-new-3} \\
&\sum_{i,j,k\geq 0}\frac{(-1)^jq^{i^2+6j^2+4k^2+4ij+4ik+8jk+2k}a^ib^{j+k}}{(q^2;q^2)_i(q^4;q^4)_j(q^4;q^4)_k}=(-aq;q^2)_\infty, \label{(1,2,2)-new-5} \\
 &\sum_{i,j,k\geq 0}\frac{(-1)^{i+j}q^{i^2+6j^2+4k^2+4ij+4ik+8jk-2i-6j-5k}a^{i+2j+k}}{(q^2;q^2)_i(q^4;q^4)_j(q^4;q^4)_k}= {(aq;q^4)_\infty}, \label{(1,2,2)-new-12} \\
&\sum_{i,j,k\geq 0}\frac{(-1)^{i+j}q^{i^2+6j^2+4k^2+4ij+4ik+8jk-i-2j-4k}}{(q^2;q^2)_i(q^4;q^4)_j(q^4;q^4)_k}= (q^6,q^{10},q^{16};q^{16})_\infty(q^4,q^{28};q^{32})_\infty, \label{(1,2,2)-new-15} \\
&\sum_{i,j,k\geq 0}\frac{(-1)^{i+j}q^{i^2+6j^2+4k^2+4ij+4ik+8jk-2j-k}a^{i+2j+k}}{(q^2;q^2)_i(q^4;q^4)_j(q^4;q^4)_k}= {(aq;q^4)_\infty}, \label{(1,2,2)-new-18}\\
&\sum_{i,j,k\geq 0}\frac{(-1)^{i+j}q^{i^2+6j^2+4k^2+4ij+4ik+8jk-i-4j-2k}a^{i+2j+2k}}{(q^2;q^2)_i(q^4;q^4)_j(q^4;q^4)_k}= {(a;q^2)_\infty}, \label{(1,2,2)-new-6.30}\\
&\sum_{i,j,k\geq 0}\frac{(-1)^{i+j}q^{i^2+6j^2+4k^2+4ij+4ik+8jk-i-6j}a^i}{(q^2;q^2)_i(q^4;q^4)_j(q^4;q^4)_k}= {(a;q^2)_\infty-(aq^4;q^2)_\infty}, \label{(1,2,2)-new-11} \\
 &\sum_{i,j,k\geq 0}\frac{(-1)^{i+j}q^{i^2+6j^2+4k^2+4ij+4ik+8jk+i-4j-4k}}{(q^2;q^2)_i(q^4;q^4)_j(q^4;q^4)_k}= {2(q^2;q^4)_\infty}. \label{(1,2,2)-new-19}
\end{align}
\end{theorem}
\begin{proof}
We define
\begin{align}
F(u,v,w)=F(u,v,w;q):=\sum_{i,j,k\geq 0} \frac{u^iv^jw^kq^{i^2/2+3j^2+2k^2+2ij+2ik+4jk}}{(q;q)_i(q^2;q^2)_j(q^2;q^2)_k}.
\end{align}
We have
\begin{align*}
&F(u,v,w)=\sum_{j,k\geq 0}\frac{v^jw^k q^{3j^2+2k^2+4jk}}{(q^2;q^2)_j(q^2;q^2)_k}\sum_{i\geq 0}\frac{u^i q^{i^2/2+2ij+2ik}}{(q;q)_i}  \nonumber \\
&=\sum_{j,k\geq 0}\frac{v^jw^k q^{3j^2+2k^2+4jk}}{(q^2;q^2)_j(q^2;q^2)_k}{(-uq^{2j+2k+1/2};q)_\infty}\nonumber \\
&={(-uq^{-1/2};q)_\infty}\sum_{j,k\geq 0}\frac{v^jw^k q^{3j^2+2k^2+4jk}}{(q^2;q^2)_j(q^2;q^2)_k(-uq^{-1/2};q)_{2j+2k+1}}\nonumber \\
&={(-uq^{-1/2};q)_\infty}\sum_{j,k\geq 0}\frac{v^jw^k q^{3j^2+2k^2+4jk}}{(q^2;q^2)_j(q^2;q^2)_k(-uq^{-1/2};q^2)_{j+k+1}(-uq^{1/2};q^2)_{j+k}}\nonumber \\
&={(-uq^{-1/2};q)_\infty}\sum_{n=0}^\infty\sum_{j=0}^n \frac{w^n{(v/w)}^j q^{2n^2+j^2}}{(q^2;q^2)_j(q^2;q^2)_{n-j}(-uq^{-1/2};q^2)_{n+1}(-uq^{1/2};q^2)_n}\nonumber \\
&={(-uq^{-1/2};q)_\infty}\sum_{n=0}^\infty\frac{w^nq^{2n^2}}{(-uq^{-1/2};q^2)_{n+1}(-uq^{1/2};q^2)_n}\sum_{j=0}^n\frac{{(v/w)}^j q^{j^2}}{(q^2;q^2)_j(q^2;q^2)_{n-j}}\nonumber \\
&={(-uq^{-1/2};q)_\infty}\sum_{n=0}^\infty\frac{w^nq^{2n^2}}{(-uq^{-1/2};q^2)_{n+1}(-uq^{1/2};q^2)_n(q^2;q^2)_n}\sum_{j=0}^n (v/w)^j q^{j^2}{n \brack j}_{q^2}\nonumber \\
&={(-uq^{-1/2};q)_\infty}\sum_{n=0}^\infty\frac{w^nq^{2n^2}}{(-uq^{-1/2};q^2)_{n+1}(-uq^{1/2};q^2)_n(q^2;q^2)_n}(-vq/w;q^2)_n.
\end{align*}

Setting $(u,v,w)=(q^{3/2},q^3,q^2)$ and using \eqref{Slater38} we obtain \eqref{(1,2,2)-1}.

Setting $(u,v,w)=(q^{1/2},q,1)$ and using \eqref{Slater39} we obtain \eqref{(1,2,2)-2}.

Setting $(u,v,w)=(-q^{3/2},-q^2,q)$ and using \eqref{Slater5} we obtain \eqref{(1,2,2)-4}.

Setting $(u,v,w)=(q^{3/2},-q^2,q^2)$ and using \eqref{Slater27} we obtain \eqref{(1,2,2)-5}.	

Setting $(u,v,w)=(q^{1/2},-q^{-2},q^{-2})$ and using \eqref{Slater27} we obtain \eqref{(1,2,2)-new-6}.

Setting $(u,v,w)=(q^{1/2},-1,1)$ and using \eqref{eq-BMS} we obtain \eqref{(1,2,2)-new-8}.

We have by \eqref{Euler-2} that
\begin{align}
 F(q^{1/2},a,a)&=(-q;q)_\infty \sum_{n=0}^\infty \frac{a^nq^{2n^2}}{(q^4;q^4)_n}=(-q;q)_\infty (-aq^2;q^4)_\infty, \\
    F(q^{3/2},aq^2,a)&=(-q;q)_\infty \sum_{n=0}^\infty \frac{a^nq^{2n^2}}{(1+q)(q^4;q^4)_n}=(-q^2;q)_\infty(-aq^2;q^4)_\infty.
\end{align}
This proves \eqref{(1,2,2)-new-1} and \eqref{(1,2,2)-new-2}.

We define
\begin{align}
E(u,v,w)=E(u,v,w;q):=\sum_{i,j,k\geq 0} \frac{(-1)^{i+j}u^iv^{2j}w^{2k}q^{i^2+6j^2+4k^2+4ij+4ik+8jk-i-4j-2k}}{(q^2;q^2)_i(q^4;q^4)_j(q^4;q^4)_k}.\end{align}
By \eqref{Euler-1}, \eqref{Euler-2} and \eqref{Jacobi} we have
\begin{align}
&E(u,v,w) \nonumber \\
&=\CT\left[ \sum_{i\geq 0}\frac{u^i z^i }{(q^2;q^2)_i} \sum_{j\geq 0}\frac{(-1)^jv^{2j} z^{2j}q^{2j^2-2j} }{(q^4;q^4)_j} \sum_{k\geq 0}\frac{  w^{2k}  z^{2k} }{(q^4;q^4)_k} \sum_{n=-\infty}^\infty (-1)^nq^{n^2-n}z^{-n} \right]  \nonumber \\
&=\CT \frac{(vz,-vz,z^{-1},zq^2,q^2;q^2)_\infty}{(uz,wz,-wz;q^2)_\infty}.
\end{align}

Setting $(u,v,w)=(-1,\zeta_4 q,q)$, we have
\begin{align}
&E(-1,\zeta_4 q,q)=\CT \frac{(\zeta_4 qz,-\zeta_4 qz,zq^2,z^{-1},q^2;q^2)_\infty}{(- z,qz,-qz;q^2)_\infty} \nonumber \\
&=\CT \frac{1}{(-z;q^2)_\infty} \cdot \frac{(-q^2z^2;q^4)_\infty}{(q^2z^2;q^4)_\infty} \cdot (z^{-1},q^2z,q^2;q^2)_\infty \label{05yue.10ri-1} \\
&=\CT \sum_{i\geq 0} \frac{(-z)^i}{(q^2;q^2)_i} \sum_{j\geq 0} \frac{(-1;q^4)_jq^{2j}z^{2j}}{(q^4;q^4)_j} \sum_{k=-\infty}^\infty (-1)^kq^{k^2-k}z^{-k} \nonumber \\
&=\sum_{i,j\geq 0} \frac{(-1;q^4)_jq^{(i+2j)^2-i}}{(q^2;q^2)_i(q^4;q^4)_j} =\sum_{j\geq 0} \frac{(-1;q^4)_jq^{4j^2}}{(q^4;q^4)_j} \sum_{i\geq 0} \frac{q^{i^2+(4j-1)i}}{(q^2;q^2)_i} \nonumber \\
&=2(-q^2;q^2)_\infty \sum_{j\geq 0} \frac{q^{4j^2}}{(-q^2;q^4)_j(q^4;q^4)_j}.
\end{align}
Substituting \eqref{Slater39} with $q$ replaced by $-q^2$ into it, we prove \eqref{(1,2,2)-new-3}.

Setting $(u,v,w)=(-aq,b^{1/2} q^2,b^{1/2} q^2)$, we have
\begin{align}
&E(-aq,b^{1/2} q^2,b^{1/2} q^2)=\CT \frac{(zq^2,z^{-1},q^2;q^2)_\infty}{(-aq z;q^2)_\infty}   \nonumber \\
&=\mathrm{CT}\left[\sum_{i\geq 0}\frac{(-aqz)^i}{(q^2;q^2)_i}\sum_{n=-\infty}^\infty (-1)^nq^{n^2-n}z^{-n}  \right] =\sum_{n\geq 0}\frac{a^nq^{n^2}}{(q^2;q^2)_n}=(-aq;q^2)_\infty.
\end{align}
This proves \eqref{(1,2,2)-new-5}.

Setting $(u,v,w)=(aq^{-1},aq^{-1},a^{1/2}q^{-3/2})$, we have
\begin{align}
&E(aq^{-1},aq^{-1},a^{1/2}q^{-3/2})=\CT \left[ \frac{(-aq^{-1}z,q^2z,z^{-1},q^2;q^2)_\infty}{(aq^{-3} z^2;q^4)_\infty} \right]   \nonumber \\
&=\mathrm{CT}\left[\sum_{i\geq 0}\frac{(aq^{-1}z)^iq^{i^2-i}}{(q^2;q^2)_i}\sum_{j\geq 0}\frac{(aq^{-3}z^2)^j}{(q^4;q^4)_j}\sum_{n=-\infty}^\infty (-1)^nq^{n^2-n}z^{-n}  \right] \nonumber \\
&=\sum_{i,j\geq 0}\frac{(-1)^iq^{2i^2+4j^2+4ij-3i-5j}a^{i+j}}{(q^2;q^2)_i(q^4;q^4)_j}.
\end{align}
Using \eqref{thm-4-new} we prove \eqref{(1,2,2)-new-12}.

Setting $(u,v,w)=(1,q,q^{-1})$, we have
\begin{align}
&E(1,q,q^{-1})=\CT \frac{(qz,-qz,q^2z,z^{-1},q^2;q^2)_\infty}{(z,q^{-1} z,-q^{-1} z;q^2)_\infty}=\CT \frac{(z^{-1},q^2;q^2)_\infty}{(1-z)(1-q^{-2}z^2)}
\label{1-(1,2,2)-new-15}  \\
&=\CT \left[-z^{-1}\frac{(q^2z^{-1},q^2;q^2)_\infty}{1-q^{-2}z^2}\right] =-(q^2;q^2)_\infty \CT \sum_{i=0}^\infty q^{-2i}z^{2i-1} \sum_{k=0}^\infty \frac{(-1)^kq^{k^2+k}z^{-k}}{(q^2;q^2)_k} \nonumber \\
&=(q^2;q^2)_\infty \sum_{i=0}^\infty \frac{q^{4i^2+4i}}{(q^2;q^2)_{2i+1}}.
\end{align}
Substituting \eqref{Slater38} into it, we prove \eqref{(1,2,2)-new-15}.

Setting $(u,v,w)=(aq,aq,a^{1/2}q^{1/2})$, we have
\begin{align}
&E(aq,aq,a^{1/2}q^{1/2})=\CT \frac{(-aqz,q^2z,z^{-1},q^2;q^2)_\infty}{(aq z^2;q^4)_\infty}   \nonumber \\
&=\mathrm{CT}\left[\sum_{i\geq 0}\frac{(aqz)^iq^{i^2-i}}{(q^2;q^2)_i}\sum_{j\geq 0}\frac{(aqz^2)^j}{(q^4;q^4)_j}\sum_{n=-\infty}^\infty (-1)^nq^{n^2-n}z^{-n}  \right] \nonumber \\
&=\sum_{i,j\geq 0}\frac{(-1)^iq^{2i^2+4j^2+4ij-i-j}a^{i+j}}{(q^2;q^2)_i(q^4;q^4)_j}.
\end{align}
Using \eqref{Cao-Wang-3.8-1} we prove \eqref{(1,2,2)-new-18}.

Setting $(u,v,w)=(a,a,a)$, we have
\begin{align}
&E(a,a,a)=\CT \frac{(-az,q^2z,z^{-1},q^2;q^2)_\infty}{(a^2z^2;q^4)_\infty}   \nonumber \\
&=\mathrm{CT}\left[\sum_{i\geq 0}\frac{(az)^iq^{i^2-i}}{(q^2;q^2)_i}\sum_{j\geq 0}\frac{(a^2z^2)^j}{(q^4;q^4)_j}\sum_{n=-\infty}^\infty (-1)^nq^{n^2-n}z^{-n}  \right] \nonumber \\
&=\sum_{i,j\geq 0}\frac{(-1)^iq^{2i^2+4j^2+4ij-2i-2j}a^{i+2j}}{(q^2;q^2)_i(q^4;q^4)_j}.
\end{align}
Using \eqref{Cao-Wang-3.8-2} we prove \eqref{(1,2,2)-new-6.30}.

Setting $(u,v,w)=(a,q^{-1},q)$, we have
\begin{align}
&E(a,q^{-1},q)=\CT \frac{(q^{-1}z,-q^{-1}z,q^2z,z^{-1},q^2;q^2)_\infty}{( az,qz,-qz;q^2)_\infty}   \nonumber \\
&=\CT \frac{(1-q^{-2}z^2)(q^2z,z^{-1},q^2;q^2)_\infty}{( az;q^2)_\infty}   \nonumber \\
&=\mathrm{CT}\left[(1-q^{-2}z^2)\sum_{i\geq 0}\frac{a^iz^i}{(q^2;q^2)_i}\sum_{n=-\infty}^\infty (-1)^nq^{n^2-n}z^{-n}  \right] \nonumber \\
&=\sum_{n\geq 0}\frac{(-1)^na^nq^{n^2-n}}{(q^2;q^2)_n}-\sum_{n\geq 0}\frac{(-1)^na^nq^{n^2+3n}}{(q^2;q^2)_n}  \nonumber \\
 &=(a;q^2)_\infty-(aq^4;q^2)_\infty.
\end{align}
This proves \eqref{(1,2,2)-new-11}.

Setting $(u,v,w)=(q^2,1,q^{-1})$, we have
\begin{align}
&E(q^2,1,q^{-1})=\mathrm{CT}\left[   \frac{(z,-z,z^{-1},q^2;q^2)_\infty}{(q^{-1} z,-q^{-1} z;q^2)_\infty}   \right]=(q^2;q^2)_\infty\mathrm{CT}\left[ \frac{(z^2;q^4)_\infty}{(q^2z^2;q^4)_\infty} (z^{-1};q^2)_\infty  \right]  \nonumber \\
&=(q^2;q^2)_\infty\mathrm{CT}\left[ \sum_{i\geq 0}\frac{(q^2;q^4)_i(q^{-2}z^2)^i}{(q^4;q^4)_i}   \sum_{j\geq 0}\frac{(-z^{-1})^jq^{j^2-j}}{(q^2;q^2)_j}   \right] \nonumber \\
&=(q^2;q^2)_\infty \sum_{i\geq 0}\frac{q^{4i^2-4i}}{(q^4;q^4)^2_i}  =\frac{2(q^2;q^2)_\infty}{(q^4;q^4)_\infty}.   \quad \text{(by \eqref{5yue14ri})}
\end{align}

This proves \eqref{(1,2,2)-new-19}.
\end{proof}

The next theorem provides some new companions to \cite[Theorem 4.5]{Cao-Wang}.
\begin{theorem}\label{thm-9}
We have
\begin{align}
\sum_{i,j,k\geq 0}\frac{(-1)^{i+k}q^{i^2+2j^2+k^2+2ij-2ik-2jk-j+k}a^{i+j}}{(q;q)_i(q^2;q^2)_j(q^2;q^2)_k}&=(q^2;q^2)_\infty, \quad (a \in \mathbb{C}) \label{(1,2,2)--1} \\
\sum_{i,j,k\geq 0}\frac{(-1)^{i+k}q^{i^2+2j^2+k^2+2ij-2ik-2jk+i+k}a^j}{(q;q)_i(q^2;q^2)_j(q^2;q^2)_k}&=(q^2;q^2)_\infty, \quad (a\in \mathbb{C}) \label{(1,2,2)--2} \\
\sum_{i,j,k\geq 0}\frac{q^{i^2+2j^2+k^2+2ij-2ik-2jk+j+k}}{(q;q)_i(q^2;q^2)_j(q^2;q^2)_k}&=\frac{(q^4;q^4)_\infty}{(q;q)^2_\infty}, \label{(1,2,2)--3} \\
% \sum_{i,j,k\geq 0}\frac{(-1)^iq^{i^2+2j^2+k^2+2ij-2ik-2jk+k}}{(q;q)_i(q^2;q^2)_j(q^2;q^2)_k}&=\frac{(q;q^2)^2_\infty}{(q^2;q^4)^2_\infty}, \label{(1,2,2)--4} \\
\sum_{i,j,k\geq 0}\frac{(-1)^iq^{i^2+2j^2+k^2+2ij-2ik-2jk+j+k}}{(q;q)_i(q^2;q^2)_j(q^2;q^2)_k}&={(q;q^2)^2_\infty}{(q^4;q^4)_\infty}, \label{(1,2,2)--5} \\
\sum_{i,j,k\geq 0}\frac{(-1)^iq^{i^2+2j^2+k^2+2ij-2ik-2jk+j+2k}}{(q;q)_i(q^2;q^2)_j(q^2;q^2)_k}&={(q;q)_\infty}{(q^2;q^4)_\infty}, \label{(1,2,2)--6} \\
% \sum_{i,j,k\geq 0}\frac{(-1)^iq^{i^2+2j^2+k^2+2ij-2ik-2jk-j}a^k}{(q;q)_i(q^2;q^2)_j(q^2;q^2)_k}&={(q^2;q^2)_\infty}, \label{(1,2,2)--7} \\
% \sum_{i,j,k\geq 0}\frac{(-1)^iq^{i^2+2j^2+k^2+2ij-2ik-2jk+i+j+k}a^k}{(q;q)_i(q^2;q^2)_j(q^2;q^2)_k}&={(q^2;q^2)_\infty}, \label{(1,2,2)--8} \\
\sum_{i,j,k\geq 0}\frac{(-1)^{i+k}q^{i^2+2j^2+k^2+2ij-2ik-2jk-i-j-k}a^{i+j}b^k}{(q;q)_i(q^2;q^2)_j(q^2;q^2)_k}&=\frac{(a,b;q^2)_\infty}{{(ab/q^2;q^2)_\infty}},\label{(1,2,2)--16} \\
\sum_{i,j,k\geq 0}\frac{(-1)^{i+k}q^{i^2+2j^2+k^2+2ij-2ik-2jk-2i-3j-k}a^{i+j}b^k}{(q;q)_i(q^2;q^2)_j(q^2;q^2)_k}&=\frac{(a,b;q^2)_\infty}{{(ab/q^2;q^2)_\infty}}.\label{(1,2,2)--17}
\end{align}
\end{theorem}
\begin{proof}
We define
\begin{align}\label{F-defn-122-cor}
F(u,v,w)=F(u,v,w;q):=\sum_{i,j,k\geq 0} \frac{(-1)^{i+k}u^iv^jw^kq^{i^2+2j^2+k^2+2ij-2ik-2jk-i-2j+k}}{(q;q)_i(q^2;q^2)_j(q^2;q^2)_k}.
\end{align}
By \eqref{Euler-1}, \eqref{Euler-2} and \eqref{Jacobi} we have
\begin{align}
F(u,v,w)&=\mathrm{CT}\left[  \sum_{i\geq 0}\frac{u^i z^i }{(q;q)_i} \sum_{j\geq 0}\frac{(-1)^j v^j z^j q^{j^2-j}}{(q^2;q^2)_j} \sum_{k\geq 0}\frac{ w^k  z^{-k} }{(q^2;q^2)_k} \sum_{n=-\infty}^\infty (-1)^nq^{n^2-n}z^{-n} \right]  \nonumber \\
&=\mathrm{CT}\left[  \frac{(vz,z^{-1},zq^2,q^2;q^2)_\infty}{(uz,uqz,wz^{-1};q^2)_\infty} \right].  \label{F-const}
\end{align}

 Setting $(u,v,w)=(aq,aq,1)$, we have
\begin{align}
&F(aq,aq,1)=\mathrm{CT}\left[      \frac{(q^2z, q^2;q^2)_\infty}{(aq^{2}z;q^2)_\infty}  \right]
=(q^2;q^2)_\infty.
\end{align}
This proves \eqref{(1,2,2)--1}.
		
Setting $(u,v,w)=(q^2,aq^2,1)$, we have
\begin{align}
&F(q^2,aq^{2},1)=\mathrm{CT}\left[  \frac{(aq^{2}z, q^2;q^2)_\infty}{(q^3 z;q^2)_\infty} \right]
=(q^2;q^2)_\infty.
\end{align}
This proves \eqref{(1,2,2)--2}.
		
Setting $(u,v,w)=(-q,q^3,-1)$, we have
\begin{align}
& F(-q,q^3,-1)=\CT \left[\frac{(q^3z;q^2)_\infty (q^2z,z^{-1},q^2;q^2)_\infty}{(-qz;q)_\infty (-z^{-1};q^2)_\infty} \right] \nonumber \\
&=(q^2;q^2)_\infty \CT \left[\frac{(q^2z;q)_\infty (z^{-1};q^2)_\infty}{(-qz;q)_\infty (-z^{-1};q^2)_\infty} \right] \nonumber \\
&=(q^2;q^2)_\infty  \CT \sum_{n=0}^\infty \frac{(-q;q)_n(-qz)^n}{(q;q)_n}\sum_{k=0}^\infty \frac{(-1;q^2)_k}{(q^2;q^2)_k}(-z^{-1})^k\nonumber \\
&= (q^2;q^2)_\infty \sum_{n=0}^\infty \frac{(-q;q)_n(-1;q^2)_nq^n}{(q;q)_n(q^2;q^2)_n} \nonumber \\
&=\sum_{n=0}^\infty \frac{(\zeta_4,-\zeta_4;q)_nq^n}{(q,q;q)_n}=\frac{(q^4;q^4)_\infty}{(q;q)_\infty^2}.
\end{align}
Her for the last equality we used \eqref{eq-Gauss}.
This proves \eqref{(1,2,2)--3}.

Setting $(u,v,w)=(q,q^3,-1)$, we have
\begin{align}
&F(q,q^3,-1)=\mathrm{CT}\left[  \frac{(q^3z,z^{-1},q^2;q^2)_\infty}{(qz,-z^{-1};q^2)_\infty}  \right] \nonumber \\
&=(q^2;q^2)_\infty\mathrm{CT}\left[      \sum_{i\geq 0}{q^i z^i } \sum_{n=0}^\infty \frac{(-1;q^2)_n(-z^{-1})^n}{(q^2;q^2)_n}  \right] \nonumber \\
&=(q^2;q^2)_\infty \sum_{n=0}^\infty\frac{(-1;q^2)_n(-q)^n}{(q^2;q^2)_n}=(q;q^2)^2_\infty(q^4;q^4)_\infty.
\end{align}
This proves \eqref{(1,2,2)--5}.
		
Setting $(u,v,w)=(q,q^3,-q)$, we have
\begin{align}
&F(q,q^3,-q)=\mathrm{CT}\left[  \frac{(q^3z,z^{-1},q^2;q^2)_\infty}{(qz,-qz^{-1};q^2)_\infty} \right] \nonumber \\
&=(q^2;q^2)_\infty\mathrm{CT}\left[ \sum_{i\geq 0}{q^i z^i } \sum_{n=0}^\infty \frac{(-1/q;q^2)_n(-qz^{-1})^n}{(q^2;q^2)_n}  \right] \nonumber \\
&=(q^2;q^2)_\infty \sum_{n=0}^\infty\frac{(-1/q;q^2)_n(-q^2)^n}{(q^2;q^2)_n}=(q;q)_\infty(q^2;q^4)_\infty.
\end{align}
This proves \eqref{(1,2,2)--6}.
		
% Setting $(u,v,w)=(q,q,a)$, we have
% \begin{align}
% F(q,q,a)=\mathrm{CT}\left[  \frac{(z^{-1},q^2;q^2)_\infty}{(az^{-1};q^2)_\infty} \right]
% =(q^2;q^2)_\infty\mathrm{CT}\left[   \sum_{n=0}^\infty\frac{(1/a;q^2)_n(az^{-1})^n}{(q^2;q^2)_n} \right]=(q^2;q^2)_\infty.  \nonumber
% \end{align}
% This proves \eqref{(1,2,2)--7}.

% Setting $(u,v,w)=(q^2,q^3,-a)$, we have
% \begin{align}
% &F(q^2,q^3,-a)=\mathrm{CT}\left[ \frac{(z^{-1},q^2;q^2)_\infty}{(-az^{-1};q^2)_\infty} \right] \nonumber \\
% &=(q^2;q^2)_\infty\mathrm{CT}\left[  \sum_{n=0}^\infty\frac{(-1/a;q^2)_n(-az^{-1})^n}{(q^2;q^2)_n} \right]=(q^2;q^2)_\infty.  \nonumber
% \end{align}
% Here for the last equality we used \eqref{eq-qbinomial}. This proves \eqref{(1,2,2)--8}.

We have
\begin{align}
&F(a,aq,bq^{-2})=\mathrm{CT}\left[ \frac{1}{(az,bq^{-2}z^{-1};q^2)_\infty} \sum_{n=-\infty}^\infty (-1)^nq^{n^2-n}z^{-n} \right] \\
&=\mathrm{CT}\left[  \sum_{i\geq 0}\frac{a^i z^i}{(q^2;q^2)_i} \sum_{j\geq 0}\frac{(bq^{-2}z^{-1})^j}{(q^2;q^2)_j}  \sum_{n=-\infty}^\infty (-1)^nq^{n^2-n}z^{-n} \right] =\sum_{i,j\geq 0}\frac{q^{i^2+j^2-2ij-i-j}a^ib^j}{(q^2;q^2)_i(q^2;q^2)_j}. \nonumber
\end{align}
By \eqref{eq-Cao-Wang-Thm31} we prove \eqref{(1,2,2)--16}. In the same way, we have
\begin{align}
&F(aq^{-1},aq^{-1},bq^{-2}) =\sum_{i,j\geq 0}\frac{q^{i^2+j^2-2ij-i-j}a^ib^j}{(q^2;q^2)_i(q^2;q^2)_j}. \nonumber
\end{align}
By \eqref{eq-Cao-Wang-Thm31} we prove \eqref{(1,2,2)--17}.
\end{proof}

\subsection{Identities of index $(1,2,4)$}
We find some new companions to \cite[Theorem 4.8]{Cao-Wang}.
\begin{theorem}\label{thm-8}
We have
\begin{align}
\sum_{i,j,k\geq 0}\frac{(-1)^jq^{i^2+j^2+6k^2+2ij+4ik+4jk+2k}}{(q;q)_i(q^2;q^2)_j(q^4;q^4)_k}&=\frac{(q^6,q^{10},q^{16};q^{16})_\infty}{(q^2;q^2)_\infty}, \label{(1,2,4)-2} \\
\sum_{i,j,k\geq 0}\frac{(-1)^jq^{i^2+j^2+6k^2+2ij+4ik+4jk-i-2k}}{(q;q)_i(q^2;q^2)_j(q^4;q^4)_k}&=\frac{2(q^6,q^{10},q^{16};q^{16})_\infty}{(q^2;q^2)_\infty}, \label{(1,2,4)-1} \\
\sum_{i,j,k\geq 0}\frac{(-1)^jq^{i^2+j^2+6k^2+2ij+4ik+4jk+i+2k}}{(q;q)_i(q^2;q^2)_j(q^4;q^4)_k}&=\frac{(-q^4;q^{8})_\infty}{(q^3,q^5,-q,-q^7;q^{8})_\infty\cdot(q^4;q^{8})_\infty}, \label{(1,2,4)-4} \\
\sum_{i,j,k\geq 0}\frac{(-1)^jq^{i^2+j^2+6k^2+2ij+4ik+4jk+2j+2k}}{(q;q)_i(q^2;q^2)_j(q^4;q^4)_k}&=\frac{(-q^4;q^{8})_\infty}{(q,q^7,-q^3,-q^5;q^{8})_\infty\cdot(q^4;q^{8})_\infty}, \label{(1,2,4)-3} \\
\sum_{i,j,k\geq 0}\frac{(-1)^jq^{i^2+j^2+6k^2+2ij+4ik+4jk+2i+2j+6k}}{(q;q)_i(q^2;q^2)_j(q^4;q^4)_k}&=\frac{(q^2,q^{14},q^{16};q^{16})_\infty}{(q^2;q^2)_\infty}, \label{(1,2,4)-5} \\
\sum_{i,j,k\geq 0}\frac{(-1)^kq^{i^2+j^2+6k^2+2ij+4ik+4jk}}{(q;q)_i(q^2;q^2)_j(q^4;q^4)_k}&=(-q;q^2)_\infty^2, \label{(1,2,4)-7} \\
\sum_{i,j,k\geq 0}\frac{(-1)^kq^{i^2+j^2+6k^2+2ij+4ik+4jk+j}}{(q;q)_i(q^2;q^2)_j(q^4;q^4)_k}&=\frac{(q^3,q^{3},q^6;q^{6})_\infty}{(q;q)_\infty}, \label{(1,2,4)-8} \\
\sum_{i,j,k\geq 0}\frac{(-1)^kq^{i^2+j^2+6k^2+2ij+4ik+4jk+i+3j+4k}}{(q;q)_i(q^2;q^2)_j(q^4;q^4)_k}&=\frac{(q,q^5,q^6;q^6)_\infty}{(q;q)_\infty}, \label{(1,2,4)-11} \\
\sum_{i,j,k\geq 0}\frac{(-1)^kq^{i^2+j^2+6k^2+2ij+4ik+4jk+i+2k}}{(q;q)_i(q^2;q^2)_j(q^4;q^4)_k}&=\frac{1}{(q;q^2)_\infty}, \label{(1,2,4)-10} \\
\sum_{i,j,k\geq 0}\frac{(-1)^kq^{i^2+j^2+6k^2+2ij+4ik+4jk+2i+j+4k}}{(q;q)_i(q^2;q^2)_j(q^4;q^4)_k}&=\frac{(q,q^5,q^6;q^6)_\infty}{(q;q)_\infty}, \label{(1,2,4)-13} \\
\sum_{i,j,k\geq 0}\frac{(-1)^{i+k}q^{i^2+j^2+6k^2+2ij+4ik+4jk-2i-2j-2k}}{(q;q)_i(q^2;q^2)_j(q^4;q^4)_k}&=(q^6,q^{10},q^{16};q^{16})_\infty (q^4,q^{28};q^{32})_\infty, \label{(1,2,4)-16} \\
\sum_{i,j,k\geq 0}\frac{(-1)^{i+k}q^{i^2+j^2+6k^2+2ij+4ik+4jk+2k}}{(q;q)_i(q^2;q^2)_j(q^4;q^4)_k}&=
(q^2,q^{14},q^{16};q^{16})_\infty (q^{12},q^{20};q^{32})_\infty,  \label{(1,2,4)-17} \\
\sum_{i,j,k\geq 0}\frac{(-1)^{i+k}q^{i^2+j^2+6k^2+2ij+4ik+4jk+2i+2j+6k}}{(q;q)_i(q^2;q^2)_j(q^4;q^4)_k}&=(q^6,q^{10},q^{16};q^{16})_\infty (q^4,q^{28};q^{32})_\infty, \label{(1,2,4)-18} \\
\sum_{i,j,k\geq 0}\frac{(-1)^{i+k}q^{i^2+j^2+6k^2+2ij+4ik+4jk-i-j-4k}a^{i+j+2k}}{(q;q)_i(q^2;q^2)_j(q^4;q^4)_k}&=(aq;q^{2})_\infty, \label{(1,2,4)-19} \\
\sum_{i,j,k\geq 0}\frac{(-1)^{i+k}q^{i^2+j^2+6k^2+2ij+4ik+4jk-i-2k}a^{i+j+2k}}{(q;q)_i(q^2;q^2)_j(q^4;q^4)_k}&=(a;q^{2})_\infty, \label{(1,2,4)-20}  \\
\sum_{i,j,k\geq 0}\frac{(-1)^{i+k}q^{i^2+j^2+6k^2+2ij+4ik+4jk}}{(q;q)_i(q^2;q^2)_j(q^4;q^4)_k}&=(q^2;q^4)_\infty, \label{(1,2,4)-21} \\
\sum_{i,j,k\geq 0}\frac{(-1)^{j+k}q^{i^2+j^2+6k^2+2ij+4ik+4jk-2i-2j-6k}}{(q;q)_i(q^2;q^2)_j(q^4;q^4)_k}&=\frac{2}{(q^2;q^4)_\infty}, \label{(1,2,4)-22} \\
\sum_{i,j,k\geq 0}\frac{(-1)^{j+k}q^{i^2+j^2+6k^2+2ij+4ik+4jk-i-2k}}{(q;q)_i(q^2;q^2)_j(q^4;q^4)_k}&=\frac{2}{(q^2;q^4)_\infty}, \label{(1,2,4)-23} \\
\sum_{i,j,k\geq 0}\frac{(-1)^{j+k}q^{i^2+j^2+6k^2+2ij+4ik+4jk}}{(q;q)_i(q^2;q^2)_j(q^4;q^4)_k}&=\frac{(q^6,q^6,q^{12};q^{12})_\infty}{(q^2;q^2)_\infty}, \label{(1,2,4)-26} \\
\sum_{i,j,k\geq 0}\frac{(-1)^{j+k}q^{i^2+j^2+6k^2+2ij+4ik+4jk+i+2j+2k}}{(q;q)_i(q^2;q^2)_j(q^4;q^4)_k}&=\frac{1}{(q^2;q^4)_\infty}, \label{(1,2,4)-27} \\
\sum_{i,j,k\geq 0}\frac{(-1)^{j+k}q^{i^2+j^2+6k^2+2ij+4ik+4jk+2i+2j+4k}}{(q;q)_i(q^2;q^2)_j(q^4;q^4)_k}&=\frac{1}{(q^4,q^6,q^8;q^{12})_\infty}.\label{(1,2,4)-29}
\end{align}
\end{theorem}
\begin{proof}
We define
\begin{align}\label{Thm124-F-start}
F(u,v,w)=F(u,v,w;q):=\sum_{i,j,k\geq 0} \frac{(-1)^{i+j+k}u^iv^jw^kq^{i^2+j^2+6k^2+2ij+4ik+4jk-i-j-4k}}{(q;q)_i(q^2;q^2)_j(q^4;q^4)_k}.
\end{align}
By \eqref{Euler-1}, \eqref{Euler-2} and \eqref{Jacobi} we have
\begin{align}
&F(u,v,w)=\CT \sum_{i\geq 0}\frac{u^i z^i }{(q;q)_i} \sum_{j\geq 0}\frac{v^j z^j }{(q^2;q^2)_j} \sum_{k\geq 0}\frac{ (-1)^k w^k  z^{2k} q^{2k^2-2k}}{(q^4;q^4)_k} \sum_{n=-\infty}^\infty (-1)^nq^{n^2-n}z^{-n}  \nonumber \\
&=\CT \frac{(wz^2;q^4)_\infty(z^{-1},q^2z,q^2;q^2)_\infty}{(uz;q)_\infty(vz;q^2)_\infty} \nonumber   \\
&=\CT \frac{(w^{1/2}z,-w^{1/2}z,z^{-1},zq^2,q^2;q^2)_\infty}{(uz,uqz,vz;q^2)_\infty}. \label{F-start}
\end{align}
Setting $(u,v,w)=(-q,q,-q^6)$ and using \eqref{Euler-1}--\eqref{eq-qbinomial}, we have
\begin{align}
&F(-q,q,-q^6)
=\CT\frac{(-q^6z^2;q^4)_\infty (z^{-1},q^2z,q^2;q^2)_\infty}{(-qz;q)_\infty (qz;q^2)_\infty} \label{1-(1,2,4)-2} \\
&=\CT \frac{1}{(-q^2z;q^2)_\infty} \cdot \frac{(-q^6z^2;q^4)_\infty}{(q^2z^2;q^4)_\infty} \cdot (z^{-1},q^2z,q^2;q^2)_\infty \nonumber \\
&=\CT \sum_{i\geq 0} \frac{(-1)^iq^{2i}z^i}{(q^2;q^2)_i} \sum_{j=0}^\infty \frac{(-q^4;q^4)_jq^{2j}z^{2j}}{(q^4;q^4)_j} \sum_{k=-\infty}^\infty (-1)^kq^{k^2-k}z^{-k} \nonumber \\
&=\sum_{i,j\geq 0} \frac{(-q^4;q^4)_j q^{(i+2j)^2+i}}{(q^2;q^2)_i(q^4;q^4)_j} =\sum_{j\geq 0}\frac{q^{4j^2}(-q^4;q^4)_j}{(q^4;q^4)_j} \sum_{i\geq 0} \frac{q^{i^2+(4j+1)i}}{(q^2;q^2)_i} \nonumber \\
&=\sum_{j\geq 0} \frac{q^{4j^2}(-q^4;q^4)_j}{(q^4;q^4)_j} (-q^{4j+2};q^2)_\infty =(-q^2;q^2)_\infty \sum_{j\geq 0}\frac{q^{4j^2}(-q^4;q^4)_j}{(q^4;q^4)_j(-q^2;q^2)_{2j}} \nonumber \\
&=(-q^2;q^2)_\infty \sum_{j\geq 0} \frac{q^{4j^2}}{(-q^2,q^4;q^4)_j}. \label{F-relation-0}
\end{align}
Substituting \eqref{Slater39} with $q$ replaced by $-q^2$ into it, we obtain \eqref{(1,2,4)-2}.

Setting $(u,v,w)=(-1,q,-q^2)$, we have
\begin{align}
& F(-1,q,-q^2)=\CT \frac{1}{(-z;q^2)_\infty} \cdot \frac{(-q^2z^2;q^4)_\infty}{(q^2z^2;q^4)_\infty} \cdot (z^{-1},q^2z,q^2;q^2)_\infty \label{1-(1,2,4)-1}
\end{align}
The right side is the same with the right side of \eqref{05yue.10ri-1}. Hence by the deduction there,  we obtain \eqref{(1,2,4)-1}.

Setting $(u,v,w)=(-q^2,q,-q^6)$, we have
\begin{align}
&F(-q^2,q,-q^6)=\CT \left[(1+qz)\frac{1}{(-q^2z;q^2)_\infty}\cdot \frac{(-q^6z^2;q^4)_\infty}{(q^2z^2;q^4)_\infty} \cdot (z^{-1},q^2z,q^2;q^2)_\infty  \right] \nonumber \\
&=\sum_{i,j\geq 0} \frac{(-q^4;q^4)_j (q^{(i+2j)^2+i}-q^{(i+2j+1)^2+i})}{(q^2;q^2)_i(q^4;q^4)_j} \nonumber \\
&=(-q^2;q^2)_\infty \sum_{j\geq 0} \frac{q^{4j^2}}{(-q^2;q^4)_j(q^4;q^4)_j}-q(-q^2;q^2)_\infty \sum_{j\geq 0} \frac{q^{4j^2+4j}}{(-q^2;q^4)_{j+1}(q^4;q^4)_j}. \label{F-relation-1}
\end{align}
Substituting  \eqref{Slater38} and \eqref{Slater39} into it, we prove \eqref{(1,2,4)-4}.

In the same way, we have
\begin{align}
&F(-q,q^3,-q^6)=\CT\left[\frac{(1-qz)(-q^6z^2;q^4)_\infty (z^{-1},q^2z,q^2;q^2)_\infty}{(-q^2z;q^2)_\infty (q^2z^2;q^4)_\infty}\right] \nonumber \\
&=(-q^2;q^2)_\infty \sum_{j\geq 0} \frac{q^{4j^2}}{(-q^2;q^4)_j(q^4;q^4)_j}+q(-q^2;q^2)_\infty \sum_{j\geq 0} \frac{q^{4j^2+4j}}{(-q^2;q^4)_{j+1}(q^4;q^4)_j}.  \label{F-relation-2}
\end{align}
Substituting \eqref{Slater38} and \eqref{Slater39} into it, we prove \eqref{(1,2,4)-3}.

Similarly, we have
\begin{align}\label{F-relation-3}
&F(-q^3,q^3,-q^{10})=\CT \left[ \frac{1}{(-q^4z;q^2)_\infty} \cdot \frac{(-q^{10}z^2;q^4)_\infty}{(q^6z^2;q^4)_\infty} \cdot (z^{-1},q^2z,q^2;q^2)_\infty \right] \nonumber \\
&=\sum_{i,j\geq 0} \frac{(-q^4;q^4)_jq^{(i+2j)^2+3i+4j}}{(q^2;q^2)_i(q^4;q^4)_j} =(-q^2;q^2)_\infty \sum_{j\geq 0} \frac{q^{4j^2+4j}}{(-q^2;q^4)_{j+1}(q^4;q^4)_j}.
\end{align}
Substituting \eqref{Slater38} into it, we prove \eqref{(1,2,4)-5}.

  Setting $(u,v,w)=(-q,-q,q^{4})$, we have
\begin{align}
&F(-q,-q,q^{4})=\CT \left[\frac{(q^2 z,z^{-1},q^2z,q^2;q^2)_\infty}{(-q z,-q z;q^2)_\infty} \right] \nonumber \\
&=\mathrm{CT}\left[  \sum_{i\geq 0}\frac{(-zq^2)^i q^{i^2-i} }{(q^2;q^2)_i} \sum_{j\geq 0}\frac{(-qz)^j }{(q^2;q^2)_j}\sum_{k\geq 0}\frac{(-qz)^k }{(q^2;q^2)_k}  \sum_{n=-\infty}^\infty (-1)^nq^{n^2-n}z^{-n} \right] \nonumber \\
&=\sum_{i,j,k\geq0}\frac{q^{(i+j+k)^2+i^2}}{(q^2;q^2)_i(q^2;q^2)_j(q^2;q^2)_k}.
\end{align}
Using \eqref{(2,2,2)-1} we prove \eqref{(1,2,4)-7}.

Setting $(u,v,w)=(-q,-q^2,q^4)$, we have
\begin{align}
&F(-q,-q^2,q^4)=\CT \frac{(q^2z;q^2)_\infty (q^2z,z^{-1},q^2;q^2)_\infty}{(-qz;q)_\infty} \nonumber \\
&=\CT \sum_{i\geq 0} \frac{(-qz)^i}{(q;q)_i} \sum_{j\geq 0} \frac{(-1)^jz^jq^{j^2+j}}{(q^2;q^2)_j} \sum_{k=-\infty}^\infty (-1)^kq^{k^2-k}z^{-k} \nonumber \\
&=\sum_{i,j\geq 0} \frac{q^{(i+j)^2+j^2}}{(q;q)_i(q^2;q^2)_j}.
\end{align}
Now by \eqref{Bressoud-1} we obtain  \eqref{(1,2,4)-8}.

Similarly, we have
\begin{align}
F(-q^2,-q,q^6)=\CT \left[\frac{(q^3z;q^2)_\infty (z^{-1},q^2z,q^2;q^2)_\infty}{(-q^2z;q^2)_\infty (-qz;q^2)_\infty} \right]=\sum_{i,j\geq 0} \frac{q^{i^2+2ij+2j^2+j}}{(q;q)_i(q^2;q^2)_j}, \\
F(-q^2,-q^4,q^8)=\CT \left[\frac{(q^4z,q^2z,z^{-1},q^2;q^2)_\infty}{(-q^2z;q)_\infty} \right] =\sum_{i,j\geq 0} \frac{q^{i^2+2ij+2j^2+i+2j}}{(q;q)_i(q^2;q^2)_j}.
\end{align}
Now by \eqref{Bressoud-2} and \eqref{Bressoud-3} we obtain \eqref{(1,2,4)-10} and \eqref{(1,2,4)-11}, respectively.

Setting $(u,v,w)=(-q^3,-q^2,q^8)$, we have
\begin{align}
&F(-q^3,-q^2,q^8)= \CT\frac{(zq^4,zq^2,z^{-1},q^2;q^2)_\infty}{(-q^2 z,-q^3 z;q^2)_\infty} \nonumber \\
&=\mathrm{CT}\left[  \sum_{i\geq 0}\frac{(-zq^4)^i q^{i^2-i} }{(q^2;q^2)_i} \sum_{j\geq 0}\frac{(-q^2z)^j }{(q^2;q^2)_j}\sum_{k\geq 0}\frac{(-q^3z)^k }{(q^2;q^2)_k}  \sum_{n=-\infty}^\infty (-1)^nq^{n^2-n}z^{-n} \right] \nonumber \\
&=\sum_{i,j,k\geq0}\frac{q^{(i+j+k)^2+i^2+2i+j+2k}}{(q^2;q^2)_i(q^2;q^2)_j(q^2;q^2)_k}.
\label{1-(1,2,4)-13}
\end{align}
Now using \eqref{(2,2,2)-3} we prove \eqref{(1,2,4)-13}.

Setting $(u,v,w)=(q^{-1},-q^{-1},q^2)$, we have
\begin{align}
F(q^{-1},-q^{-1},q^2)=\CT \frac{(q z,-q z,q^2 z,z^{-1},q^2;q^2)_\infty}{(q^{-1}z,z,-q^{-1} z;q^2)_\infty}.
\end{align}
The right side is the same with the right side in \eqref{1-(1,2,2)-new-15}. Therefore, by the deductions there we obtain \eqref{(1,2,4)-16}.

Setting $(u,v,w)=(q,-q,q^6)$, we have
\begin{align}
&F(q,-q,q^6)=\CT \frac{(q^3 z,-q^3 z,z^{-1},q^2;q^2)_\infty}{(qz,-qz;q^2)_\infty}  =(q^2;q^2)_\infty \CT \left[ \frac{(z^{-1};q^2)_\infty}{1-q^{2}z^2}\right]\nonumber \\
&=(q^2;q^2)_\infty \CT \left[\sum_{i\geq 0}q^{2i}z^{2i} \sum_{k\geq 0} \frac{(-1)^kz^{-k}q^{k^2-k}}{(q^2;q^2)_k} \right] =(q^2;q^2)_\infty \sum_{i=0}^\infty \frac{q^{4i^2}}{(q^2;q^2)_{2i}}.
\end{align}
Substituting \eqref{Slater39} with $q$ replaced by $q^2$ into it, we obtain \eqref{(1,2,4)-17}.

Setting $(u,v,w)=(q^{3},-q^{3},q^{10})$, we have
\begin{align}
&F(q^{3},-q^{3},q^{10})=\CT \frac{(q^5 z,-q^5 z,q^2 z,z^{-1},q^2;q^2)_\infty}{(q^3 z,-q^3 z,q^4 z;q^2)_\infty} \nonumber \\
&=(q^2;q^2)_\infty \CT \frac{(z^{-1};q^2)_\infty(1-q^2z)}{1-q^6z^2}
=(q^2;q^2)_\infty \CT \left[-q^2z\frac{(q^{-2}z^{-1};q^2)_\infty}{1-q^6z^2} \right] \nonumber \\
&=-q^2(q^2;q^2)_\infty\CT\left[\sum_{i=0}^\infty q^{6i}z^{2i+1}\sum_{j=0}^\infty \frac{(-1)^jz^{-j}q^{j^2-3j}}{(q^2;q^2)_j}\right] \nonumber \\
&=(q^2;q^2)_\infty \sum_{i=0}^\infty \frac{q^{4i^2+4i}}{(q^2;q^2)_{2i+1}}.
\end{align}
Using \eqref{Slater38}, we prove \eqref{(1,2,4)-18}.

Setting $(u,v,w)=(a,-a,a^2)$, we have
\begin{align}
&F(a,-a,a^2)=\CT \left[\frac{1}{(aq z;q^2)_\infty}\sum_{n=-\infty}^\infty (-1)^nq^{n^2-n}z^{-n}\right]  \nonumber \\
&=\mathrm{CT}\left[\sum_{i\geq 0}\frac{(aqz)^i}{(q^2;q^2)_i}\sum_{n=-\infty}^\infty (-1)^nq^{n^2-n}z^{-n}  \right] =\sum_{n\geq 0}\frac{(-1)^na^nq^{n^2}}{(q^2;q^2)_n}=(aq;q^2)_\infty.
\end{align}
This proves \eqref{(1,2,4)-19}.

Setting $(u,v,w)=(a,-aq,a^2q^2)$, we have
\begin{align}
&F(a,-aq,a^2q^2)=\CT \left[\frac{1}{(a z;q^2)_\infty}\sum_{n=-\infty}^\infty (-1)^nq^{n^2-n}z^{-n} \right] \nonumber \\
&=\mathrm{CT}\left[\sum_{i\geq 0}\frac{(az)^i}{(q^2;q^2)_i}\sum_{n=-\infty}^\infty (-1)^nq^{n^2-n}z^{-n}  \right]=\sum_{n\geq 0}\frac{(-1)^na^nq^{n^2-n}}{(q^2;q^2)_n}=(a;q^2)_\infty.
\end{align}
This proves \eqref{(1,2,4)-20}.

Setting $(u,v,w)=(q,-q,q^{4})$, we have
\begin{align}
&F(q,-q,q^{4})=\CT \frac{(-q^2 z,z^{-1},q^2z,q^2;q^2)_\infty}{(q z,-q z;q^2)_\infty}  \nonumber \\
&=\mathrm{CT}\left[ \sum_{i\geq 0}\frac{(zq^2)^i q^{i^2-i} }{(q^2;q^2)_i} \sum_{j\geq 0}\frac{(qz)^j }{(q^2;q^2)_j}\sum_{k\geq 0}\frac{(-qz)^k }{(q^2;q^2)_k}  \sum_{n=-\infty}^\infty (-1)^nq^{n^2-n}z^{-n} \right] \nonumber \\
&=\sum_{i,j,k\geq0}\frac{(-1)^{i+j}q^{(i+j+k)^2+i^2}}{(q^2;q^2)_i(q^2;q^2)_j(q^2;q^2)_k}.
\end{align}
Using \eqref{(2,2,2)-1}, we obtain \eqref{(1,2,4)-21}.

Setting $(u,v,w)=(-q^{-1},q^{-1},q^{-2})$, we have
\begin{align}
&F(-q^{-1},q^{-1},q^{-2})=\CT \frac{(z^{-1},q^2z,q^2;q^2)_\infty}{(- z;q^2)_\infty}  \nonumber \\
&=\mathrm{CT}\left[ \sum_{i\geq 0}\frac{(-z)^i }{(q^2;q^2)_i}   \sum_{n=-\infty}^\infty (-1)^nq^{n^2-n}z^{-n} \right] =\sum_{i\geq0}\frac{q^{i^2-i}}{(q^2;q^2)_i}=(-1;q^2)_\infty.
\end{align}
This proves \eqref{(1,2,4)-22}.

Setting $(u,v,w)=(-1,q,q^{2})$, we have
\begin{align}
&F(-1,q,q^{2})=\CT \frac{(z^{-1},q^2z,q^2;q^2)_\infty}{(- z;q^2)_\infty}  \nonumber \\
&=\mathrm{CT}\left[ \sum_{i\geq 0}\frac{(-z)^i }{(q^2;q^2)_i}   \sum_{n=-\infty}^\infty (-1)^nq^{n^2-n}z^{-n} \right] =\sum_{i\geq0}\frac{q^{i^2-i}}{(q^2;q^2)_i}=(-1;q^2)_\infty.
\end{align}
This proves \eqref{(1,2,4)-23}.

Setting $(u,v,w)=(-q,q,q^4)$, we have
\begin{align}
&F(-q,q,q^4)=\CT \left[\frac{(q^2z;q^2)_\infty (z^{-1},q^2z,q^2;q^2)_\infty}{(q^2z^2;q^4)_\infty } \right]   \nonumber \\
&=\CT \sum_{i\geq 0} \frac{(-1)^iq^{2i}z^iq^{i^2-i}}{(q^2;q^2)_i} \sum_{j\geq 0} \frac{q^{2j}z^{2j}}{(q^4;q^4)_j} \sum_{n=-\infty}^\infty (-1)^nq^{n^2-n}z^{-n} \nonumber \\
&=\sum_{i,j\geq 0} \frac{q^{2i^2+4ij+4j^2}}{(q^2;q^2)_i(q^4;q^4)_j}. \nonumber
\end{align}
Using \eqref{Bressoud-1} we obtain \eqref{(1,2,4)-26}.

Setting $(u,v,w)=(-q^2,q^3,q^6)$, we have
\begin{align}
&F(-q^2,q^3,q^6)=\CT\frac{(z^{-1},q^2z,q^2;q^2)_\infty}{(- q^2z;q^2)_\infty}  \nonumber \\
&=\mathrm{CT}\left[ \sum_{i\geq 0}\frac{(-q^2z)^i }{(q^2;q^2)_i}   \sum_{n=-\infty}^\infty (-1)^nq^{n^2-n}z^{-n} \right] =\sum_{i\geq0}\frac{q^{i^2+i}}{(q^2;q^2)_i}=(-q^2;q^2)_\infty.
\end{align}
This proves \eqref{(1,2,4)-27}.

Setting $(u,v,w)=(-q^{3},q^3,q^8)$, we have
\begin{align}
&F(-q^{3},q^3,q^8)=\CT \frac{(q^4z,z^{-1},q^2z,q^2;q^2)_\infty}{(q^3z,- q^3z;q^2)_\infty}  \nonumber \\
&=\mathrm{CT}\left[ \sum_{i\geq 0}\frac{(-q^4z)^i q^{n^2-n}}{(q^2;q^2)_i} \sum_{j\geq 0}\frac{(q^3z)^j }{(q^2;q^2)_j} \sum_{k\geq 0}\frac{(-q^3z)^k }{(q^2;q^2)_k}  \sum_{n=-\infty}^\infty (-1)^nq^{n^2-n}z^{-n} \right] \nonumber \\
&=\sum_{i,j,k\geq0}\frac{(-1)^jq^{2i^2+j^2+k^2+2ij+2ik+2jk+2i+2j+2k}}{(q^2;q^2)_i(q^2;q^2)_j(q^2;q^2)_k}.  \label{1-(1,2,4)-29}
\end{align}
Now using \eqref{(2,2,2)-11} we prove \eqref{(1,2,4)-29}.
\end{proof}
\begin{rem}
Comparing \eqref{F-relation-0}--\eqref{F-relation-3}, we obtain the following relations:
\begin{align}
F(-q^2,q,-q^6)=F(-q,q,-q^6)-qF(-q^3,q^3,-q^{10}), \\
F(-q,q^3,-q^6)=F(-q,q,-q^6)+qF(-q^3,q^3,-q^{10}).
\end{align}
From the right-hand sides of \eqref{(1,2,4)-4} and \eqref{(1,2,4)-3}, we deduce that
\begin{align}
    F(-q^2,q,-q^6;q)=F(q,-q^3,-q^6;-q).
\end{align}
Comparing \eqref{(1,2,4)-16} and \eqref{(1,2,4)-18}, we deduce that
\begin{align}
    F(q^{-1},-q^{-1},q^2)=F(q^3,-q^3,q^{10}).
\end{align}
Comparing \eqref{(1,2,4)-22}, \eqref{(1,2,4)-23} and \eqref{(1,2,4)-27}, we deduce that
\begin{align}
F(-q^{-1},q^{-1},q^{-2})=F(-1,q,q^2)=2F(-q^2,q^3,q^6).
\end{align}
\end{rem}

\section{Identities involving quadruple sums}\label{sec-quadruple}
 It is difficult to exhaust all and tedious to discuss quadruple sum identities of the form \eqref{id-form}. Here we only focus on a special class of identities of index $(1,1,1,2)$ and use it as an example to illustrate the phenomenon.

Recall that Sills \cite[Eqs.\ (5.68),(5.69)]{Sills2003} discovered two identities of index $(1,1,1,2)$.  See  \eqref{03.10-3} and \eqref{intro-Sills-3}.  We find more companions of them. Note that \eqref{03.10-3} follows from \eqref{03.10-0} by setting $a=b=-1$.

\subsection{Identities of index $(1,1,1,2)$}

\begin{theorem}
We have
\begin{align}
&\sum_{i,j,k,l\geq 0}\frac{(-1)^{i+j+l}q^{i^2+j^2+k^2+6l^2+2ij+2ik+2jk+4il+4jl+4kl+k}a^{i+k+2l}b^{j+k+2l}}{(q^2;q^2)_i(q^2;q^2)_j(q^2;q^2)_k(q^4;q^4)_l}=(aq,bq;q^2)_\infty, \label{03.10-0} \\
&\sum_{i,j,k,l\geq 0}\frac{(-1)^lq^{i^2+j^2+k^2+6l^2+2ij+2ik+2jk+4il+4jl+4kl+j+k}}{(q^2;q^2)_i(q^2;q^2)_j(q^2;q^2)_k(q^4;q^4)_l}=\frac{(q^3;q^3)^2_\infty}{(q;q)_\infty(q^6;q^6)_\infty}, \label{03.10-2} \\
&\sum_{i,j,k,l\geq 0}\frac{(-1)^lq^{i^2+j^2+k^2+6l^2+2ij+2ik+2jk+4il+4jl+4kl-i-j+k}}{(q^2;q^2)_i(q^2;q^2)_j(q^2;q^2)_k(q^4;q^4)_l}=\frac{3}{(q^2,q^2;q^4)_\infty}, \label{03.10-4} \\
&\sum_{i,j,k,l\geq 0}\frac{(-1)^iq^{i^2+j^2+k^2+6l^2+2ij+2ik+2jk+4il+4jl+4kl+2i+2j+3k+6l}}{(q^2;q^2)_i(q^2;q^2)_j(q^2;q^2)_k(q^4;q^4)_l}=\frac{(q^2,q^{14},q^{16};q^{16})_\infty}{(q^2;q^2)_\infty}, \label{03.10-new} \\
&\sum_{i,j,k,l\geq 0}\frac{(-1)^iq^{i^2+j^2+k^2+6l^2+2ij+2ik+2jk+4il+4jl+4kl+k+2l}}{(q^2;q^2)_i(q^2;q^2)_j(q^2;q^2)_k(q^4;q^4)_l}=\frac{(q^6,q^{10},q^{16};q^{16})_\infty}{(q^2;q^2)_\infty}, \label{03.10-8} \\
&\sum_{i,j,k,l\geq 0}\frac{(-1)^iq^{i^2+j^2+k^2+6l^2+2ij+2ik+2jk+4il+4jl+4kl+j+2k+2l}}{(q^2;q^2)_i(q^2;q^2)_j(q^2;q^2)_k(q^4;q^4)_l} \nonumber \\
&\quad \quad \quad =\frac{(q,q^7,q^8,q^{9},q^{15};q^{16})_\infty}{(q^2,q^3,q^4,q^4,q^5,q^{11},q^{12},q^{12},q^{13},q^{14};q^{16})_\infty}, \quad \text{(\cite[Eq.\ (5.68)]{Sills2003})}\label{03.10-9} \\
&\sum_{i,j,k,l\geq 0}\frac{(-1)^iq^{i^2+j^2+k^2+6l^2+2ij+2ik+2jk+4il+4jl+4kl+2i+k+2l}}{(q^2;q^2)_i(q^2;q^2)_j(q^2;q^2)_k(q^4;q^4)_l}\nonumber \\
&\quad \quad \quad =\frac{(q^3,q^5,q^8,q^{11},q^{13};q^{16})_\infty}{(q,q^4,q^4,q^6,q^{7},q^{9},q^{10},q^{12},q^{12},q^{15};q^{16})_\infty}, \label{03.10-10} \\
&\sum_{i,j,k,l\geq 0}\frac{(-1)^iq^{i^2+j^2+k^2+6l^2+2ij+2ik+2jk+4il+4jl+4kl-j-2l}}{(q^2;q^2)_i(q^2;q^2)_j(q^2;q^2)_k(q^4;q^4)_l}=\frac{2(q^6,q^{10},q^{16};q^{16})_\infty}{(q^2;q^2)_\infty}, \label{03.10-13} \\
&\sum_{i,j,k,l\geq 0} \frac{(-1)^jq^{i^2+j^2+k^2+6l^2+2ij+2ik+2jk+4il+4jl+4kl+k}a^{i+j+2l}}{(q^2;q^2)_i(q^2;q^2)_j(q^2;q^2)_k(q^4;q^4)_l}=(-q^2;q^2)_\infty (-a^2q^4;q^8)_\infty. \label{1112-parameter-1}
\end{align}
\end{theorem}

\begin{proof}
We define
\begin{align}
&F(u,v,w,t)=F(u,v,w,t;q)\nonumber \\
&:=\sum_{i,j,k,l\geq 0} \frac{(-1)^{i+j+k+l}u^iv^jw^kt^{2l}q^{i^2+j^2+k^2+6l^2+2ij+2ik+2jk+4il+4jl+4kl-i-j-k-4l}}{(q^2;q^2)_i(q^2;q^2)_j(q^2;q^2)_k(q^4;q^4)_l}.
\end{align}
By \eqref{Euler-1}, \eqref{Euler-2} and \eqref{Jacobi} we have
\begin{align}
F(u,v,w,t)&=\CT \sum_{i\geq 0}\frac{u^i z^i}{(q^2;q^2)_i} \sum_{j\geq 0}\frac{v^j z^j }{(q^2;q^2)_j} \sum_{k\geq 0}\frac{  w^k  z^{k} }{(q^2;q^2)_k}\sum_{l\geq 0}\frac{  (-1)^lt^{2l}  z^{2l} q^{2l^2-2l}}{(q^4;q^4)_l} \nonumber \\
&\quad \quad \quad \cdot \sum_{n=-\infty}^\infty (-1)^nq^{n^2-n}z^{-n}   \nonumber \\
&=\CT \frac{(tz,-tz,z^{-1},zq^2,q^2;q^2)_\infty}{(uz,vz,wz;q^2)_\infty}.
\end{align}

When $t=-w$, we have
\begin{align}
&F(u,v,w,-w)=\CT \frac{(-wz,zq^2,z^{-1},q^2;q^2)_\infty}{(u z,v z;q^2)_\infty}   \nonumber \\
 &= \mathrm{CT}\left[  \sum_{i\geq 0}\frac{(wz)^i q^{i^2-i}}{(q^2;q^2)_i} \sum_{j\geq 0}\frac{(uz)^j }{(q^2;q^2)_j} \sum_{k\geq 0}\frac{  (vz)^k }{(q^2;q^2)_k} \sum_{n=-\infty}^\infty (-1)^nq^{n^2-n}z^{-n}   \right] \nonumber \\
 &=  \sum_{i,j,k\geq 0} \frac{(-1)^{i+j+k}q^{2i^2+j^2+k^2+2ij+2ik+2jk-2i-j-k}u^jv^kw^i}{(q^2;q^2)_i(q^2;q^2)_j(q^2;q^2)_k}. \label{quadruple-F-start}
\end{align}

By \eqref{quadruple-F-start} and \eqref{(2,2,2)-1} we have
\begin{align*}
F(aq,bq,-abq^2,abq^2)=\sum_{i,j,k\geq 0} \frac{(-1)^{j+k}q^{2i^2+j^2+k^2+2ij+2ik+2jk}a^{i+j}b^{i+k}}{(q^2;q^2)_i(q^2;q^2)_j(q^2;q^2)_k}=(aq,bq;q^2)_\infty.
\end{align*}
This proves \eqref{03.10-0}.

Setting $(u,v,w,t)=(-q,-q^2,-q^2,q^2)$, we have
\begin{align}
&F(-q,-q^2,-q^2,q^2)= \sum_{i,j,k\geq 0} \frac{q^{2i^2+j^2+k^2+2ij+2ik+2jk+k}}{(q^2;q^2)_i(q^2;q^2)_j(q^2;q^2)_k}.  \label{1-03.10-2}
\end{align}
Now using \eqref{(2,2,2)-4} we prove \eqref{03.10-2}.

Setting $(u,v,w,t)=(-1,-1,-q^2,q^2)$, we have
\begin{align}
&F(-1,-1,-q^2,q^2)= \sum_{i,j,k\geq 0} \frac{q^{2i^2+j^2+k^2+2ij+2ik+2jk-j-k}}{(q^2;q^2)_i(q^2;q^2)_j(q^2;q^2)_k}.   \label{(2,2,2)-j-k}
	\end{align}
Now using \eqref{(2,2,2)-6} we prove \eqref{03.10-4}.

Setting $(u,v,w,t)=(q^3,-q^3,-q^4,\zeta_4q^5)$, we have
\begin{align}
&F(q^3,-q^3,-q^4,\zeta_4q^5)=\CT \frac{(-q^{10}z^2;q^4)_\infty(q^2z,z^{-1},q^2;q^2)_\infty}{(q^3z,-q^3z,-q^4 z;q^2)_\infty} \nonumber \\
&= \mathrm{CT}\left[  \sum_{i\geq 0}\frac{(-1)^i q^{4i} z^i }{(q^2;q^2)_i} \sum_{j\geq 0}\frac{(q^6z^2)^j }{(q^4;q^4)_j} \sum_{k\geq 0}\frac{  (q^{10}z^2)^k q^{2k^2-2k} }{(q^4;q^4)_k} \sum_{n=-\infty}^\infty (-1)^nq^{n^2-n}z^{-n}   \right] \nonumber \\
&=\sum_{i,j,k\geq 0} \frac{q^{(i+2j+2k)^2+2k^2+3i+4j+6k}}{(q^2;q^2)_i(q^4;q^4)_j(q^4;q^4)_k}.
\end{align}
Now using \eqref{(1,2,2)-1} we prove \eqref{03.10-new}.

Setting $(u,v,w,t)=(q,-q,-q^2,\zeta_4q^3)$, we have
\begin{align}
&F(q,-q,-q^2,\zeta_4q^3)=\CT \frac{(-q^6z^2;q^4)_\infty(q^2z,z^{-1},q^2;q^2)_\infty}{(qz,-qz,-q^2 z;q^2)_\infty}.   \label{1-03.10-8}
\end{align}
This agrees with the right side of \eqref{1-(1,2,4)-2}, and therefore we prove \eqref{03.10-8}.

Setting $(u,v,w,t)=(q,-q^2,-q^3,\zeta_4q^3)$, we have
\begin{align}
&	F(q,-q^2,-q^3,\zeta_4q^3)=\CT\frac{(-q^6z^2;q^4)_\infty(zq^2,z^{-1},q^2;q^2)_\infty}{(qz,-q^2z,- q^3z;q^2)_\infty}  \nonumber \\
&= \mathrm{CT}\left[  \sum_{i\geq 0}\frac{(-1)^i q^{2i} z^i }{(q;q)_i} \sum_{j\geq 0}\frac{(qz)^j }{(q^2;q^2)_j} \sum_{k\geq 0}\frac{  (q^6z^2)^k q^{2k^2-2k} }{(q^4;q^4)_k} \sum_{n=-\infty}^\infty (-1)^nq^{n^2-n}z^{-n}   \right] \nonumber \\
&=  \sum_{i,j,k\geq 0} \frac{(-1)^jq^{i^2+j^2+6k^2+2ij+4ik+4jk+i+2k}}{(q;q)_i(q^2;q^2)_j(q^4;q^4)_k}. \label{1-03.10-9}
\end{align}
Now using \eqref{(1,2,4)-4} we prove \eqref{03.10-9}.

Setting $(u,v,w,t)=(q^3,-q,-q^2,\zeta_4q^3)$, we have
\begin{align}
&	F(q^3,-q,-q^2,\zeta_4q^3)=\CT \frac{(-q^6z^2;q^4)_\infty(zq^2,z^{-1},q^2;q^2)_\infty}{(-qz,-q^2z,q^3z;q^2)_\infty}   \nonumber \\
&= \mathrm{CT}\left[  \sum_{i\geq 0}\frac{(-1)^i q^{i} z^i }{(q;q)_i} \sum_{j\geq 0}\frac{(q^3z)^j }{(q^2;q^2)_j} \sum_{k\geq 0}\frac{  (q^6z^2)^k q^{2k^2-2k} }{(q^4;q^4)_k} \sum_{n=-\infty}^\infty (-1)^nq^{n^2-n}z^{-n}   \right] \nonumber \\
&=  \sum_{i,j,k\geq 0} \frac{(-1)^jq^{i^2+j^2+6k^2+2ij+4ik+4jk+2j+2k}}{(q;q)_i(q^2;q^2)_j(q^4;q^4)_k}. \label{1-03.10-10}
\end{align}
Now using \eqref{(1,2,4)-3} we prove \eqref{03.10-10}.

Setting $(u,v,w,t)=(q,-1,-q,\zeta_4q)$, we have
\begin{align}
&	F(q,-1,-q,\zeta_4q)=\CT \frac{(-q^2z^2;q^4)_\infty(q^2z,z^{-1},q^2;q^2)_\infty}{(qz,-z,-qz;q^2)_\infty}  \nonumber \\
&= \mathrm{CT}\left[  \sum_{i\geq 0}\frac{(-1)^i  z^i }{(q;q)_i} \sum_{j\geq 0}\frac{(qz)^j }{(q^2;q^2)_j} \sum_{k\geq 0}\frac{  (q^2z^2)^k q^{2k^2-2k} }{(q^4;q^4)_k} \sum_{n=-\infty}^\infty (-1)^nq^{n^2-n}z^{-n}   \right] \nonumber \\
&=  \sum_{i,j,k\geq 0} \frac{(-1)^jq^{i^2+j^2+6k^2+2ij+4ik+4jk-i-2k}}{(q;q)_i(q^2;q^2)_j(q^4;q^4)_k}. \label{1-03.10-13}
\end{align}
Now using \eqref{(1,2,4)-1} we prove \eqref{03.10-13}.

Setting $(u,v,w,t)=(-aq,aq,-q^2,\zeta_4aq^2)$, we have
\begin{align}
&F(-aq,aq,-q^2,\zeta_4aq^2)=\CT \frac{(-a^2q^4z^2;q^4)_\infty(q^2z,z^{-1},q^2;q^2)_\infty}{(aqz,-aqz,-q^2 z;q^2)_\infty} \nonumber \\
&= \mathrm{CT}\left[  \sum_{i\geq 0}\frac{(-1)^i q^{2i} z^i }{(q^2;q^2)_i} \sum_{j\geq 0}\frac{(a^2q^2z^2)^j }{(q^4;q^4)_j} \sum_{k\geq 0}\frac{  (a^2q^4z^2)^k q^{2k^2-2k} }{(q^4;q^4)_k} \sum_{n=-\infty}^\infty (-1)^nq^{n^2-n}z^{-n}   \right] \nonumber \\
&=\sum_{i,j,k\geq 0} \frac{q^{(i+2j+2k)^2+2k^2+i}a^{2j+2k}}{(q^2;q^2)_i(q^4;q^4)_j(q^4;q^4)_k}.
\end{align}
Now using \eqref{(1,2,2)-new-1} we prove \eqref{1112-parameter-1}.
\end{proof}

If we set $a=1$, $-q$ or $-q^{-1}$ in \eqref{1112-parameter-1}, we obtain
\begin{align}
&\sum_{i,j,k,l\geq 0}\frac{(-1)^iq^{i^2+j^2+k^2+6l^2+2ij+2ik+2jk+4il+4jl+4kl+k}}{(q^2;q^2)_i(q^2;q^2)_j(q^2;q^2)_k(q^4;q^4)_l}=\frac{(q^8,q^8,q^{16};q^{16})_\infty}{(q^2;q^2)_\infty}, \label{03.10-7} \\
&\sum_{i,j,k,l\geq 0}\frac{(-1)^iq^{i^2+j^2+k^2+6l^2+2ij+2ik+2jk+4il+4jl+4kl+i+j+k+2l}}{(q^2;q^2)_i(q^2;q^2)_j(q^2;q^2)_k(q^4;q^4)_l} \nonumber \\
&\quad \quad =\frac{(q^{12};q^{16})_\infty}{(q^2,q^6,q^6,q^{10},q^{14},q^{14};q^{16})_\infty}, \label{03.10-11} \\
&\sum_{i,j,k,l\geq 0}\frac{(-1)^iq^{i^2+j^2+k^2+6l^2+2ij+2ik+2jk+4il+4jl+4kl-i-j+k-2l}}{(q^2;q^2)_i(q^2;q^2)_j(q^2;q^2)_k(q^4;q^4)_l} \nonumber \\
&\quad \quad =\frac{(q^{4};q^{16})_\infty}{(q^2,q^2,q^6,q^{10},q^{10},q^{14};q^{16})_\infty}. \label{03.10-12}
\end{align}

\section{Concluding remarks}\label{sec-remarks}

\subsection{Identities with non-modular product side}

If we allow the product side of the identity \eqref{id-form} to be non-modular, then we can find many more such identities. For instance, through a search using Maple we find the following identities in addition to Theorem \ref{thm-7}:
\begin{align}
\sum_{i,j,k\geq 0} \frac{(-1)^iq^{i^2+2j^2+k^2+2ij+2ik+2jk+3i}}{(q^2;q^2)_i(q^2;q^2)_j(q^2;q^2)_k}&=\frac{(q^6;q^4)_\infty}{(q;q^4)_\infty(q^7;q^4)_\infty}, \label{(2,2,2)-12} \\
\sum_{i,j,k\geq 0} \frac{q^{i^2+2j^2+k^2+2ij+2ik+2jk+3j+k}}{(q^2;q^2)_i(q^2;q^2)_j(q^2;q^2)_k}&=\frac{1-q^3}{{(1-q^4)(q;q^2)_\infty}}, \label{non-(2,2,2)-7}\\
\frac{(-1)^iq^{i^2+2j^2+k^2+2ij+2ik+2jk+2i-j+k}}{(q^2;q^2)_i(q^2;q^2)_j(q^2;q^2)_k}&=\frac{(1-q^2)(1-q^3)}{{(1-q)(1-q^4)(q^2;q^4)_\infty}}, \label{(2,2,2)-9} \\
\sum_{i,j,k\geq 0} \frac{(-1)^iq^{i^2+2j^2+k^2+2ij+2ik+2jk+3i+3k}}{(q^2;q^2)_i(q^2;q^2)_j(q^2;q^2)_k}&=\frac{1}{(1-q^2)(1-q^8)(q^{10};q^4)_\infty}, \label{(2,2,2)-13} \\
\sum_{i,j,k\geq 0} \frac{(-1)^jq^{i^2+2j^2+k^2+2ij+2ik+2jk+i-j+2k}}{(q^2;q^2)_i(q^2;q^2)_j(q^2;q^2)_k}&=\frac{1-q}{(1-q^2)(1-q^3)(1-q^4)(q^{10};q^4)_\infty}. \label{(2,2,2)-14}
\end{align}
In addition to Theorem \ref{thm-(1,1,3)}, we also find the following identities:
\begin{align}
\sum_{i,j,k\geq 0}\frac{(-1)^kq^{2i^2+2j^2+9k^2+2ij+6ik+6jk+4j+3k}}{(q^2;q^2)_i(q^2;q^2)_j(q^6;q^6)_k}&=\frac{1}{(1-q^6)(q^2,q^{10};q^6)_\infty}, \label{(1,1,3)-1} \\
\sum_{i,j,k\geq 0}\frac{(-1)^kq^{2i^2+2j^2+9k^2+2ij+6ik+6jk+8j+9k}}{(q^2;q^2)_i(q^2;q^2)_j(q^6;q^6)_k}&=\frac{1}{(1-q^{12})(q^2,q^{10};q^6)_\infty}, \label{(1,1,3)-2} \\
\sum_{i,j,k\geq 0}\frac{(-1)^kq^{2i^2+2j^2+9k^2+2ij+6ik+6jk+i+6j+6k}}{(q^2;q^2)_i(q^2;q^2)_j(q^6;q^6)_k}&=\frac{1+q^3}{(1-q^{9})(q^5,q^{7};q^6)_\infty}, \label{(1,1,3)-3} \\
\sum_{i,j,k\geq 0}\frac{(-1)^kq^{2i^2+2j^2+9k^2+2ij+6ik+6jk+2i+8j+9k}}{(q^2;q^2)_i(q^2;q^2)_j(q^6;q^6)_k}&=\frac{1}{(1-q^{6})(1-q^{12})(q^4,q^{14};q^6)_\infty}, \label{(1,1,3)-4} \\
 \sum_{i,j,k\geq 0}\frac{(-1)^kq^{2i^2+2j^2+9k^2+2ij+6ik+6jk+4i+3k}}{(q^2;q^2)_i(q^2;q^2)_j(q^6;q^6)_k}&=\frac{1}{(1-q^{6})(q^2,q^{10};q^6)_\infty}, \label{(1,1,3)-5} \\
\sum_{i,j,k\geq 0}\frac{(-1)^kq^{2i^2+2j^2+9k^2+2ij+6ik+6jk+2i-j}}{(q^2;q^2)_i(q^2;q^2)_j(q^6;q^6)_k}&=\frac{1-q^2}{(1-q^{3})(q,q^{5};q^6)_\infty}.\label{(1,1,3)-7}
\end{align}
We also find the following identities in addition to Theorem \ref{thm-6}:
\begin{align}
\sum_{i,j,k\geq 0}\frac{(-1)^jq^{i^2+6j^2+4k^2+4ij+4ik+8jk+i-3j+3k}}{(q^2;q^2)_i(q^4;q^4)_j(q^4;q^4)_k}&= \frac{(1-q^3)(1-q^{10})}{(1-q^5)(1-q^8)(q^2;q^{4})_\infty}, \label{(1,2,2)-new-7}\\
 \sum_{i,j,k\geq 0}\frac{(-1)^jq^{i^2+6j^2+4k^2+4ij+4ik+8jk+i+6k}}{(q^2;q^2)_i(q^4;q^4)_j(q^4;q^4)_k}&= \frac{1-q^6}{(1-q^8)(q^2;q^{4})_\infty}, \label{(1,2,2)-new-9} \\
 \sum_{i,j,k\geq 0}\frac{(-1)^{i+j}q^{i^2+6j^2+4k^2+4ij+4ik+8jk+2i-2j-3k}}{(q^2;q^2)_i(q^4;q^4)_j(q^4;q^4)_k}&= {(1+q)(q^3;q^4)_\infty}, \label{(1,2,2)-new-20}\\
\sum_{i,j,k\geq 0}\frac{(-1)^{i+j}q^{i^2+6j^2+4k^2+4ij+4ik+8jk+3i-4k}}{(q^2;q^2)_i(q^4;q^4)_j(q^4;q^4)_k}&= 2(q^6;q^4)_\infty, \label{(1,2,2)-new-23}\\
 \sum_{i,j,k\geq 0}\frac{(-1)^{i+j}q^{i^2+6j^2+4k^2+4ij+4ik+8jk+4i+2j-3k}}{(q^2;q^2)_i(q^4;q^4)_j(q^4;q^4)_k}&= {(1+q)(q^7;q^4)_\infty}, \label{(1,2,2)-new-24}\\
\sum_{i,j,k\geq 0}\frac{(-1)^{i+j}q^{i^2+6j^2+4k^2+4ij+4ik+8jk+4i+2j-k}}{(q^2;q^2)_i(q^4;q^4)_j(q^4;q^4)_k}&= {(1+q^3)(q^5;q^4)_\infty}, \label{(1,2,2)-new-25}\\
 \sum_{i,j,k\geq 0}\frac{(-1)^{i+j}q^{i^2+6j^2+4k^2+4ij+4ik+8jk+5i+4j-2k}}{(q^2;q^2)_i(q^4;q^4)_j(q^4;q^4)_k}&= \frac{{(q^4;q^4)_\infty}}{1-q^2}, \label{(1,2,2)-new-28}\\
\sum_{i,j,k\geq 0}\frac{(-1)^{i+j}q^{i^2+6j^2+4k^2+4ij+4ik+8jk+5i+4j}}{(q^2;q^2)_i(q^4;q^4)_j(q^4;q^4)_k}&= {(1+q^4)(q^6;q^4)_\infty}, \label{(1,2,2)-new-29}\\
 \sum_{i,j,k\geq 0}\frac{(-1)^{i+j}q^{i^2+6j^2+4k^2+4ij+4ik+8jk+6i+6j-k}}{(q^2;q^2)_i(q^4;q^4)_j(q^4;q^4)_k}&= {(1+q^3)(q^9;q^4)_\infty}, \label{(1,2,2)-new-32}\\
\sum_{i,j,k\geq 0}\frac{(-1)^{i+j}q^{i^2+6j^2+4k^2+4ij+4ik+8jk+6i+6j+k}}{(q^2;q^2)_i(q^4;q^4)_j(q^4;q^4)_k}&= {(1+q^5)(q^7;q^4)_\infty}, \label{(1,2,2)-new-33}\\
\sum_{i,j,k\geq 0}\frac{(-1)^{i+j}q^{i^2+6j^2+4k^2+4ij+4ik+8jk+7i+8j}}{(q^2;q^2)_i(q^4;q^4)_j(q^4;q^4)_k}&= {(1+q^4)(q^{10};q^4)_\infty}, \label{(1,2,2)-new-35}\\
\sum_{i,j,k\geq 0}\frac{(-1)^{i+j}q^{i^2+6j^2+4k^2+4ij+4ik+8jk+7i+8j+2k}}{(q^2;q^2)_i(q^4;q^4)_j(q^4;q^4)_k}&= {(1+q^6)(q^{8};q^4)_\infty}. \label{(1,2,2)-new-36}
\end{align}

In addition to Theorem \ref{thm-9}, we also find the following identity:
\begin{align}
\sum_{i,j,k\geq 0}\frac{(-1)^iq^{i^2+2j^2+k^2+2ij-2ik-2jk+i+3j}}{(q;q)_i(q^2;q^2)_j(q^2;q^2)_k}&={(q^2;q)_\infty}{(q^2;q^4)_\infty}. \label{(1,2,2)--9} 
\end{align}
The above identities can be proved in a way similar to the theorems we mentioned. To save space, here we only give the proof for \eqref{(2,2,2)-12} and \eqref{(1,2,2)--9}. Other identities can be proved similarly.
\begin{proof}[Proof of \eqref{(2,2,2)-12}]
Recall the function $F(u,v,w)$ defined in \eqref{111-F-defn}. Setting $(u,v,w)=(-q^3,1,1)$,  by \eqref{111-F-start} we have
\begin{align*}
&F(-q^3,1,1)={(q^4;q^2)_\infty}\sum_{n=0}^\infty\frac{q^{n^2}(-q;q^2)_n}{(q^2;q^2)_n(q^4;q^2)_n} \nonumber \\
&={(-q;q^2)_\infty}\sum_{n=0}^\infty\frac{(-1)^nq^{n^2+3n}(q^{-2};q^2)_n}{(-q;q^2)_n(q^2;q^2)_n}  \quad  \text{(by \eqref{(Laughlin-(6.1.3))})}       \nonumber \\
&={(-q;q^2)_\infty}\Big(1+\frac{(1-q^{-2})(-q^4)}{(1+q)(1-q^2)}\Big)=\frac{(q^6;q^4)_\infty}{(q;q^4)_\infty(q^7;q^4)_\infty}. \nonumber
\end{align*}
This proves \eqref{(2,2,2)-12}.
\end{proof}

\begin{proof}[Proof of \eqref{(1,2,2)--9}]
Recall the function $F(u,v,w)$ defined in \eqref{F-defn-122-cor}.
Setting $(u,v,w)=(q^2,q^5,-q^{-1})$, by \eqref{F-const} we have
\begin{align*}
&F(q^2,q^5,-q^{-1})=\mathrm{CT}\left[ \frac{(q^5z,z^{-1},q^2;q^2)_\infty}{(q^3z,-q^{-1}z^{-1};q^2)_\infty} \right] \nonumber \\
&=(q^2;q^2)_\infty\mathrm{CT}\left[      \sum_{i\geq 0}{q^{3i} z^i } \sum_{n=0}^\infty \frac{(-q;q^2)_n(-q^{-1}z^{-1})^n}{(q^2;q^2)_n}    \right] \nonumber \\
&=(q^2;q^2)_\infty \sum_{n=0}^\infty\frac{(-q;q^2)_n(-q^2)^n}{(q^2;q^2)_n}=(q^2;q)_\infty(q^2;q^4)_\infty. 
\end{align*}
This proves \eqref{(1,2,2)--9}.
\end{proof}

\subsection{Alternative proofs}
As we have already seen, some of the identities can be  proved via different methods. Here we introduce one more method for proving the identities in this paper. This is the so-called integral method, which has been used in recent works such as \cite{Rosengren,Cao-Wang,Laughlin,Wei,Wang-rank2,Wang-rank3}. Let us first explain how the integral method works. When the function $f(z)=\sum_{n=-\infty}^\infty a(n)z^n$ can be integrated along some positively oriented simple closed curve $K$ around the origin, then we have
\begin{align}\label{int}
\oint_K f(z)\frac{dz}{2\pi iz}=a(0).
\end{align}
We will use an identity from the book of Gasper and Rahman \cite{GR-book}.
\begin{lemma}\label{lem-integral}
(Cf.\ \cite[Eq.\ (4.10.6)]{GR-book})
Suppose that
$$P(z):=\frac{(a_1z,\dots,a_Az,b_1/z,\dots,b_B/z;q)_\infty}{(c_1z,\dots,c_Cz,d_1/z,\dots,d_D/z;q)_\infty}$$
has only simple poles. We have
\begin{align}\label{eq-integral}
\oint P(z)\frac{dz}{2\pi iz}=& \frac{(b_1c_1,\dots,b_Bc_1,a_1/c_1,\dots,a_A/c_1;q)_\infty }{(q,d_1c_1,\dots,d_Dc_1,c_2/c_1,\dots,c_C/c_1;q)_\infty} \nonumber \\
& \times \sum_{n=0}^\infty \frac{(d_1c_1,\dots,d_Dc_1,qc_1/a_1,\dots,qc_1/a_A;q)_n}{(q,b_1c_1,\dots,b_Bc_1,qc_1/c_2,\dots,qc_1/c_C;q)_n} \nonumber \\
&\times \Big(-c_1q^{(n+1)/2}\Big)^{n(C-A)}\Big(\frac{a_1\cdots a_A}{c_1\cdots c_C} \Big)^n +\text{idem} ~(c_1;c_2,\dots,c_C)
\end{align}
when $C>A$, or if $C=A$ and
\begin{align}\label{cond}
\left|\frac{a_1\cdots a_A }{c_1\cdots c_C}\right|<1.
\end{align}
Here the integration is over a positively oriented contour so that
\begin{itemize}
    \item the poles of $(c_1z,\dots,c_Cz;q)_\infty^{-1}$ lie outside the contour;
    \item and the origin and poles of $(d_1/z,\dots,d_D/z;q)_\infty^{-1}$ lie inside the contour.
\end{itemize}
The symbol ``idem $(c_1;c_2,\dots,c_C)$'' after an expression stands for the sum of the $(C-1)$ expressions obtained from the preceding expression by interchanging $c_1$ with each $c_k$, $k=2,3,\dots,C$.
\end{lemma}

Here we use two identities in Theorem \ref{thm-8} to illustrate this method.
\begin{proof}[Second proof of \eqref{(1,2,4)-16} and \eqref{(1,2,4)-18}]
Recall the function $F(u,v,w)$ defined in \eqref{Thm124-F-start}.
Setting $(u,v,w)=(q^{-1},-q^{-1},q^2)$,  by \eqref{F-start} we have
\begin{align}
&F(q^{-1},-q^{-1},q^2)=\oint \frac{(qz,-q z,q^2 z,z^{-1},q^2;q^2)_\infty}{(q^{-1}z,z,-q^{-1} z;q^2)_\infty} \frac{\mathrm{d}z} {2\pi iz} \nonumber \\
&=(q^2;q^2)_\infty\Big(F_1(q)+F_2(q)+F_3(q)\Big), \label{1-(1,2,4)-16}
\end{align}
where (by Lemma \ref{lem-integral})
 \begin{align}
 F_1(q)=-\frac{1}{2}q^{-1}(q;q^2)_\infty, \quad F_2(q)=0, \quad F_3(q)=\frac{1}{2}q^{-1}(-q;q^2)_\infty. \label{4-(1,2,4)-16}
 \end{align}
 Substituting  \eqref{4-(1,2,4)-16} into \eqref{1-(1,2,4)-16}, we prove \eqref{(1,2,4)-16}.

 In the same way we can prove \eqref{(1,2,4)-18}.
\end{proof}

\subsection{Byproducts and a conjectural identity}
The common feature of our proofs is to reduce the desired multi-sum identities to some known single-sum identities. Interestingly, we may also obtain some new single-sum identities from the results we proved. For instance, we deduce the following result from \eqref{(1,2,2)--3}.
\begin{corollary}\label{cor-mod12}
We have
\begin{align}\label{eq-cor-mod12}
\sum_{n=0}^\infty \frac{(-q;q^2)_n q^{2n}}{(q;q^2)_{n+1}}=\frac{(q^6,q^{10},q^{16};q^{16})_\infty (q^4,q^{28};q^{32})_\infty}{(q;q)_\infty}.
\end{align}
\end{corollary}
\begin{proof}
Recall the function $F(u,v,w)$ defined in \eqref{F-defn-122-cor}. Applying \eqref{F-const} and Lemma \ref{lem-integral} we have
\begin{align}
&F(-q,q^3,-1)=\oint \left[ \frac{(q^3 z,q^2 z,z^{-1},q^2;q^2)_\infty}{(-qz,-q^2z,-z^{-1};q^2)_\infty} \right] \frac{dz}{2\pi iz}
=(q^2;q^2)_\infty\Big(F_1(q)+F_2(q)\Big), \label{1-(1,2,2)--3}
\end{align}
where
\begin{align}
&F_1(q)=\frac{(-q,-q,-q^2;q^2)_\infty}{(q,q,q^2;q^2)_\infty}\sum_{n=0}^\infty\frac{(-1;q^2)_n q^{2n}}{(q^2;q^2)_n} \nonumber \\
&= \frac{(-q,-q,-q^2,-q^2;q^2)_\infty}{(q,q,q^2,q^2;q^2)_\infty} \quad (\text{by \eqref{eq-qbinomial}}) \label{2-(1,2,2)--3}, \\
&F_2(q)=\frac{(-1,-q,-q^2;q^2)_\infty}{(q^{-1},q^2,q^2;q^2)_\infty}\sum_{n=0}^\infty\frac{(-q;q^2)_n q^{2n}}{(q^3;q^2)_n}.\label{3-(1,2,2)--3}
\end{align}
Now substituting \eqref{(1,2,2)--3} and \eqref{2-(1,2,2)--3} into \eqref{1-(1,2,2)--3}, we get an expression for $F_2(q)$ in terms of infinite products. This expression can be shown to be equivalent to the desired representation via the method in \cite{Frye-Garvan}.
\end{proof}
The right side of \eqref{eq-cor-mod12} is the same with the following identity:
\begin{align}\label{eq-MSZ-mod12}
\sum_{n=0}^\infty \frac{q^{n(n+1)}(-q;q^2)_n}{(q;q)_{2n+1}}=\frac{(q^6,q^{10},q^{16};q^{16})_\infty (q^4,q^{28};q^{32})_\infty}{(q;q)_\infty}.
\end{align}
This appears as \cite[Eq.\ (2.32.2)]{MSZ2008}. If one can show directly that the sum sides of \eqref{eq-cor-mod12} and \eqref{eq-MSZ-mod12} are equivalent, then it provides a new proof for \eqref{(1,2,2)--3}.

\begin{corollary}
We have
\begin{align}
&\sum_{n=0}^\infty \frac{(-1;q^4)_n(-\zeta_4 q)^n}{(-1,q^2;q^2)_n}=\frac{(q^6,q^{10},q^{16};q^{16})_\infty}{(-\zeta_4q;q^2)_\infty (q^4;q^4)_\infty}, \label{cor-id-1} \\
&\sum_{n=0}^\infty \frac{(\zeta_4q;q)_{2n}(\zeta_4q^2)^n}{(q;q)_{2n+1}}=\frac{(q^2,q^{14},q^{16};q^{16})_\infty}{(q;q)_\infty (\zeta_4q;q^2)_\infty}, \label{cor-id-2} \\
&\sum_{n=0}^\infty \frac{(\zeta_4;q)_{2n}(\zeta_4q^2)^n}{(q;q)_{2n}}=\frac{(-q^4;q^8)_\infty}{(q,\zeta_4q^2;q^2)_\infty (-q,q^3,q^4,q^5,-q^7;q^8)_\infty}, \label{cor-id-3} \\
&\sum_{n=0}^\infty \frac{(-\zeta_4;q)_{2n}(-\zeta_4q^2)^n}{(q;q)_{2n}}=\frac{(-q^4;q^8)_\infty}{(q,-\zeta_4q^2;q^2)_\infty (-q,q^3,q^4,q^5,-q^7;q^8)_\infty}.  \label{cor-id-4}
\end{align}
\end{corollary}
\begin{proof}
From \cite[Eq.\ (3.2)]{Rosengren} we know that when $\alpha_1\alpha_2=\beta_1\beta_2\beta_3$,
\begin{align}\label{R32}
&\oint \frac{(\alpha_1z,\alpha_2z,qz,1/z;q)_\infty}{(\beta_1z,\beta_2z,\beta_3z;q)_\infty}\frac{\diff z}{2\pi iz}=\frac{(\beta_1,\alpha_1/\beta_1;q)_\infty}{(q;q)_\infty} \nonumber \\
&=\frac{(\beta_1,\alpha_1/\beta_1;q)_\infty}{(q;q)_\infty}\sum_{n=0}^\infty \frac{(\alpha_2/\beta_2,\alpha_2/\beta_3;q)_n}{(\beta_1,q;q)_n}\left(\frac{\alpha_1}{\beta_1}\right)^n
\end{align}
Recall the function $F(u,v,w)$ defined in \eqref{Thm124-F-start}. From \eqref{1-(1,2,4)-1} and \eqref{R32} we have
\begin{align}
&F(-1,q,-q^2)=\oint \frac{(\zeta_4 q z,-\zeta_4 q z,q^2z,z^{-1},q^2;q^2)_\infty}{(-z,qz,-qz;q^2)_\infty} \frac{\mathrm{d}z} {2\pi iz} \nonumber  \\
&=(-1,-\zeta_4 q;q^2)_\infty \sum_{n=0}^\infty\frac{(\zeta_4 ,-\zeta_4 ;q^2)_n (-\zeta_4 q)^n}{(-1,q^2;q^2)_n}. \label{cor-id-1-proof}
\end{align}
By \eqref{cor-id-1-proof} and \eqref{(1,2,4)-1} we obtain \eqref{cor-id-1}.

Using \eqref{F-start} and \eqref{R32} we have
\begin{align}
&F(-q^3,q^3,-q^{10})=\oint \frac{(\zeta_4 q^5 z,-\zeta_4 q^5 z,z^{-1},q^2z,q^2;q^2)_\infty}{(q^3 z,-q^3 z,-q^4 z;q^2)_\infty} \frac{\mathrm{d}z} {2\pi iz} \nonumber \\
&=(\zeta_4 q^2,q^3;q^2)_\infty \sum_{n=0}^\infty\frac{(\zeta_4 q,\zeta_4 q^2;q^2)_n (\zeta_4 q^2)^n}{(q^2,q^3;q^2)_n}.
\end{align}
Using \eqref{(1,2,4)-5} we obtain \eqref{cor-id-2}.

Using \eqref{F-start} and \eqref{R32} we have
\begin{align}
&F(-q^2,q,-q^6)=\oint \frac{(\zeta_4 q^3 z,-\zeta_4 q^3 z,z^{-1},q^2z,q^2;q^2)_\infty}{(q z,-q^2 z,-q^3 z;q^2)_\infty} \frac{\mathrm{d}z} {2\pi iz} \nonumber \\
&=(q,\zeta_4 q^2;q^2)_\infty \sum_{n=0}^\infty\frac{(\zeta_4 ,\zeta_4 q;q^2)_n (\zeta_4 q^2)^n}{(q,q^2;q^2)_n}.
\end{align}
Using \eqref{(1,2,4)-4} we obtain \eqref{cor-id-3}.

Similarly, using \eqref{F-start} and \eqref{R32} we have
\begin{align}
&F(-q,q^3,-q^6)=\oint \frac{(\zeta_4 q^3 z,-\zeta_4 q^3 z,z^{-1},q^2z,q^2;q^2)_\infty}{(-qz,-q^2 z,q^3 z;q^2)_\infty} \frac{\mathrm{d}z} {2\pi iz} \nonumber \\
 &=(-q,-\zeta_4 q^2;q^2)_\infty \sum_{n=0}^\infty\frac{(-\zeta_4 ,\zeta_4 q;q^2)_n (-\zeta_4 q^2)^n}{(-q,q^2;q^2)_n}.
\end{align}
Using \eqref{(1,2,4)-3} we obtain \eqref{cor-id-4} after replacing $q$ by $-q$.
\end{proof}

We end this paper by proposing a conjecture. When searching for Rogers--Ramanujan type identities with index $(1,3)$, we find the following interesting identity, which gives a new companion to \eqref{K-conj-1}.
\begin{conj}\label{conj-13}
We have
\begin{align}
\sum_{i,j\geq 0}\frac{q^{i^2+3j^2-3ij+j}}{(q;q)_i(q^3;q^3)_j}&=\frac{(q^6;q^9)_\infty}{(q,q^2,q^2,q^4,q^5,q^5;q^6)_\infty}.
\end{align}
\end{conj}

\subsection*{Acknowledgements}
This work was supported by the National Natural Science Foundation of China (12171375). We thank S.O. Warnaar for providing the proofs described in Remarks \ref{rem-Warnaar-pq} and \ref{rem-Warnaar} and some helpful comments.

\end{document}